\documentclass[10pt,oneside,shortlabels]{amsart}
\usepackage[T1]{fontenc}
\usepackage[utf8]{inputenc}
\usepackage{amsthm,amsmath}
\usepackage{amsbsy}
\usepackage{amstext}
\usepackage{amssymb}
\usepackage{esint}

\makeatletter
\numberwithin{equation}{section}
\numberwithin{figure}{section}

\usepackage{mathtools}

\newtheorem {theo} {Theorem} [section]
\newtheorem {prop} [theo] {Proposition}
\newtheorem {cory} [theo] {Corollary}
\newtheorem {lem} [theo] {Lemma}
\newtheorem {defi} [theo] {Definition}

\newtheorem*{thmA}{Theorem~A}
\newtheorem*{thmB}{Theorem~B}
\newtheorem*{thmC}{Theorem~C}
\theoremstyle{definition}
\newtheorem{example}{Example}
\newtheorem {obs} [theo] {Remark}

\def\sideremark#1{\ifvmode\leavevmode\fi\vadjust{\vbox to0pt{\vss 
    \hbox to 0pt{\hskip\hsize\hskip1em           
 \vbox{\hsize2cm\tiny\raggedright\pretolerance10000
 \noindent #1\hfill}\hss}\vbox to8pt{\vfil}\vss}}}%

\subjclass[2010]{}

\usepackage[all]{xy}        
\newcommand{\xyL}[1]{%
	\xydef@\xymatrixrowsep@{#1}
} 
\newcommand{\xyC}[1]{%
	\xydef@\xymatrixcolsep@{#1}
} 
\makeatother

\begin{document}

\title[Length of $\varepsilon$-neighborhoods of orbits of Dulac maps]{Length of $\varepsilon$-neighborhoods of orbits of Dulac maps}

\author{P. Marde\v si\'c$^{1}$, M. Resman$^{2}$, J.-P. Rolin$^{3}$,
V. \v Zupanovi\'c$^{4}$}
\begin{abstract} We consider a class of parabolic Dulac germs of hyperbolic polycycles. In view of formal or analytic characterization of such a germ $f$ by fractal properties of several of its orbits, we study the length $A_f(x_0,\varepsilon)$ of $\varepsilon$-neighborhoods of orbits of $f$ with initial points $x_0$. We show that, even if $f$ is an analytic germ, $\varepsilon\mapsto A_f(x_0,\varepsilon)$ does not have a full asymptotic expansion in $\varepsilon$ in a scale of powers and (iterated) logarithms. This result is already stated in \cite{nonlin} for complex analytic germs. This partial asymptotic expansion cannot contain necessary information for analytic classification.

Hence, we introduce a new notion: the \emph{continuous time length of\linebreak the $\varepsilon$-neighborhood} $A^c_f(x_0,\varepsilon)$. We show that this function has a full asymptotic expansion in $\varepsilon$ in the power, iterated logarithm scale.
Moreover, its asymptotic expansion extends the initial, existent part of the asymptotic expansion of the classical length $\varepsilon\mapsto A_f(x_0,\varepsilon)$. 

Finally, we prove that this initial part of the asymptotic expansion determines the class of formal conjugacy of the Dulac map $f$.

\end{abstract}
\maketitle
\noindent \emph{Acknowledgement}. This research was supported by:
Croatian Science Foundation (HRZZ) project no. 2285, Croatian Unity Through Knowledge Fund (UKF) \emph{My first collaboration grant} project no. 7, French ANR project
STAAVF. Part of the research was made during the $6$-month stay of $^2$ at University of Burgundy in 2018, financed by the UKF project.

\section{\label{sec:introduction}Introduction}

We are interested here in the informal question: \emph{can we ``see'' a diffeomorphism by observing some of its orbits?} This work is a continuation of the study of one-dimensional
discrete dynamical systems, based on fractal properties of their
orbits \cite{mrz,resman,nonlin,belgproc}.   Recall
that fractal properties of a bounded subset $U$ of $\mathbb{R}$
or $\mathbb{C}$ reflect the asymptotic behavior at $0$ of the function
$\varepsilon\mapsto A\left(U_{\varepsilon}\right)$, for $\varepsilon>0$,
where $A\left(U_{\varepsilon}\right)$ denotes the Lebesgue measure
of the $\varepsilon$-neighborhood of $U$. In particular, the \emph{box
dimension} of $U$ (see \cite{falconer} for a precise definition) gives
the growth of $A\left(U_{\varepsilon}\right)$, when $\varepsilon$
tends to $0$.\\

Consider a germ $f$ in one variable with a fixed
point $a$, and a point $x_{0}$ close to $a$. We denote by $A_{f}\left(x_{0},\varepsilon\right)$
the Lebesgue measure of the $\varepsilon$-neighborhood of the orbit
of $x_{0}$ (or a \emph{directed} version of it in the complex case,
see \cite{resman}). It is proved in \cite{elezovic_zupanovic_zubrinic}
that, for a differentiable germ with an attracting fixed point $a$, the multiplicity of $a$ is determined
by the box dimension of \emph{any} attracted orbit. This result has
been generalized to a class of non--differentiable germs in \cite{mrz},
where the asymptotic behavior of $A_{f}(x_{0},\varepsilon)$ is given explicitely as that of $g^{-1}(\varepsilon)$, $g=\textrm{id}-f$, and related to the multiplicity of $f$ in a given Chebyshev scale.

In the same spirit, it is proved in \cite{resman} that the class
of \emph{formal} conjugacy of a (real or complex) parabolic analytic
germ is determined by an initial part of the asymptotic expansion
of $A_{f}\left(x_{0},\varepsilon\right)$ in power-log monomials.
Describing the class of \emph{analytic} conjugacy of $f$ would
require not only an initial part, but a full asymptotic expansion
of $A_{f}\left(x_{0},\varepsilon\right)$ in power-log monomials. Unfortunately, it is
proved in \cite{resman} that such a complete expansion does not exist.
The reason is that the computation of $A_{f}\left(x_{0},\varepsilon\right)$
needs the determination of an appropriate critical iterate $f^{n_{\varepsilon}}$
of $f$, where $n_{\varepsilon}$ is obtained using the
integer part function. One goal of the present paper is to correct
this flaw by proposing a convenient modification of the definition
of $A_{f}\left(x_{0},\varepsilon\right)$. This modification involves, if they exist,
the \emph{fractional iterates} $f^{t}$ of $f$, for $t\in\mathbb{R}$,
instead of the usual integer iterates of $f$. They are related to an embedding of a germ in a flow. We call the new function thus obtained the \emph{continuous
time length of $\varepsilon$-neighborhoods of orbits of $f$}, and
denote it by $A_{f}^{c}(x_{0},\varepsilon)$. For the germs we consider (analytic germs or Dulac germs), we show that it generalizes the classical, \emph{discrete time}
length $A_{f}\left(x_{0},\varepsilon\right)$ of the $\varepsilon$-neighborhood
of the orbits of $f$ in the following sense: for any orbit, the beginning
of the asymptotic expansion of $A_{f}^c(x_{0},\varepsilon)$ coincides with the existing beginning of the asymptotic
expansion of $A_{f}(x_{0},\varepsilon)$.\\

In this paper we work with \emph{Dulac germs}. By a Dulac germ, we mean an analytic germ on some open interval $(0,d)$, admitting a non-trivial power-log Dulac asymptotic expansion at $0$ which can moreover be extended to an analytic and bounded germ on a bigger complex domain containing $(0,d)$ (Definition~\ref{def:dulac_germ}), see e.g. \cite{ilyalim} for precise definition. It was shown in \cite{Dulac,ilyalim} that the Poincar\' e germs of hyperbolic polycycles are Dulac germs.

We restrict moreover to parabolic Dulac maps. It is the most interesting case in the study of the \emph{cyclicity} of a polycyle. The cyclicity is the maximal number of limit cycles which can appear in a neighborhood of the polycycle in an analytic deformation of the original vector field. For this purpose, we have to study zeros of the \emph{displacement map} $\Delta(x)$, that is the Dulac map minus the identity. In the case of a non-parabolic Dulac map, the displacement map is of the form $\Delta(x)=ax+o(x)$, with $a\ne0$ \cite{cherkas}. In the parabolic case the displacement function is flatter, of order higher then linear, so one should expect higher cyclicity. This is indeed a theorem in the case of \emph{hyperbolic loops} (\emph{i.e.} polycycles with one vertex): there exist parabolic Dulac maps corresponding to homoclinic loops of arbitrary high cyclicity \cite{roussarie_number}, whereas in the non-parabolic case the cyclicity is one \cite{cherkas}. In \cite{mourtada}, universal bounds for cyclicity of hyperbolic polycyles of $2$, $3$ and $4$ vertices are given under some generic conditions. They imply that the corresponding Dulac map is non--parabolic. The cyclicity problem, even for polycycles with such low number of vertices in the parabolic case, is widely open.\\

The results of this paper rely on \cite{mrrz2} and \cite{MRRZ2Fatou}. 
In \cite{mrrz2}, we introduced an algebra $\mathcal L$ of formal \emph{power-log transseries} which contains the Dulac asymptotic
expansions of Dulac germs. Here, this algebra is denoted by $\widehat{\mathcal L}$, to distinguish from algebras of germs. In particular, we have proved in \cite{mrrz2} that \emph{parabolic} (tangent to the identity) elements of $\widehat{\mathcal{L}}$ can be embedded into (unique) formal flows, and hence admit formal fractional
iterates. It is standard that the existence of a Fatou coordinate for a germ $f$, which conjugates $f$ to the shift $w\mapsto w+1$, is equivalent to the existence of an embedding. In \cite{MRRZ2Fatou}, we have explicitely constructed a Fatou coordinate $\Psi $ of a Dulac germ $f$ and the formal Fatou coordinate $\widehat\Psi$ of its Dulac expansion $\widehat f$. We have proved that there exists a unique Fatou coordinate of Dulac germ $f$ which admits asymptotic expansion in the class $\widehat{\mathfrak L}$ of transseries in power-iterated log monomials of finite depth in iterated logarithms. Moreover, we have proved that this expansion is equal, up to a constant, to the unique formal Fatou coordinate $\widehat\Psi$ of $\widehat f$.

Given a parabolic element $\widehat{f}\in\widehat{\mathcal{L}}$, 
we define in Section~\ref{sec:continuous_time_length} a \emph{formal continuous time length of $\varepsilon$-neighborhoods
of orbits}, denoted by $\widehat{A}_{\widehat{f}}^{c}(\varepsilon)$. Similarly, given a germ $f$ which embeds in a flow and an initial point $x_0$, we define the \emph{continuous time length of $\varepsilon$-neighborhood
of orbit of $x_0$}, denoted by $A_f^{c}(\varepsilon,x_0)$. Both notions rely on embeddings. Here we prove three main results:
\smallskip

\textbf{1. Theorem A.} In Theorem~A, we show that, for a parabolic
transseries $\widehat{f}\in\widehat{\mathcal{L}},$ the \emph{formal} continuous
time length $\widehat{A}_{\widehat{f}}^{c}(\varepsilon)$ is a transseries
with a well-ordered support made of monomials of the type: 
\[
x^{\alpha}\left(\frac{1}{-\log x}\right)^{\gamma_{1}}\Big(\frac{1}{\log(-\log x)}\Big)^{\gamma_{2}},\ \alpha\geq0,\ \gamma_{1},\,\gamma_{2}\in\mathbb{R}.
\]
In particular, it belongs to $\widehat{\mathfrak{L}}$.
\smallskip

\textbf{2. Theorem B.} Embedding of a parabolic Dulac germ $f$ 
in the flow $\{f^{t}\}$ corresponding to the Fatou coordinate constructed in the Theorem in \cite{MRRZ2Fatou} allows to define the function of the continuous time length of $\varepsilon$-neighborhood of orbit of $x_0$, 
$\varepsilon\mapsto A_{f}^{c}(x_{0},\varepsilon)$. In Theorem~B, we show that the function $\varepsilon\mapsto A_{f}^{c}(x_{0},\varepsilon)$ admits a particular type of (trans)asymptotic expansion   in $\widehat{\mathfrak L}$, which is equal, up to a term $C\varepsilon$, $C\in\mathbb R$,  to the formal
length $\widehat{A}_{\widehat{f}}^{c}(\varepsilon)$ for the Dulac expansion $\widehat f$. Moreover, the expansions of $A_{f}(x_{0},\varepsilon)$ and $A_{f}^{c}(x_{0},\varepsilon)$ are equal up to the order $O(\varepsilon)$. Thus, the expansion of $A_{f}^{c}(x_{0},\varepsilon)$ continues the power-log expansion of $A_{f}(x_{0},\varepsilon)$ after it ceases to exist. In this sense, this paper extends the result from \cite{mrz} for Dulac maps, which gives only the leading term in the asymptotic expansion of $A_{f}(x_{0},\varepsilon)$ to be equal to the leading term of $g^{-1}(\varepsilon)$, $\varepsilon\to 0$, where $g=\mathrm{id}-f$.
\smallskip

\textbf{3. Theorem C. }
In Theorem C we show that for a parabolic Dulac germ $f$ the initial (existent) part of the (trans)asymptotic expansion of $A_f(x_0,\varepsilon)$ determines the class of \emph{formal} conjugacy of $f$. This generalizes the result from \cite{resman} for parabolic analytic germs.
\medskip

Having ensured the existence of a complete (trans)asymptotic expansion
for the function $\varepsilon\mapsto A_{f}^{c}(x_{0},\varepsilon)$, we plan in the future to define Ecalle-Voronin-like moduli for analytic classification of parabolic Dulac germs and show that they can be read by comparison of sectorial functions $A_f^c(x_0,\varepsilon)$. Note that such new classification moduli would be defined in the $\varepsilon$-plane and not in the phase $x$-plane.
\\

\paragraph{\textbf{Organization of the paper}}

In Section \ref{sec:continuous_time_length},
we introduce the notion of the continuous time length $A_{f}^{c}\left(x_{0},\varepsilon\right)$
of $\varepsilon$-neighborhoods of orbits for a germ, and of its formal
analogue for power-log transseries. We also recall the notion of \emph{Fatou
coordinate} for a germ $f$, which essentially conjugates $f$ to
the translation by $1$. It allows to give an explicit formula for
the continuous time length $A_{f}^{c}\left(x_{0},\varepsilon\right)$ for a germ $f$ which
embeds as the time-one map of a flow $\left\{f^{t}\right\} $. 
Note that having an equality is much more convenient than having to work with inequalities, as is the case for computing $A_f(x_0,\varepsilon)$ in e.g. \cite{mrz} or \cite{resman}. The price to pay is searching for the Fatou coordinate, which is for Dulac maps done in \cite{MRRZ2Fatou}. We
also state Theorems A, B and C. 

In Section~\ref{sec:asymptotic_expansions} we recall the notion of \emph{sectional asymptotic expansions}. The term was introduced in \cite{MRRZ2Fatou}, and was necessary to uniquely define the asymptotic expansion of a germ in the class of transseries $\widehat{\mathfrak L}$. The expansion was dependent on the choice of the so-called \emph{section function}, which assigns a unique germ to an asymptotic (trans)series. In short, the germs chosen are dictated by the solution of the Abel equation. In \cite{MRRZ2Fatou}, a special type of \emph{integral section function} was introduced to be adapted to such equation. Here, the integral section function is generalized to be adapted to the continuous length of the $\varepsilon$-neighborhoods of orbits as well.

In Section~\ref{sec:examples} we compute formal continuous lengths of $\varepsilon$-neighborhoods of orbits in some simple cases: for a simple Dulac germ and for regular germs. For regular germs, as already mentioned, this formal length is a continuation of the power-log asymptotic expansion of $A_f(x_0,\varepsilon)$ which by \cite{nonlin} does not exist after $O(\varepsilon)$, as $\varepsilon\to 0$. We also demonstrate that in the Dulac case the formal length is indeed transfinite.

In Section~\ref{sec:inverse} we describe the formal inverse of a Dulac series and show that it is equal to the (trans)asymptotic expansion of the inverse of a Dulac germ. It is worth noticing that we do not impose here to the Dulac series or germ to be parabolic. It obviously makes the computations more involved, but this generality is required by the definition of the continuous time lengths $\widehat{A}_{\widehat{f}}^{c}\left(\varepsilon\right)$ and $A_{f}^{c}\left(\varepsilon,x_0\right)$. Although the proof that the inverse of a transseries is also a transseries is given in full generality in \cite{dries}, we need in our framework an explicit
description of the monomials. This does not follow easily from \cite{dries}.

In Section \ref{sec:proof_theorem_A} we prove Theorem~A 
and in Section \ref{sec:proof_theorem_B} we prove Theorems~B and C, using the results from \cite {MRRZ2Fatou} about the Fatou coordinate  and from Section~\ref{sec:inverse} about the inverse. Finally, Section \ref{sec:appendix} is dedicated to technical proofs and definitions.

\section{Continuous time length of $\varepsilon$-neighborhoods of orbits
and main results}

\label{sec:continuous_time_length}

\subsection{The continuous time length of $\varepsilon$-neighborhoods of orbits
of germs}

Suppose that a germ $f$ is analytic on $(0,d)$, $d>0,$ has zero as a fixed point and that the function $\mathrm{id}-f$ is increasing and strictly positive. Let $x_{0}$ belong
to the basin of attraction of $0$, so that the orbit $\mathcal{O}^{f}(x_{0})=\{f^{\circ n}(x_{0}):\, n\in\mathbb{N}_{0}\}$
tends to zero. Recall that the function 
\[
\varepsilon\mapsto A_{f}(x_{0},\varepsilon)
\]
denotes the 1-dimensional Lebesgue measure of the $\varepsilon$-neighborhood
of the orbit. By \cite{tricot}, $A_{f}(x_{0},\varepsilon)$ is calculated
by decomposing the $\varepsilon$-neighborhood of $\mathcal{O}^{f}(x_{0})$
in two parts: the \emph{nucleus} $N(x_{0},\varepsilon)$, and the
\emph{tail} $T(x_{0},\varepsilon)$. The nucleus is the overlapping
part of the $\varepsilon$-neighborhood, and the tail is the union
of the disjoint intervals of length $2\varepsilon$. They are determined
in function of the \emph{discrete critical time} $n_{\varepsilon}(x_{0})$,
which is described by the condition: 
\begin{equation}
\begin{cases}
f^{n_{\varepsilon}(x_{0})}(x_{0})-f^{n_{\varepsilon}(x_{0})+1}(x_{0})\leq2\varepsilon,\\
f^{n_{\varepsilon}(x_{0})-1}(x_{0})-f^{n_{\varepsilon}(x_{0})}(x_{0})>2\varepsilon.
\end{cases}\label{neps}
\end{equation}
Then 
\begin{equation}
A_{f}(x_{0},\varepsilon)=|N(x_{0},\varepsilon)|+|T(x_{0},\varepsilon)|=\big(f^{n_{\varepsilon}(x_{0})}(x_{0})+2\varepsilon\big)+n_{\varepsilon}(x_{0})\cdot2\varepsilon.\label{nhood}
\end{equation}
The function $\varepsilon\mapsto A_f(x_0,\varepsilon)$ describes the density of the orbit $\mathcal O^f(x_0)$ of $x_0$.

Now suppose additionally that $f$ can be embedded as the time-one
map in a flow $\{f^{t}\}$, $f^{t}$ analytic on $(0,d)$, of class
$C^{1}$ in $t\in\mathbb{R}$. An embedding
in a flow allows us to define the \emph{continuous critical time with
respect to the flow $\{f^{t}\}$}, denoted by $\tau_{\varepsilon}(x_0)$,
in analogy to \eqref{neps} in the discrete case. Furthermore, in analogy to \eqref{nhood}, we define the \emph{continuous time length of
the $\varepsilon$-neighborhood of orbit $\mathcal{O}^{f}(x_{0})$
with respect to the flow $\{f^{t}\}$}, $A_f^c(\varepsilon,x_0)$ (see Definition~\ref{cont}). As will be shown in Proposition
\ref{prop:integer_part}, it turns out that the discrete critical time $n_{\varepsilon}(x_{0})$ in the standard definition \eqref{nhood}
is the \emph{ceiling function} of the continuous critical time $\tau_{\varepsilon}(x_{0})$
(the smallest integer bigger than or equal to $\tau_{\varepsilon}\left(x_{0}\right)$).

\begin{defi}[The continuous time length of \(\varepsilon\)-neighborhoods of orbits]\label{cont}
Assume that $f$ embeds as the time-one map in a flow $\{f^{t}\}$,
$f^{t}$ analytic on $(0,d)$. The \emph{continuous time length of
the $\varepsilon$-neighborhood of orbit $\mathcal{O}^{f}(x_{0})$
with respect to the flow $\{f^{t}\}$} is defined as: 
\begin{equation}
A_{f}^{c}(x_{0},\varepsilon)=\big(f^{\tau_{\varepsilon}(x_{0})}(x_{0})+2\varepsilon\big)+\tau_{\varepsilon}(x_{0})\cdot2\varepsilon.\label{Ac}
\end{equation}
Here, $\tau_{\varepsilon}(x_{0})$ is the \emph{continuous critical
time} for the $\varepsilon$-neighborhood of the orbit $\mathcal{O}^{f}(x_{0})$,
defined by the equation: 
\begin{equation}
f^{\tau_{\varepsilon}(x_{0})}(x_{0})-f^{\tau_{\varepsilon}(x_{0})+1}(x_{0})=2\varepsilon.\label{tau}
\end{equation}
\end{defi} 

We have shown in \cite[Section 4]{MRRZ2Fatou} that the embedding in a flow is closely related to the existence of a \emph{Fatou
coordinate} for $f$, which translates $f$ to a shift by $1$ and which is defined as follows:

\begin{defi}[Fatou coordinate]\label{def:ffatou}~

1. Let $f$ be an analytic germ on $(0,d)$, $d>0$. We say that a strictly monotone
analytic germ $\Psi$ on $(0,d)$ is a \emph{Fatou coordinate} for
$f$ if 
\begin{equation}
\Psi(f(x))-\Psi(x)=1,\ x\in(0,d),\ d>0.\label{eq:fat0}
\end{equation}

2. Let $\widehat{f}\in\widehat{\mathcal{L}}$ be parabolic. We say
that $\widehat{\Psi}\in\widehat{\mathfrak{L}}$ is a \emph{formal
Fatou coordinate} for $\widehat{f}$ if the following equation is
satisfied formally in $\widehat{\mathfrak{L}}$ : 
\begin{equation*}
\widehat{\Psi}(\widehat{f})-\widehat{\Psi}=1.
\end{equation*}
\end{defi}

\begin{prop}[Reformulation of definition~\eqref{Ac} of $A_{f}^{c}(x_{0},\varepsilon)$]\label{cor} Assume that the germ $f$ embeds as the
time-one map in a $C^1$-flow $\{f^{t}\}$, $f^{t}$ analytic on $(0,d)$. Let the germ $\xi:=\frac{d}{dt}f^t\big|_{t=0}\big.$ be non-oscillatory (no accumulation of zero points at $0$). Let $g=\mathrm{id}-f$. Then
\begin{equation}
A_{f}^{c}(x_{0},\varepsilon)=\big(g^{-1}(2\varepsilon)+2\varepsilon\big)+2\varepsilon\cdot\big[\Psi\big(g^{-1}(2\varepsilon)\big)-\Psi(x_{0})\big].\label{eq:germac}
\end{equation}
\end{prop}
We have proven in \cite[Proposition 4.4]{MRRZ2Fatou} that the embedding
of $f$ as the time-one map of an analytic $C^1$-flow with non-oscillatory $\xi$ implies the existence
of an analytic Fatou coordinate for $f$. Therefore $\Psi$ in \eqref{eq:germac} exists by the assumptions of Proposition~\ref{cor}. The monotonicity of $\Psi$ is equivalent to the non-oscillatority
at $0$ of the vector field $\xi$ for the flow $\{f^{t}\}$ (see
the proof of Proposition~4.4 in \cite{MRRZ2Fatou}).

\begin{proof} Using $g=\mathrm{id}-f$ and $g$ strictly increasing,
\eqref{tau} simplifies to 
\[
g\big(f^{\tau_{\varepsilon}(x_{0})}(x_{0})\big)=2\varepsilon,\ f^{\tau_{\varepsilon}(x_{0})}(x_{0})=g^{-1}(2\varepsilon).
\]
The equation \eqref{eq:fat0} for the Fatou coordinate now gives:
\[
\tau_{\varepsilon}(x_{0})=\Psi(f^{\tau_{\varepsilon}(x_{0})}(x_{0}))-\Psi(x_{0})=\Psi(g^{-1}(2\varepsilon))-\Psi(x_{0}).
\]
Putting this into definition \eqref{Ac}, we get the desired formula.
\end{proof}

\begin{prop} \label{prop:integer_part}Let $f$ and $\mathcal{O}^{f}(x_{0})$
be as above. Let $\{f^{t}\}$ and $\Psi$ be as in Proposition~\ref{cor}. Then: 
\[
n_{\varepsilon}(x_{0})=\lceil\tau_{\varepsilon}(x_{0})\rceil,\ \varepsilon>0.
\]
Here, we denote by $\lceil a\rceil:=\min\{k\in\mathbb{Z}:\, a\leq k\}$,
$a\in\mathbb{R}$. \end{prop}

\begin{proof} Let $g=\mathrm{id}$. The relations \eqref{neps} and
\eqref{tau} which define $n_{\varepsilon},\,\tau_{\varepsilon}$
give: 
\[
g\big(f^{n_{\varepsilon}(x_{0})}(x_{0})\big)\leq2\varepsilon,\ g\big(f^{\tau_{\varepsilon}(x_{0})}\big)(x_{0})=2\varepsilon,\ g\big(f^{n_{\varepsilon}(x_{0})-1}(x_{0})\big)>2\varepsilon.
\]
Since $g$ is a strictly increasing germ (since $f$ is such), we
get that 
\[
f^{n_{\varepsilon}(x_{0})}(x_{0})\leq f^{\tau_{\varepsilon}(x_{0})}(x_{0})<f^{n_{\varepsilon}(x_{0})-1}(x_{0}).
\]
By monotonicity of $\Psi$, we get: 
\[
n_{\varepsilon}(x_{0})\geq\tau_{\varepsilon}(x_{0})>n_{\varepsilon}(x_{0})-1,
\]
which proves the result.\end{proof}

\begin{obs}
It is proven in \cite[Theorem A]{MRRZ2Fatou} that a parabolic Dulac germ $f$ embeds as the time
one map in a \emph{unique} flow which admits a sectional asymptotic expansion in $\mathcal{\mathfrak L}$. For
Dulac germs, this very flow is used in the definition of $A^c_f(\varepsilon,x_0)$ without explicit reference. 
\end{obs}

\subsection{The formal continuous time length of $\varepsilon$-neighborhoods
of orbits}

We define here an analogue of the continuous time length of $\varepsilon$-neighborhoods
of orbits \emph{in the formal setting}.
\medskip

We recall the necessary classes of transseries already introduced in \cite{MRRZ2Fatou}.  We put $\boldsymbol{\ell}_{0}:=x$, $\boldsymbol{\ell}:=\boldsymbol\ell_1:=\frac{1}{-\log x}$, 
and define inductively $\boldsymbol{\ell}_{j+1}=\boldsymbol{\ell}\circ\boldsymbol{\ell}_{j}$,
$j\in\mathbb{N}$, as symbols for iterated logarithms. 

\begin{defi}[The classes $\widehat {\mathcal L}_j^\infty$ and $\widehat{\mathfrak L}$]\label{def:eljot}
Denote by $\widehat{\mathcal{L}}_{j}^{\infty}$, $j\in\mathbb N_0$, the set of all transseries of the
type: 
\begin{equation*}\label{summable}
\widehat f(x)=
\sum_{i_{0}=0}^{\infty}\sum_{i_{1}=0}^{\infty}\cdots\sum_{i_{j}=0}^{\infty}a_{i_{0}\ldots i_{j}}x^{\alpha_{i_{0}}}\boldsymbol{\ell}^{\alpha_{i_{0}i_{1}}}\cdots\boldsymbol{\ell}_{j}^{\alpha_{i_{0}\cdots i_{j}}},\  a_{i_0\cdots i_j}\in\mathbb R,\ x>0,
\end{equation*}
where $\left(\alpha_{i_{0}\cdots i_{k}}\right)_{i_{k}\in\mathbb{N}}$
is a strictly increasing sequence of real numbers tending to $+\infty$, for
every $k=0,\ldots,j$. If moreover $\alpha_0> 0$ $($the infinitesimal cases$)$, we denote the class by $\widehat{\mathcal L}_j$. The subset of $\widehat{\mathcal L}_1$ resp. $\widehat{\mathcal L}_1^\infty$ of transseries with only integer powers of $\boldsymbol\ell$ will be denoted by $\widehat{\mathcal L}$ resp. $\widehat{\mathcal L}^\infty$. 

Put
$$
\widehat{\mathfrak L}:=\bigcup_{j\in\mathbb N_0} \widehat{\mathcal L}_j^\infty
$$
for the class of all power-iterated logarithm transseries of finite depth in iterated logarithms.\end{defi}

The classes $\widehat{\mathcal{L}}_{j}^{\infty}$, $j\in\mathbb{N}_{0}$, are the sub-classes of power-iterated logarithm transseries, whose support
is any well-ordered subset of $\mathbb{R}^{j+1}$ (for the lexicographic
order). We restrict only to the subclass with strictly increasing exponents. For Dulac germs and their expansions 
this condition is satisfied. 

Notice that $x=\boldsymbol{\ell}_{0}$. The classes $\widehat{\mathcal{L}}_{0}$ or $\widehat{\mathcal{L}}_{0}^{\infty}$
are made of formal power series: 
\[
\widehat{f}(x)=\sum_{i\in\mathbb N}a_{i}x^{\alpha_i},\ a_{i}\in\mathbb{R},\ x>0,
\]
such that $(\alpha_i)$ is a strictly increasing real sequence tending to $+\infty$. \\
\smallskip

For $\widehat f\in \widehat{\mathcal L}_j^\infty$, we denote by $\mathrm{Lt}(\widehat f)$ its \emph{leading term} $a_{\gamma_{0},\gamma_{1},\ldots,\gamma_{j}}\, x^{\gamma_{0}}\boldsymbol{\ell}^{\gamma_{1}}\boldsymbol{\ell}_{2}^{\gamma_{2}}\cdots\boldsymbol{\ell}_{j}^{\gamma_{j}}$. The tuple $(\gamma_0,\gamma_{1},\ldots,\gamma_{j})$ is called \emph{the
order} of $\widehat f$, and is denoted by $\text{ord}(\widehat f)=(\gamma_0,\gamma_{1},\ldots,\gamma_{j})$. The transseries $\widehat f\in\widehat{\mathcal{L}}_{j}$, $j\in\mathbb N_0$, is called \emph{parabolic} if $\text{ord}(\widehat f)=(1,0,\ldots,0).$
\bigskip

Let $\widehat{f}\in\widehat{\mathcal{L}}$ be parabolic. Unlike a
germ, a transseries does not have orbits. Indeed, evaluating a transseries
$\widehat{f}$ at a point different from zero is senseless. Therefore, Definition~\ref{cont}
of the continuous time length of $\varepsilon$-neighborhoods of orbits
cannot be directly transported to the formal setting. Fortunately,
we can use the \emph{equivalent} definition \eqref{eq:germac}
from Proposition~\ref{cor}. Note that all the functions used in \eqref{eq:germac}
have their direct formal analogues. The formal definition is independent of the orbit.

Moreover, we have shown in \cite{mrrz2} and we recall in
\cite{MRRZ2Fatou} that any parabolic transseries $\widehat{f}\in\widehat{\mathcal{L}}$
can be embedded in a \emph{unique} $C^{1}$-flow $\{\widehat{f}^{t}\}_{t\in\mathbb{R}}$
of elements of $\widehat{\mathcal{L}}$. That is, its formal Fatou
coordinate $\widehat{\Psi}$ in $\widehat{\mathcal{L}}_{2}^{\infty}$,
as defined in Definition~\ref{def:ffatou}, is \emph{unique} in $\widehat{\mathfrak L}$ (up
to an additive constant) (see \cite[Proposition 4.4]{MRRZ2Fatou}). Therefore,
unlike for germs, the definition of the emph{formal} continuous time length
in the formal setting is unambiguous (it does not depend on the chosen
flow).


\begin{defi}[The formal continuous time length of \(\varepsilon\)-neighborhoods of orbits]\label{fa}
Let $\widehat{f}\in\widehat{\mathcal{L}}$ be parabolic. Let $\widehat{g}=\mathrm{id}-\widehat{f}$.
We define the \emph{formal continuous time length of $\varepsilon$-neighborhoods
of orbits of $\widehat{f}$} by: 
\begin{equation}
\widehat{A}_{\widehat{f}}^c(\varepsilon)=\widehat{g}^{-1}(2\varepsilon)+2\varepsilon\cdot\widehat{\Psi}\big(\widehat{g}^{-1}(2\varepsilon)\big).\label{Acf}
\end{equation}        
Here, $\widehat{g}^{-1}$ denotes the formal inverse of $\widehat{g}$,
$\widehat{\Psi}$ is the formal Fatou coordinate for $\widehat{f}$,
which exists and is unique by the main Theorem of \cite{MRRZ2Fatou}. \end{defi}

Note that $\widehat{A}_{\widehat{f}}^{c}(\varepsilon)$ defined by \eqref{Acf} is unique up to a term $K\cdot\varepsilon$, $K\in\mathbb R$, due to the fact that the formal Fatou coordinate $\widehat \Psi$ is unique up to an additive constant.
\smallskip
 
\begin{thmA}Let $\widehat{f}\in\widehat{\mathcal{L}}$ be parabolic. Then the formal continuous time length of $\varepsilon$-neighborhoods of orbits $\widehat{A}_{\widehat{f}}^{c}(\varepsilon)$ belongs to $\widehat{\mathcal{L}}_{2}$.
\end{thmA}
We prove that the formal Fatou coordinate
$\widehat{\Psi}$ and the formal inverse $\widehat{g}^{-1}$ exist in $\widehat{\mathfrak L}$, and that $\widehat{A}_{\widehat{f}}^{c}(\varepsilon)$ defined by \eqref{Acf} belongs to $\widehat{\mathcal L}_2$.
The proof of Theorem~A and a precise form of transseries $\widehat{A}_{\widehat{f}}^{c}(\varepsilon)$
is given in Section~\ref{sec:proof_theorem_A}. 

In particular, it will be proved that, if the leading term of $\widehat{g}$
does not involve a logarithm, then there exists at most one term
of $\widehat{A}_{\widehat{f}}^{c}(\varepsilon)$ which
contains a ``double logarithm''.

\subsection{Theorem B}

By $\mathcal G$ we denote the set of all one-dimensional germs defined on some positive open neighborhood of the origin. Furthermore, by $\mathcal G_{AN}\subset \mathcal{G}$ we denote the set of all germs which are moreover analytic on some positive open neighborhood of the origin.
\smallskip

\begin{defi}[Dulac germs, definition from \cite{mrrz2}]\label{def:dulac_germ} \  

1. We say that $\widehat{f}\in\widehat{\mathcal{L}}$ is a \emph{Dulac
series} $($\cite{Dulac}, \cite{ilya}, \cite{roussarie}$)$ if it is of the form: 
\begin{equation*}
\widehat{f}=\sum_{i=1}^{\infty}P_{i}(\boldsymbol{\ell})x^{\alpha_{i}},
\end{equation*}
where $P_{i}$ is a sequence of polynomials and $(\alpha_{i})_{i}$, $\alpha_{i}>0,$ is a strictly increasing sequence,
finite or finitely generated tending to $+\infty$.

2. We say that $f\in \mathcal G_{AN}$ is a \emph{Dulac germ} if:
\begin{itemize}
\item there exists a sequence of \emph{polynomials} and a strictly increasing, finitely generated sequence $(\alpha_i)$ tending to $+\infty$ or finite, such that
$$
f-\sum_{i=1}^{n} P_i(\boldsymbol\ell) x^{\alpha_i}=o(x^{\alpha_{n}}),\ n\in\mathbb N,
$$ 
\item $f$ is \emph{quasi-analytic}: it can be extended to an analytic, bounded function to a standard quadratic domain in $\mathbb C$, as precisely defined by Ilyashenko, see \cite{ilya}, \cite{roussarie}. 
\end{itemize}
If moreover $P_1\equiv1$, $\alpha_1=1$, and at least one of the polynomials $P_i$, $i>1$, is not zero, then $f$ is called a \emph{parabolic Dulac germ}. 
\end{defi}
Note that the germs of first return maps of hyperbolic polycycles of planar analytic vector fields are Dulac germs in the sense of Definition~\ref{def:dulac_germ}, see e.g. \cite{Dulac,ilya}.
\medskip

Let $f\in\mathcal G_{AN}$ be a parabolic Dulac germ and let $\mathcal{O}^{f}(x_{0})$
be an orbit of $f$ accumulating at $0$. Let $\widehat{f}\in\widehat{\mathcal{L}}$
be its Dulac expansion.
\medskip

Theorem~B is twofold. On one hand, it expresses an equality, up to a certain order,
between the classical (discrete time) length of the $\varepsilon$-neighborhood
of the orbit $\mathcal{O}^{f}(x_{0})$ and its continuous time
length as defined in Definition~\ref{cont}. On the other hand, it states that the continuous time
length of the $\varepsilon$-neighborhood of the orbit $\mathcal{O}^{f}(x_{0})$
admits the unique (trans)asymptotic expansion of a particular type in $\widehat{\mathfrak L}$ (what we call \emph{the sectional asymptotic expansion with respect to integral sections}, see \cite[Section 3]{MRRZ2Fatou} and Section~\ref{sec:asymptotic_expansions} here). Moreover, this expansion is equal to the formal continuous time length of $\varepsilon$-neighborhoods
of orbits for $\widehat{f}$, up to a term $\varepsilon K$, where $K$ is an arbitrary
constant. 
\medskip

Let us explain shortly the importance of Theorem~B. It was shown
in \cite{nonlin} that, for analytic parabolic germs, the classical
length of the $\varepsilon$-neighborhood of an orbit $\varepsilon\mapsto A_f(x_0,\varepsilon)$
does not have a complete asymptotic expansion in any scale of continuous functions, as $\varepsilon\to 0$. Moreover, the germ $\varepsilon\mapsto A_f(x_0,\varepsilon)$ does not belong to $\mathcal G_{AN}$ (there is an accumulation of singularities at $\varepsilon=0$). 
The continuous time length of the $\varepsilon$-neighborhood
of an orbit $\varepsilon\mapsto A_f^c(x_0,\varepsilon)$ is an \emph{analytic} generalization of the standard length and belongs to $\mathcal G_{AN}$. Theorem B shows that, for parabolic Dulac germs, its asymptotic expansion as $\varepsilon\to 0$ gives a continuation of the asymptotic
expansion of the classical length from the moment where the former
expansion ceases to exist. 
\medskip

For a Dulac germ $f\in \mathcal G_{AN}$, we consider a function $\varepsilon\mapsto A^c_f(x_0,\varepsilon)\in \mathcal G_{AN}$ defined in \eqref{eq:germac}. We have shown in the Theorem in \cite{MRRZ2Fatou} that such a Fatou coordinate exists and is unique up to an additive constant. Consequently, $\varepsilon\mapsto A^c_f(x_0,\varepsilon)$ is unique up to an additive term $\varepsilon\cdot K$, $K\in\mathbb R$.

We have noticed in \cite[Section~3]{MRRZ2Fatou} that the (trans)asymptotic expansion in $\widehat{\mathfrak L}$ is not uniquely defined. This is due to the non-unique summability of asymptotic series at limit ordinal steps, due to the fact that $x=e^{-\frac{1}{\boldsymbol\ell}},\ \boldsymbol\ell=e^{-\frac{1}{\boldsymbol\ell_2}}$ etc, see the example in \cite[Section~3]{MRRZ2Fatou} for better explanation. To ensure uniqueness of the asymptotic expansion of the Fatou coordinate in $\widehat{\mathfrak L}$, we have introduced in \cite{MRRZ2Fatou} the notion of \emph{sectional asymptotic expansions with respect to an integral section}. The idea was to define a \emph{section function} which prescribes a unique choice of the sum at limit ordinal steps, dictated by the Abel equation for the Fatou coordinate. In Section~\ref{sec:asymptotic_expansions} of this paper, we will extend the \emph{integral section function} to a wider class of transseries in $\widehat{\mathfrak L}$ which include also ${\widehat A}_{\widehat f}^{c}(\varepsilon)$.

\begin{thmB} Let $f\in \mathcal G_{AN}$, $f(x)=x-ax^\alpha\boldsymbol\ell^m+o(x^\alpha\boldsymbol\ell^m),\ a>0,\ \alpha>1,\ m\in\mathbb N_0^-$, be a parabolic Dulac germ . Let $\widehat{f}\in\widehat{\mathcal{L}}$ be its Dulac expansion. Let $\mathcal{O}^{f}(x_{0})$
be any orbit of $f$. Let $\varepsilon\mapsto A_f(x_0,\varepsilon)$, $\varepsilon\mapsto A_f^c(x_0,\varepsilon)$ and $\widehat{A}_{\widehat{f}}^{c}(\varepsilon)$ be defined as above. Then:

$(i)$ $\widehat{A}_{\widehat{f}}^{c}(\varepsilon)\in\widehat{\mathcal L}_2$.

$(ii)$ The sectional asymptotic expansion of $A_{f}^{c}(x_{0},\varepsilon)$ with respect to integral sections is unique in $\widehat{\mathfrak L}$ up to a term $\varepsilon K$, $K\in\mathbb{R}$. Different choices of integral sections correspond to different constants $K$ in the expansion. Moreover, up to $\varepsilon K$, the expansion is equal to $\widehat{A}_{\widehat{f}}^{c}(\varepsilon).$ 

$(iii)$ $A_{f}^{c}(x_{0},\varepsilon)-A_{f}(x_{0},\varepsilon)=\varepsilon^{2-\frac{1}{\alpha}}\boldsymbol\ell^{\frac m \alpha} h(\varepsilon), $ where $h(\varepsilon)=O(1),\ \varepsilon\to 0,$ is a high-amplitude oscillatory\footnote{$\varepsilon\mapsto h(\varepsilon)$ \emph{high-amplitude oscillatory at $0$} means that there exist two sequences $(\varepsilon_n^1)\to 0$ and $(\varepsilon_n^2)\to 0$ with \emph{strongly separated} values of $h(\varepsilon)$. That is, if there exist $A,\ B\in\mathbb R,\ A<B,$ such that $h(\varepsilon_n^1)<A<B<h(\varepsilon_n^2),\ n\in\mathbb N$.} function at $0$ which does not admit a power-log asymptotic behavior.
\end{thmB}
Note that, since $\alpha>1$ ($f$ is Dulac), $2-\frac{1}{\alpha}>1$. Therefore, the two lengths still coincide at the order $\varepsilon^1$ in $\varepsilon$.
\medskip 

A direct consequence of Theorem B $(iii)$ is the following statement, which motivates the introduction of the continuous time length of the $\varepsilon$-neighborhood of orbits of Dulac maps, as a natural generalization of the standard length $\varepsilon\mapsto A_f(x_0,\varepsilon)$. 
\begin{cory}
The formal continuous length of $\varepsilon$-neighborhoods of orbits $\widehat A_{\widehat f}^{c}(\varepsilon)$ is a continuation of the asymptotic expansion of $\varepsilon\mapsto A_{f}(x_{0},\varepsilon)$ after its remainder term $\varepsilon^{2-\frac{1}{\alpha}}\boldsymbol\ell^m r(\varepsilon)$, $r(\varepsilon)=O(1)$ high-amplitude oscillatory,  which does not admit a power-log asymptotic behavior. That is, after the point where an asymptotic expansion of $\varepsilon\mapsto A_f(x_0,\varepsilon)$ in a power-logarithm scale ceases to exist. 
\end{cory}
The statement follows directly from Theorem~B $(iii)$ and $(i)$, since $\varepsilon\mapsto A_f(x_0,\varepsilon)$ admits a complete power-log asymptotic expansion, $\widehat A_{\widehat f}^c(\varepsilon)\in\widehat{\mathcal L}_2$.

\bigskip

Theorem~B $(i),\,(ii)$ can be resumed in the following commutative diagram (Figure~\ref{diag}). In the diagram, $\Psi\in \mathcal G_{AN}$ denotes the Fatou coordinate of the Dulac germ $f$ and $\widehat\Psi$ the formal Fatou coordinate of its Dulac expansion $\widehat f$. Recall from \cite{MRRZ2Fatou} that the (unique up to an additive constant) sectional asymptotic expansion in $\widehat{\mathfrak L}$ of $\Psi$ with respect to any integral section is equal to the formal Fatou coordinate $\widehat\Psi$.
\begin{figure}[h!]
	
	\[
	\xyC{0.3cm}\xymatrix{ & f\ar@{->}[d]\ar@{->}[rrrrrrr]\sp-{\text{Dulac expansion}}\ar@{->}[d] &  &  &  &  &  &  & \widehat{f}\ar@{->}[d]\\
		\text{in }\mathcal{G}_{\mathrm{AN}} & \psi\ar@{->}[d]\ar@{->}[rrrrrrr]\sp-{\text{asymptotic expansion}}\sb-{\text{w.r.t. an integral section}} &  &  &  &  &  &  & \widehat{\psi}\ar@{->}[d] & \text{in }\widehat{\mathfrak L}\\
		& A_{f}^{c}\left(x_{\text{0}},\varepsilon\right)\ar@{->}[rrrrrrr]\sp-{\text{asymptotic expansion}}\sb-{\text{w.r.t. an integral section}} &  &  &  &  &  &  & \widehat{A}_{\widehat{f}}^{c}\left(\varepsilon\right)
	}
	\]
	\label{diag} \caption{The commutative diagram for a Dulac germ $f\in\mathcal{G}_{AN}$.} \end{figure}
\medskip

The proof of Theorem~B is given in Section~\ref{sec:proof_theorem_B}. 

\subsection{Theorem C}\label{sec:app}

Theorem~C expresses the formal class of a Dulac germ $f\in\mathcal G_{AN}$ from the initial (i.e. existent) part of the sectional asymptotic expansion of the length of the $\varepsilon$-neighborhood $A_f(x_0,\varepsilon)$ of one of its orbits. This is a generalization of the result from \cite{resman} for regular parabolic germs belonging to $\mathbb R\{x\}$. 

The proof of Theorem~C is given in Section~\ref{sec:ptC}. It relies on Theorem~B, which allows to work with the initial part of the expansion of the continuous time length $\varepsilon\mapsto A_f(x_0,\varepsilon)$ instead of the initial part of the discrete time length $\varepsilon\mapsto A_f(x_0,\varepsilon)$. The terms are given by the beginnning of the formal continuous length of the $\varepsilon$-neighborhoods of orbits, $\widehat A_{\widehat f}(\varepsilon)$. The expression of the terms in the continuous time length is in terms of the Fatou coordinate and the inverse, which can be explicitely determined. Working with inequalities in the definition of the discrete time length to estimate its initial terms would be much more cumbersome.

It was proved in \cite[Theorem~A]{mrrz2} that a formal normal form
in $\widehat{\mathcal{L}}$ for a (normalized) parabolic Dulac germ $f\in\mathcal G_{AN}$ of the form
\begin{equation}
f(x)=x-x^{\alpha}\boldsymbol\ell^m+o(x^\alpha\boldsymbol\ell^m),\ \alpha>1,\ m\in\mathbb N_0^{-},\ a\in\mathbb R, \label{eq:normdul}
\end{equation}
is given by 
\[
f_{0}(x)=x-x^{\alpha}\boldsymbol\ell^m+\rho x^{2\alpha-1}\boldsymbol{\ell}^{2m+1},\ \rho\in\mathbb R.
\]
The formal invariants are $(\alpha,m,\rho)$. Note that, since $f$ is Dulac, necessarily $m\in\mathbb N_0^{-}$ and $\alpha>1$.

\begin{thmC}[Formal normal form of a Dulac germ from fractal properties of orbits] Let $f(x)=x-x^\alpha\boldsymbol\ell^m+o(x^\alpha\boldsymbol\ell^m),\ \alpha>1,\ m\in\mathbb N_0^-,$ be a parabolic Dulac germ. Let $\mathbf s$ be an integral section. The formal invariants $(\alpha,m,\rho)$ can be expressed from finitely many terms from the initial (existent) part of the sectional asymptotic expansion with respect to $\mathbf s$ of the length of the $\varepsilon$-neighborhood of one of its orbits, $\varepsilon\mapsto A_f(x_0,\varepsilon)$. More precisely, we read $\alpha$ and $m$ from the exponents of the leading monomial in the expansion of $\varepsilon\mapsto A_f(x_0,\varepsilon/2)$:
 $$c \varepsilon^{1/\alpha}\boldsymbol\ell^{-m/\alpha},\ c\in\mathbb R.$$ The formal invariant $\rho$ is a polynomial function in the coefficients of the monomials $\varepsilon^1\boldsymbol\ell^r\boldsymbol\ell_2^s$, $r,s\in\mathbb Z$ and $(r,s)\prec (0,0)$, in the expansion of $\varepsilon\mapsto A_f(x_0,\varepsilon/2)$. The coefficients of the polynomial are explicit universal functions of $\alpha,\,m,\,r,\,s$ and depend on the germ $f$ only through $\alpha$ and $m$.  
\end{thmC}
Note that, by Theorem~B $(iii)$, the order at which the power-log asymptotic expansion of $\varepsilon\mapsto A_f(x_0,\varepsilon)$ ceases to exist is $O(\varepsilon^{1+\delta})$ for some $\delta>0$. Therefore, the power-logarithm terms with power of $\varepsilon$ equal to $\varepsilon^1$, which are needed for reading the formal invariant $\rho$, are still contained in the existent initial part of the expansion. Since $(r,s)\prec (0,0)$, their orders are even strictly lower than $\varepsilon$. The last monomial whose coefficient is used to read $\rho$ is $\varepsilon\boldsymbol\ell_2^{-1}$, for which $(r,s)=(0,-1)$.\\

Note also that in Theorem~C, we do not use the coefficient of the monomial $\varepsilon$ in the expansion of $\varepsilon\mapsto A_f(x_0,\varepsilon)$ for expressing the formal invariant $\rho$ (since $(r,s)\prec (0,0)$ and $r=s=0$ is not used). This is important since, by Theorem~B $(i)$, this coefficient is not unique in the expansion.\\
                 
In the special case when the leading term of $g=\mathrm{id}-f$ \emph{does not contain a logarithm in the leading term}, Theorem~C significantly simplifies and the formal invariants are read from just two terms in the asymptotic expansion of $\varepsilon\mapsto A_f(x_0,\varepsilon)$:
\begin{cory}[Theorem~C in the special case $m=0$]\label{cor:nolog}
Let $f=x-ax^\alpha\boldsymbol+o(x^\alpha),\ \alpha>1,$ be a parabolic Dulac germ. Let $\mathbf s$ be an integral section. The formal invariants $(\alpha,\rho)$ of $f$ can be expressed by two terms in the initial part of the asymptotic expansion of $\varepsilon\mapsto A_f(x_0,\varepsilon/2)$ of any orbit. More precisely, $\alpha$ is recovered from the exponent of the leading term $c\varepsilon^{1/\alpha}$, $c\in\mathbb R$, and $\rho$ is the coefficient of the monomial $\varepsilon\boldsymbol\ell_2^{-1}$ in the expansion. Additionally, this is the only monomial in the expansion containing the double logarithm.
\end{cory}
\noindent The proof of Theorem C and of Corollary~\ref{cor:nolog} is given in Section~\ref{sec:ptC}. 
\medskip

\begin{obs}[\emph{The subclass of analytic germs}] In the \emph{regular} case of parabolic analytic germs we can consider different formal normal forms, depending on the class in which we allow formal changes of variables. Let $f(x)=x-x^{k+1}+o(x^{k+1})$ be analytic. Its standard formal normal form with respect to formal Taylor series changes of variables is $f_0(x)=x-x^{k+1}+\rho x^{2k+1}$, with formal invariants $(k,\rho)$. Its formal normal form with respect to power-log formal changes of variables belonging to $\widehat{\mathcal L}$ is by \cite[Theorem~A, Example 6.3]{mrrz2} $\tilde f_0(x)=x-x^{k+1}+0\cdot x^{2k+2}\boldsymbol\ell$, with the formal invariants $(k,\tilde\rho=0)$. On the other hand, the formal Fatou coordinate $\widehat\Psi$ of $f$ is \emph{unique} (up to $K\varepsilon,\ K\in\mathbb R$) in the class $\widehat{\mathfrak{L}}$ (see \cite[Theorem]{MRRZ2Fatou}), as is consequently also the series $\widehat A_{\widehat f}^c$. The series $\widehat A_{\widehat f}^c$ for analytic germs is computed in Example~\ref{parabolic-analytic}. It is an universal object that does not depend on the chosen class for the change of variables. We can read both formal classes from appropriate terms in this series: 

$(i)$ We read the formal invariants $(k,\rho)$ with respect to formal Taylor series changes of variables in the exponent of the leading term and in the coefficient of the residual term $\varepsilon\boldsymbol\ell^{-1}$ of $\widehat A_{\widehat f}^c$, see Example~\ref{parabolic-analytic}. This is coherent with the previous results in \cite{resman}.

$(ii)$ Parabolic analytic germs are a subclass of Dulac germs with no logarithm in the leading term of $\mathrm{id}-f$, so Corollary~\ref{cor:nolog} is applicable. Therefore, the formal invariant $k$ is read from the exponent of the first term $c\varepsilon^{\frac{1}{k+1}},\ c\in\mathbb R$, and the formal invariant $\tilde{\rho}=0$ is the coefficient of the term $\varepsilon\boldsymbol\ell_2^{-1}$.  By Example~\ref{parabolic-analytic}, the term $\varepsilon \boldsymbol\ell_2^{-1}$does not exist in $\widehat A_{\widehat f}^c$, i.e., its coefficient is equal to $0$. Therefore, $\tilde{\rho}=0$.

\end{obs}                                     

\begin{example} The example shows that, in the \emph{general} case of a Dulac germ, \emph{general} meaning that:

(a) $g=\mathrm{id}-f$ contains a logarithm in the leading term, and

(b) $f$ is not already in the formal normal form,

\noindent only the coefficient of the term $\varepsilon\boldsymbol\ell_2^{-1}$ in $\widehat A_{\widehat f}^c(\varepsilon)$ is not sufficient to read the formal invariant $\rho$.  

Take a Dulac germ which is already in the formal normal form $f(x)=x-x^\alpha\boldsymbol\ell^m+\rho x^{2\alpha-1}\boldsymbol\ell^{2m+1}$. It can be computed that
\begin{equation}\label{eq:putti}
\widehat g^{-1}(y)=\alpha^{-\frac{m}{\alpha}}y^{\frac{1}{\alpha}}\boldsymbol\ell^{-\frac{m}{\alpha}}\big(1+\widehat F(\boldsymbol\ell_2,\frac{\boldsymbol\ell}{\boldsymbol\ell_2})\big)+\rho x\boldsymbol\ell+o(x\boldsymbol\ell),
\end{equation}
where $F$ is a Taylor expansion of a function of two variables.
By Proposition~\ref{prop:izvod}, putting \eqref{eq:putti} in \eqref{eq:aepsimp} we get that $\rho$ is exactly the coefficient of $\varepsilon\boldsymbol\ell_2^{-1}$ in $\widehat A_{\widehat f}^c(\varepsilon)$. 

However, if the germ is not in the formal normal form, the formal invariant may \emph{spread along} finitely many terms of the form $\varepsilon\boldsymbol\ell^r\boldsymbol\ell_2^s$, $(r,s)\preceq (0,-1)$, as described in Theorem~C. For example, if we add one additional term before the residual in the formal normal form, the new germ is:
$$
f_1(x)=x-x^\alpha\boldsymbol\ell^m+ax^{2\alpha-1}\boldsymbol\ell^{2m}+bx^{2\alpha-1}\boldsymbol\ell^{2m+1}.
$$
Here, we choose $a\in\mathbb R$ completely freely, and then choose $b$ such that $f_1$ belongs to the same formal class $(\alpha, \rho)$ as $f$. That is, we choose $b$ such that the coefficient in front of the term $x^{2\alpha-1}\boldsymbol\ell^{2m+1}$ after the change of variables that eliminates $ax^{2\alpha-1}\boldsymbol\ell^{2m}$ becomes equal to $\rho$. We compute:
$$
g_1^{-1}(y)=\alpha^{-\frac{m}{\alpha}}y^{\frac{1}{\alpha}}\boldsymbol\ell^{-\frac{m}{\alpha}}\big(1+F(\boldsymbol\ell_2,\frac{\boldsymbol\ell}{\boldsymbol\ell_2})\big)+\frac{a}{\alpha}x+d x\boldsymbol\ell+o(x\boldsymbol\ell),
$$
where $d\in\mathbb R$ depends on $a,\ b,\ m,\ \alpha$. By Proposition~\ref{lem:jen}, \begin{equation}\label{eq:ned}\rho=\frac{m}{\alpha}a+d.\end{equation} Then, by Proposition~\ref{prop:izvod} \eqref{eq:aepsimp}, the formal invariant is a combination of coefficients in front of $\varepsilon\boldsymbol\ell^{-1}$ and $\varepsilon\boldsymbol\ell_2^{-1}$ in $\widehat A_{\widehat f}^c(\varepsilon)$. The coefficient in front of $\varepsilon\boldsymbol\ell^{-1}$ is just $d$ and obviously by \eqref{eq:ned} not sufficient to read $\rho$. Recall that $a$ can be chosen arbitrarily. 

Furthermore, if we add more terms to $f$, $\rho$ will be a combination of finitely many coefficients of the monomials $\varepsilon\boldsymbol\ell^r\boldsymbol\ell_2^s$ up to $\varepsilon\boldsymbol\ell_2^{-1}$ in $\widehat A_{\widehat f}^c(\varepsilon)$.
\end{example}

\section{Sectional asymptotic expansions}
\label{sec:asymptotic_expansions}The Poincar\' e algorithm gives the unique asymptotic expansion $\widehat f$ of a Dulac germ $f$, called the Dulac expansion. The expansion is unique due to the fact that every power of $x$ is multiplied by \emph{finitely many} powers of logarithm. However, this is not the case neither for its Fatou coordinate nor for its continuous time length of $\varepsilon$-neighborhoods of orbits. Their expansions in $\widehat{\mathfrak L}$ are transfinite. Due to the fact that $x=e^{-\frac{1}{\boldsymbol\ell}}$, at every limit ordinal step the sum of the asymptotic series is not unique, leading to the non-uniqueness of the asymptotic expansion. This problem is illustrated in \cite[Section 3]{MRRZ2Fatou}. The problem of non-uniqueness of an asymptotic expansion in $\widehat{\mathfrak L}$ was solved in \cite{MRRZ2Fatou} by fixing a so-called \emph{section function} that attributes a unique germ to the asymptotic series on limit ordinal steps and thus defines the limit ordinal steps in the \emph{transfinite Poincar\' e algorithm}. For the Fatou coordinate of a Dulac germ we have introduced in \cite{MRRZ2Fatou} the appropriate \emph{integral section functions} based on solutions of the Abel equation.

Let us recall that the integral section function in \cite{MRRZ2Fatou} was defined only by its restriction on a subset of power series $\widehat{\mathcal L}_0^\infty$, which appeared in the algorithm for the formal Fatou coordinate of a Dulac series and which we have  named the \emph{integrally summable series} in $\widehat{\mathcal L}_0^\infty$. We repeat the definition of integrally summable series in $\widehat{\mathcal L}_0^\infty$ in Definition~\ref{defi} below. Then in Definition~\ref{diifi} we expand this definition to a subset of power-logarithm transseries $\widehat{\mathcal L}_1^\infty$ which we will call \emph{integrally summable in $\widehat{\mathcal L}_1^\infty$}. Such series appear in the formal continuous time lengths of $\varepsilon$-neighborhoods of orbits $\widehat A_{\widehat f}^c(\varepsilon)$ of Dulac maps. We then define the extension of an integral section function from the set of integrally summable series in $\widehat{\mathcal L}_0^\infty$ to a set of integrally summable series in $\widehat{\mathcal L}_1^\infty$.

After fixing the section function, the asymptotic expansions in $\widehat{\mathfrak L}$ with respect to the fixed section are unique. We consider the sectional asymptotic expansions of $\Psi$ and $A^c_f(x_0,\varepsilon)$ with respect to \emph{integral sections} defined in Definition~\ref{def:is} below. 

\begin{defi}[Integrally summable series in $\widehat {\mathcal L}_0^\infty$, Definition 3.9 in \cite{MRRZ2Fatou}]\label{defi}\ 

$(1)$ By $\widehat{\mathcal L}_0^I\subset \widehat{\mathcal L}_0^\infty$ we denote the set of all formal series $$\widehat f(y)=\sum_{n=N}^{\infty}a_n y^n\in\widehat{\mathcal L}_0^\infty,\ N\in\mathbb Z,\ a_n\in\mathbb R,$$ 
which are either: 
\begin{enumerate}
\item[(i)] convergent on an open interval $(0,\delta)$, $\delta>0$, or 

\item[(ii)] divergent and such that
 there exists $\alpha\in\mathbb R$, $\alpha\neq 0$, for which 
\begin{equation}\label{deff}
\frac{\mathrm d}{\mathrm dx}\Big(x^ {\alpha} \widehat f(\boldsymbol\ell)\Big)=x^{\alpha-1} R(\boldsymbol\ell),
\end{equation}
 formally in $\widehat{\mathcal L}^\infty$,
where $R$ is a convergent Laurent series.

\end{enumerate}

\noindent We call $\widehat {\mathcal L}_0^I$ the set of \emph{integrally summable series of $\widehat{\mathcal L}_0^\infty$}.
\medskip

$(2)$ \begin{enumerate}
\item[(i)] If $\widehat{f}$ is convergent, we define its \emph{integral sum} as the usual sum $($on an interval $(0,\delta))$.

\item[(ii)] If $\widehat f\in\widehat {\mathcal L}_0^I$ is divergent, we define its \emph{integral sum} $f\in\mathcal G_{AN}$ by:
\begin{equation}\label{eq:joj}
f(y):=
\frac{\int_d^{e^{-1/y}} s^{\alpha-1}R\big(\boldsymbol\ell(s)\big)\,\mathrm ds}{e^{-\frac{\alpha}{y}}},
\end{equation}
where $\boldsymbol\ell(s)=-\frac{1}{\log s}$. We put $d=0$ if $s^{\alpha-1}R\big(\boldsymbol\ell(s)\big)$ is integrable at $0$ $($i.e. $\alpha>0)$ and $d>0$ otherwise $($i.e. $\alpha<0)$.
\end{enumerate}
\end{defi} 

It was proved in \cite{MRRZ2Fatou} that the exponent $\alpha\neq 0$ in Definition~\ref{defi} $(1)\,(ii)$ is unique. Such $\alpha$ is called the \emph{exponent of integration of $\widehat f$}. Note that putting $\alpha=0$ in \eqref{deff} would imply $\widehat f$ convergent.

\begin{obs}(\cite{MRRZ2Fatou})
\begin{enumerate}
\item The integral sum $f$ of $\widehat f\in\widehat{\mathcal L}_0^I$, as defined by \eqref{eq:joj}, is 
unique for $\alpha>0$. It is 
\emph{unique only up to exponentially small term $Ce^{\frac{\alpha}{y}}$, $C\in\mathbb R$}, for $\alpha<0$, due to the possible choice of $d>0$. If $\widehat f$ is convergent, the integral sum is the standard sum, thus unique.

\item $\widehat f\in\widehat{\mathcal L}_0^I$ is the power asymptotic expansion of its integral sums $f$.
\end{enumerate}
\end{obs}

In the Theorem in \cite{MRRZ2Fatou}, we have shown that there exists a unique (up to an additive constant) Fatou coordinate $\Psi$ for a Dulac germ $f$ with an asymptotic expansion in the class $\widehat{\mathfrak L}$ and the unique (up to an additive constant) formal Fatou coordinate $\widehat\Psi\in\widehat{\mathfrak L}$ for its Dulac expansion $\widehat f$. We have also proved that the Fatou coordinate $\Psi$ of a Dulac germ $f$ admits the formal Fatou coordinate $\widehat\Psi\in\widehat{\mathcal L}_2^\infty$ as its (unique) sectional asymptotic expansion with respect to any integral section, up to an additive constant term. The definition of the Fatou coordinate is given in Definition~\ref{def:ffatou}.

\begin{defi}[Integrally summable series in $\widehat{\mathcal L}_1^\infty$]\label{diifi}\

1. By $\widehat{\mathcal L}_1^I\subset \widehat{\mathcal L}_1^\infty$, we denote the set of all transseries $\widehat F\in \widehat{\mathcal L}_1^\infty$ which can be regrouped as:
\begin{equation}\label{dvaa}
\widehat F(y)=\widehat G_1(y)\widehat f\big(\boldsymbol\ell (e^{-\frac{\gamma}{y}} \widehat h(y))\big)+\widehat G_0(y),
\end{equation}
where $\widehat f\in \widehat{\mathcal L}_0^I$, and $\widehat G_0,\ \widehat G_1,\ \widehat h$ and their first derivatives $\widehat G_0',\ \widehat G_1',\ \widehat h'$ are convergent transseries in $\widehat{\mathcal L}_1^\infty$ in the sense of \cite[Definition 3.7]{MRRZ2Fatou}, with the sums commuting with the derivative. Moreover, $\widehat h$ does not contain $\boldsymbol\ell(y)$ in the first term and $\gamma> 0$.

\noindent We call the class $\widehat{\mathcal L}_1^I\subset\widehat{\mathcal L}_1^\infty$ the class of \emph{integrally summable series in $\widehat{\mathcal L}_1^\infty$}. \smallskip

2. Let $f$ be \emph{the integral sum} of $\widehat f$ from \cite[Definition 3.9]{MRRZ2Fatou} $($defined up to a certain exponential factor$)$ and let $h,\ G_0,\ G_1$ be the sums of $\widehat h,\ \widehat G_0,\ \widehat G_1$ respectively. We call $F\in\mathcal G_{AN}$ given by: \begin{equation}\label{intsum}F(y)=G_1(y)f\big(\boldsymbol\ell (e^{-\frac{\gamma}{y}}h(y))\big)+G_0(y)\end{equation} \emph{the integral sum} of $\widehat F$. 
\end{defi}

\noindent Note that $\widehat{\mathcal L}_0^I\subset \widehat{\mathcal L}_1^I$ (put $\widehat G_1\equiv 1,\ \gamma=1, \ \widehat h\equiv 1,\ \widehat G_0\equiv 0$). Note also that the set of all convergent transseries from $\widehat{\mathcal L}_1^\infty$ is a subset of $\widehat{\mathcal L}_1^I$ (put $\widehat G_1\equiv 0$).

\medskip
Note that in general the decomposition~\eqref{dvaa} of $\widehat F\in\widehat{\mathcal L}_1^I$ is not unique, so its integral sum $F$ is not well-defined. Let $\alpha\in\mathbb R$ be the exponent of integration of $\widehat f$ in decomposition~\ref{dvaa}, where $\alpha=0$ if and only if $\widehat f$ is convergent. However, we prove in Proposition~\ref{difi} in the Appendix that the integral sum $F$ of $\widehat F$ is \emph{unique}, if $\alpha>0$, or \emph{unique up to an additive term} $c G_1(y)\cdot \big(e^{-\frac{\gamma}{y}} h(y)\big)^{-\alpha},\ c\in\mathbb R$, if $\alpha<0$. If $\alpha=0$, that is, if $\widehat f$ is convergent, $\widehat F$ is a convergent series in $\widehat{\mathcal L}_1^\infty$ with the unique sum. Here,  $\alpha,\ \gamma,\ \widehat h,\ \widehat G_1$ are elements of an arbitrary decomposition \eqref{dvaa} of $\widehat F$.
\medskip

Before stating the definition of integral sections, recall the definition of \emph{coherent} sections as section functions that attribute to any convergent transseries from $\widehat{\mathfrak L}$ its sum, see \cite[Definition 3.8]{MRRZ2Fatou}. As in \cite{MRRZ2Fatou}, for a given section function $\mathbf s$, let $\mathcal T^{\mathbf s}\subset\mathcal G_{AN}$ denote the set of all transseries which admit the sectional asymptotic expansion with respect to $\mathbf s$, and by $\widehat{\mathcal T}^s\subset\widehat{\mathfrak L}$ their expansions. See \cite[Section 3]{MRRZ2Fatou} for more details on section functions and sectional asymptotic expansions.

\begin{defi}[The integral sections]\label{def:is} Every coherent section $\mathbf s:\widehat{\mathcal T}^{\mathbf s}\to \mathcal T^{\mathbf s}$ whose restriction to $\widehat {\mathcal L}_1^I\subset\widehat{\mathcal L}_1^\infty$ is given by $\mathbf{s}\big|_{\widehat{\mathcal L}_1^I}(\widehat F)=F$, where $F$ is one integral sum of $\widehat F$ given in \eqref{intsum}, is called an \emph{integral section}. 
\end{defi}

Recall that the integral sections defined here are in fact restrictions of the integral sections defined in \cite[Definition 3.14]{MRRZ2Fatou}, since $\widehat{\mathcal L}_0^I\subset \widehat{\mathcal L}_1^I$, but we keep the same name for simplicity.
\medskip
 

Note that the request $\mathbf{s}\big|_{\widehat{\mathcal L}_1^I}(\widehat F)=F$ in Definition~\ref{def:is} does not contradict the definition of section functions \cite[Definition~3.2]{MRRZ2Fatou} and of coherence \cite[Definition~3.8]{MRRZ2Fatou}. By Proposition~\ref{integralsects} in the Appendix $\widehat {\mathcal L}_1^I\subset \widehat{\mathcal S}_0\cup\widehat{\mathcal S}_1^{\mathbf s}\subset \widehat {\mathcal T}^\mathbf s$ for $\mathbf s$ coherent, where $\widehat{\mathcal S}_0\subset \widehat{\mathcal L}_0^\infty$ denotes the set of all power asymptotic expansions of germs in $\mathcal G_{AN}$ and $\widehat{\mathcal S}_1^{\mathbf s}\subset \widehat{\mathcal L}_1^\infty\setminus \widehat{\mathcal L}_0^\infty$ denotes the sets of all sectional asymptotic expansions with respect to $\mathbf s$ of germs from $\mathcal G_{AN}$ in $\widehat{\mathcal L}_1^\infty\setminus \widehat{\mathcal L}_0^\infty$. This boils down to proving that integral sums $F$ defined in Definition~\ref{diifi} admit $\widehat F\in\widehat{\mathcal L}_1^I$ as their unique asympotic expansion with respect to coherent sections $\mathbf s$ by the transfinite Poincar\' e algorithm . Moreover, by definition of integral sums, if $\widehat F\in\widehat{\mathcal L}_1^I$ is convergent, then $F$ is its sum.
\smallskip 

In Proposition~\ref{difi} in the Appendix, we prove the following result about the non-uniqueness of integral sums:
\begin{obs}\label{rem:helpdif}Let $\mathbf s_1$ and $\mathbf s_2$ be two different integral sectons from Definition~\ref{def:is}. Let $\widehat F\in\widehat{\mathcal L}_1^I$. Then $F_1:=\mathbf{s}_1(\widehat F)$ and $F_2:=\mathbf{s}_2(\widehat F)$ differ by: $$F_1(y)-F_2(y)=\begin{cases}c\cdot G_1(y)\cdot\big(e^{-\frac{\gamma}{y}} h(y)\big)^{-\alpha},\ c\in\mathbb R,&\alpha<0,\\0, &\alpha\geq 0.\end{cases}$$ Here, $\alpha\in\mathbb R$ (the exponent of integration of $\widehat f$) ,\ $\gamma>0,\ \widehat h$ and $\widehat G_1$ are elements of an arbitrary decomposition \eqref{dvaa} of $\widehat F$. For $\alpha<0$, the difference is non-zero, but it is exponentially small.
\end{obs}

Let $\mathbf s$ be an integral section, as in Definition~\ref{def:is}. Obviously, if a germ $f\in\mathcal G$ admits a sectional asymptotic expansion in $\widehat{\mathcal L}_2^\infty$ with respect to the integral section $\mathbf s$, the expansion is \emph{unique}.
\bigskip

We show in Section~\ref{sec:inverse} that the sectional asymptotic expansion of the inverse $g^{-1}\in\mathcal G_{AN}$, $g=\mathrm{id}-f$, with respect to any fixed integral section $\mathbf s$ is equal to the formal inverse $\widehat g^{-1}\in\widehat{\mathcal L}_2$. Similarly, we have shown in \cite{MRRZ2Fatou}  that the sectional asymptotic expansion of the Fatou coordinate for a Dulac germ, $\Psi\in \mathcal G_{AN},$ with respect to any fixed integral section $\mathbf s$ is, \emph{up to a constant}, equal to the formal Fatou coordinate $\widehat\Psi\in \widehat{\mathcal L}_2^\infty$. In Section~\ref{sec:proof_theorem_B} we prove Theorem~B. We show that  the unique sectional asymptotic expansion  with respect to any fixed integral section $\mathbf s$ of the function of the continuous time length of $\varepsilon$-neighborhoods of orbits for $f$ Dulac, $\varepsilon\mapsto A^c_f(x_0,\varepsilon)\in \mathcal G_{AN}$, is, \emph{up to a term $C\varepsilon$, $C\in\mathbb R,$} equal to the formal continuous time length of $\varepsilon$-neighborhoods of orbits $\widehat A^c_{\widehat f}(\varepsilon)\in\widehat{\mathcal L}_2$ of its Dulac expansion $\widehat f$.

Moreover, we show that the sectional asymptotic expansions of $\Psi$ resp. $A^c_f(x_0,\varepsilon)$ with respect to different integral sections lead to different choices of constants  $C\in\mathbb R$ resp. $C\varepsilon$, $C\in\mathbb R$.
This does not cause problems, since the formal Fatou coordinate $\widehat\Psi$ is defined only up to an additive constant term. As a consequence, $\widehat A_c^{\widehat f}(\varepsilon)$ is defined only up to $C\varepsilon$, $C\in\mathbb R$, due to the term $\varepsilon\cdot \widehat\Psi\big(\widehat g^{-1}(2\varepsilon)\big)$ in its definition. Therefore, any choice of the integral section for the sectional asymptotic expansions is \emph{equally good}.

\section{Examples}

\label{sec:examples}

\medskip

Consider a formal Dulac series $\widehat{f}\in\widehat{\mathcal L}$. It is proven in
Proposition~\ref{th1} in Section \ref{sec:inverse} that, if the leading term of $\widehat{g}=\mathrm{id}-\widehat{f}$
involves a logarithm, then its formal inverse contains a double logarithm.
This will imply that it is also the case for the formal continuous
time length $\widehat{A}_{\widehat{f}}^{c}\left(\varepsilon\right)$ defined in Definition~\ref{fa}.
The following example shows another possible cause of the presence
of a double logarithm in $\widehat{A}_{\widehat{f}}^{c}\left(\varepsilon\right)$
-- the existence of a \emph{nonzero residual invariant} of $\widehat{f}$.
\begin{example}[The \emph{residual} double logarithm in the formal continuous time length $\widehat{A}_{\widehat{f}}^{c}\left(\varepsilon\right)$ of a Dulac series]\label{dubl}

According to the results of \cite{mrrz2}, every parabolic element
$\widehat f\in\widehat{\mathcal{L}}$ admits a normal form which is the exponential
of a simple vector field. It is convenient to work with an example
which is already given in this normal form. Consider the vector field
$X=\xi\left(x\right)\frac{\mathrm{d}}{\mathrm{d}x}$, where $\xi\left(x\right)=-\frac{x^{2}}{1-x+bx\boldsymbol{\ell}}$
with $b\in\mathbb{R}\setminus\left\{ 0\right\} $. Its formal time-one
map $\widehat{f}\in\widehat{\mathcal{L}}$ is given by the formula:
\begin{align*}
\widehat{f}\left(x\right) & =\mathrm{Exp}\left(X\right)\cdot\mathrm{id}=x+\xi\left(x\right)+\frac{1}{2}\xi'\left(x\right)\xi\left(x\right)+\text{h.o.t.}\\
 & =x-x^{2}+bx^{3}\boldsymbol{\ell}+\text{h.o.t.}
\end{align*}
The coefficient $b\in\mathbb{R}$ is the \emph{residual invariant}
of $\widehat{f}$.\emph{ }In order to compute the formal continuous
time length $\widehat A_{\widehat{f}}^{c}\left(\varepsilon\right)$ given by
Definition \ref{fa}, we need to compute the formal Fatou coordinate
$\widehat{\Psi}$ of $\widehat{f}$, and
the formal inverse of $\widehat{g}=\mathrm{id}-\widehat{f}$.
By the proof of \cite[Proposition~4.3]{MRRZ2Fatou}, we have:
\begin{align*}
&\widehat{\Psi}'\left(x\right) =\frac{1}{\widehat\xi\left(x\right)}=-\frac{1}{x^{2}}+\frac{1}{x}+\frac{b}{x\log x},\text{ hence}\\
&\widehat{\Psi}\left(x\right) =\frac{1}{x}+\log x+b\,\log(-\log x)+C,\ C\in\mathbb R.
\end{align*}
On the other hand:
\begin{align*}
\widehat{g}\left(x\right) & =x-\widehat{f}\left(x\right)=x^{2}-bx^{3}\boldsymbol{\ell}+\text{h.o.t.}\\
 & =x^{2}\circ\widehat{\varphi}\left(x\right),
\end{align*}
for some parabolic element $\widehat{\varphi}\left(x\right)\in\widehat{\mathcal{L}}$ of the form $\widehat\varphi(x)=x-\frac{b}{2}x^2\boldsymbol\ell+h.o.t$.
Hence:
\[
\widehat{g}^{-1}\left(x\right)=\widehat{\varphi}^{-1}\left(x^{1/2}\right)=x^{1/2}+bx\boldsymbol\ell+\text{h.o.t.}
\]
belongs to $\widehat {\mathcal{L}}$. In particular, by Proposition~\ref{th1} (2), it does not contain any double
logarithm. Let us now compute the following component of $\widehat{A}_{\widehat{f}}^{c}\left(\varepsilon\right)$:
\begin{align*}
\widehat{\Psi}\left(\widehat{g}^{-1}\left(2\varepsilon\right)\right) & =\frac{1}{\sqrt 2 \varepsilon^{1/2}(1+o(1))}+\log\left(\sqrt 2 \varepsilon^{1/2}(1+o(1))\right)+\\
&\hspace{4cm} +b\log\left(-\log\big(\sqrt 2\varepsilon^{1/2}(1+o(1))\big)\right)=\\
 & =\varepsilon^{-1/2}+\frac{1}{2}\log\varepsilon+b\log\left(-\log\varepsilon\right)+\text{h.o.t.}
\end{align*}
Here, $o(1)\in\widehat{\mathcal L}$ is of order strictly bigger than $\varepsilon^0$ in $\varepsilon$, and \emph{h.o.t.} belongs to $\widehat {\mathcal{L}}$ (no double logarithms).
Hence, we observe the presence of a \emph{unique} iterated logarithm in
$\widehat{A}_{\widehat{f}}^{c}\left(\varepsilon\right)$, caused by
the presence of a \emph{residual term} $bx^{3}\boldsymbol{\ell}$ (that is, $b\neq 0$) in the formal normal form of $\widehat f$.
\end{example}

\begin{example}[$\widehat{A}_{\widehat{f}}^{c}\left(\varepsilon\right)$ for Dulac series is in general transfinite]\label{parabolic-dulac-transfinite}

By \cite{mrrz2}, a parabolic transseries $\widehat{f}\in\mathcal{\widehat{L}}$
can be reduced to its normal form $\widehat{f}_{0}$ given as a formal exponential
by an action of a parabolic change of variables $\widehat{\varphi}\in\widehat{\mathcal{L}}$,
whose support in general has the order type strictly bigger than $\omega$. If
$\widehat{\Psi}_{0}$ is the formal Fatou coordinate of $\widehat{f}_{0}$ in $\widehat{\mathfrak L}$,
then $\widehat{\Psi}=\widehat{\Psi}_{0}\circ\widehat{\varphi}$ is
the formal Fatou coordinate of $\widehat{f}$. We deduce that the formal Fatou coordinate of $\widehat f$ and, consequently, its 
continuous time length $\widehat{A}_{\widehat{f}}^{c}\left(\varepsilon\right)$
are transseries whose support has the order type strictly bigger
than $\omega$.

Consider for example the Dulac series $\widehat f(x)=x-x^{2}\boldsymbol{\ell}^{-1}+x^{2}.$
We have shown in Example~6.2 in \cite{mrrz2} that the change of variables
reducing $\widehat f$ to its formal normal form $f_{0}(x)=\exp\Big(\frac{-x^2\boldsymbol\ell^{-1}}{1-x\boldsymbol\ell^{-1}+\frac{x}{2}}\Big).\mathrm{id}=x-x^{2}\boldsymbol{\ell}^{-1}+\rho x^{3}\boldsymbol{\ell}+\ldots$
is a transseries $\widehat{\varphi}\in\widehat{\mathcal{L}}$ whose
support is indexed by an ordinal \emph{strictly} bigger than $\omega$. Hence, it applies also to $\widehat\Psi$, as described here, and then to the formal continuous area $\widehat A^c_{\widehat f}(\varepsilon)$, by its definition.
\end{example}

\begin{example}[The \emph{regular} case, parabolic germs analytic at the origin]\label{parabolic-analytic}

Let $f(x)=x-x^{k+1}+a_{2}x^{k+2}+\ldots$, $k\in\mathbb{N}$, $a_{i}\in\mathbb{R},\ i\geq 2,$
be a germ of a parabolic diffeomorphism on $\mathbb{R}_{+}$ (prenormalized
for simplicity). Recall that its formal normal form is given by
$$
f_0(x)=x-x^{k+1}+\rho x^{2k+1}.
$$
Here, $(k,\rho)$ are the formal invariants.

Let $\mathcal{O}^{f}(x_{0})$ be the orbit of $f$ with initial
point $x_{0}>0$ close to the origin. Put $g:=\mathrm{id}-f$. 

The germ $f$ can be considered as the restriction to the positive
real axis of a complex parabolic germ $f(z)$ with real coefficients. It is known
that there exist two sectorial Fatou coordinates $\Psi_{\pm}(z)$,
with the common asymptotic expansion: 
\begin{equation}
\widehat{\Psi}(z)=\Big(\frac{1}{kz^{k}}+\big(\frac{k+1}{2}-\rho\big)\log z\Big)\circ\widehat{\varphi},\ \ \rho\in\mathbb{R},\ \widehat{\varphi}\in z+z^{2}\mathbb{R}[[z]].\label{eq:psi-1}
\end{equation}
The series $\widehat{\Psi}$ has real coefficients. Note that on the real line
there exists an analytic Fatou coordinate $\Psi:\mathbb{R}_{>0}\to\mathbb{R}_{>0}$,
with the asymptotic expansion $\widehat{\Psi}(x)$, as $x\to 0$, which is the restriction to $\mathbb{R}_{>0}$
of the complex sectorial Fatou coordinate $\Psi_+$, analytic on the attracting petal centered at $\mathbb R_+$.

The asymptotic expansion of $g^{-1}(2\varepsilon)$,
as $\varepsilon\to0$, is easily computed (write $g(x)=x^{k+1}\circ\varphi_{1}(x),\ \varphi_{1}\in x+x^{2}\mathbb{R}\{x\}$):
\begin{equation}
\widehat g^{-1}(2\varepsilon)=\widehat \varphi_1^{-1}\big((2\varepsilon)^{\frac{1}{k+1}}\big),\ \widehat\varphi_{1}^{-1}\in x+x^{2}\mathbb{R}[[x]].\label{eq:g}
\end{equation}

Using \eqref{eq:psi-1} and \eqref{eq:g} and Definition~\ref{fa} of $\widehat A_{\widehat f}^{c}(\varepsilon)$, we compute the formal length $\widehat A_{\widehat f}^{c}(\varepsilon)$ for the formal Taylor expansion $\widehat f\in\mathbb R[[x]]$ of $f$:
\begin{align}\label{eq:asyy}
\widehat A_{\widehat f}^{c}(\varepsilon)& =\frac{k+1}{k}2^{\frac{1}{k+1}}\varepsilon^{\frac{1}{k+1}}+b_{2}\varepsilon^{\frac{2}{k+1}}+\cdots+b_{k}\varepsilon^{\frac{k}{k+1}}+\big(k+1-2\rho\big)\varepsilon\log\varepsilon+\\
 &+b_{k+2}\varepsilon+b_{k+3}\varepsilon^{1+\frac{1}{k+1}}+\cdots+ b_{2k+1}\varepsilon^{1+\frac{k-1}{k+1}}+b_{2k+2}\varepsilon^{1+\frac{k}{k+1}}+\ldots.\nonumber
\end{align}
The formal length is unique up to a term $K\varepsilon$, $K\in\mathbb R$. Here, $b_{i}\in\mathbb{R}$ are real numbers depending only on the coefficients of $\widehat f$ and not on the initial condition. It belongs to $\widehat{\mathcal L}_1$, but the term $\varepsilon\log\varepsilon$ is \emph{the only} logarithmic term, all other terms are just powers. Note that in this analytic case the formal series \eqref{eq:asyy} is not really transfinite, but just a formal series indexed by $\omega$. By Theorem~B $(ii)$, \eqref{eq:asyy} is  the complete asymptotic expansion in $\widehat{\mathfrak L}$ of $\varepsilon\mapsto A_f^c(x_0,\varepsilon)$ for any orbit, up to an additive term $K \varepsilon$, $K\in\mathbb R$. 

On the other hand, by Lemma~\ref{lem:aepsi} applied in the regular case, we conclude that 
$$
A_f(x_0,\varepsilon)-A_f^c(x_0,\varepsilon)=\varepsilon^{1+\frac{k}{k+1}}k(\varepsilon),
$$
where $k(\varepsilon)=O(1)$ is \emph{high-amplitude oscillatory} and does not have an asymptotic behavior in $\widehat{\mathfrak L}$.
That is, the complete asymptotic expansion of the standard length of the $\varepsilon$-neighborhood of orbits $\varepsilon\mapsto A_f(x_0,\varepsilon)$ does not exist. It corresponds with the expansion \eqref{eq:asyy} of the continuous time length $\varepsilon\mapsto A_f^c(x_0,\varepsilon)$ up to the order $\varepsilon^{1+\frac{k-1}{k+1}}$. Its  remainder term $\varepsilon^{k+\frac{1}{k+1}} h(\varepsilon)=O(\varepsilon^{k+\frac{1}{k+1}})$ does not admit an asymptotic behavior in a power-logarithm scale:
\begin{align*}
A_{f}(x_{0},\varepsilon)&=\frac{k+1}{k}2^{\frac{1}{k+1}}\varepsilon^{\frac{1}{k+1}}+b_{2}\varepsilon^{\frac{2}{k+1}}+\cdots+b_{k}\varepsilon^{\frac{k}{k+1}}+\big(k+1-2\rho\big)\varepsilon\log\varepsilon+\\
 &\ +b_{k+2}\varepsilon+b_{k+3}\varepsilon^{1+\frac{1}{k+1}}+\cdots+ b_{2k+1}\varepsilon^{1+\frac{k-1}{k+1}}+\varepsilon^{1+\frac{k}{k+1}} h(\varepsilon),\ h(\varepsilon)=O(1).
\end{align*}
Note that this is a preciser statement of the result in \cite[Proposition 3]{nonlin}, where it was concluded that the asymptotic expansion exists up to the order $O(\varepsilon)$ and fails somewhere later.

A similar expansion of $A_{f}(x_{0},\varepsilon)$
was obtained in \cite{resman}, but for diffeomorphisms in $\mathbb{C}$.  In that case,
the area is computed instead of the length and the exponents are bigger
by $1$. 
\end{example}

\section{The formal inverse of a transseries\label{sec:inverse}}

We recall that $\widehat{\mathcal{L}}$ is the class of transseries
in $\widehat{\mathcal{L}}_{1}$ which involve only integer powers
of the variable $ $$\boldsymbol{\ell}$. Let $\widehat{g}\in\widehat{\mathcal{L}}$,
\begin{equation}\label{uvod}
\widehat{g}=ax^{\alpha}\boldsymbol{\ell}^{m}+\mathrm{h.o.t.},\ a\in\mathbb{R},\ \alpha>0,\ m\in\mathbb{Z}.
\end{equation}
Let us define the set $\widetilde{\mathcal{R}_{\widehat{g}}}$ as
the sub-semigroup of $\mathbb{R}_{\geq0}\times\mathbb{Z}\times\mathbb{Z}$,
resp. of $\mathbb{R}_{\geq0}\times\mathbb{Z}$, generated by: 
\begin{enumerate}
\item $(\beta-\alpha,\ell-m,0)$ for $(\beta,\ell)\in\mathcal{S}(\widehat{g})\setminus\{(\alpha,m)\}$,
$(0,1,-1)$ and $(0,0,1)$, if $m\neq0$, 
\item $(\beta-\alpha,\ell)$ for $(\beta,\ell)\in\mathcal{S}(\widehat{g})\setminus\{(\alpha,0)\}$ and $(0,1)$, if $m=0.$ 
\end{enumerate}
Let $\mathcal{R}_{\widehat{g}}=\begin{cases}
\widetilde{\mathcal{R}_{\widehat{g}}}+(1,0,0), & m\neq0,\\
\widetilde{\mathcal{R}_{\widehat{g}}}+(1,0), & m=0.
\end{cases}$

\noindent Here, $\mathcal S(\widehat g)$ denotes the support of $\widehat g$ \cite{mrrz2}, that is, the set of all pairs $(\beta,\ell)\in \mathbb R_{> 0}\times\mathbb Z$ that appear as exponents of power-log monomials in $\widehat g$. 

Note that the sets $\mathcal{R}_{\widehat{g}}$ and $\widetilde{\mathcal{R}_{\widehat{g}}}$ are well-ordered by Neumann's lemma.

\begin{prop}[Inverse of a transseries from $\widehat{\mathcal{L}}$]\label{th1}\ Let
$\widehat{g}\in\widehat{\mathcal{L}}$ be as in \eqref{uvod}. Then its formal inverse $\widehat{g}^{-1}$
belongs to $\widehat{\mathcal{L}}_{2}$. If $\widehat g$ moreover contains
no logarithm in the leading term $($$m=0$$)$, then its formal inverse $\widehat{g}^{-1}$
belongs to $\widehat{\mathcal{L}}_{1}$. More precisely, 
\begin{enumerate}
\item $m\neq 0$
\[
\widehat{g}^{-1}(x)=(a\alpha^{m})^{-\frac{1}{\alpha}}(x\boldsymbol{\ell}^{-m})^{\frac{1}{\alpha}}\cdot\Big(1+\widehat{R}(x)\Big),
\]
where $\widehat{R}\in\widehat{\mathcal{L}}_{2}$ with $\mathrm{ord}(\widehat{R})\succ(0,0,0)$ 
and its support $\mathcal{S}(\widehat{R})$ is made of monomials of
the type 
$$ 
(x\boldsymbol{\ell}^{-m})^{\frac{\gamma}{\alpha}}\boldsymbol{\ell}^{r}\boldsymbol{\ell}_{2}^{s},\ \ (\gamma,r,s)\in\widetilde{\mathcal{R}_{\widehat{g}}}.
$$
\item $m=0$
\[
\widehat{g}^{-1}(x)=a^{-\frac{1}{\alpha}}x^{\frac{1}{\alpha}}\cdot\Big(1+\widehat{R}(x)\Big),
\]
where $\widehat{R}\in\widehat{\mathcal{L}}$ with $\mathrm{ord}(\widehat{R})\succ(0,0)$ 
and $\mathcal{S}(\widehat{R})$ is made of monomials of
the type 
$$
x^{\frac{\gamma}{\alpha}}\boldsymbol{\ell}^{r},\ (\gamma,r)\in\widetilde{\mathcal{R}_{\widehat{g}}}.
$$
\end{enumerate}
Here, $\widetilde{\mathcal{R}_{\widehat{g}}}$ is as defined above.
\end{prop}

Note that a similar theorem about the formal inverse in a more general
setting of transseries was proved in \cite{dries}. \\

\emph{Outline of the proof. }Write 
\begin{equation}
\widehat{g}=g_{\alpha,m}\circ\widehat{\varphi},\label{comp}
\end{equation}
where $g_{\alpha,m}(x)=ax^{\alpha}\boldsymbol{\ell}^{m}$ and $\widehat{\varphi}(x)=x+\mathrm{h.o.t.}$ If $\widehat{g}(x)=ax+\text{h.o.t.}$ with
$a\ne0$, we simply have $g_{1,0}=a\cdot id$. Therefore, 
\begin{equation}
\widehat{g}^{-1}=\widehat{\varphi}^{-1}\circ\widehat{g}_{\alpha,m}^{-1}.\label{inv}
\end{equation}
Thus, in order to compute the formal inverse of the initial transseries
$\widehat{g}$, one needs to compute the formal inverse of the monomial
$g_{\alpha,m}$ and the formal inverse of the parabolic transseries
$\widehat{\varphi}$. In the following three auxiliary lemmas, we
control the support in each step. The proof of Proposition~\ref{th1}
is finally given at the end of the section.

\begin{lem}\label{onne} Let $g_{\alpha,m}(x)=ax^{\alpha}\boldsymbol{\ell}^{m}$,
$a\in\mathbb{R},\ \alpha>0,\ m\in\mathbb{Z}$. Then: 
\begin{align*}
g_{\alpha,m}^{-1}(x)&=(a\alpha^{m})^{-\frac{1}{\alpha}}\cdot x^{\frac{1}{\alpha}}\,\boldsymbol{\ell}^{-\frac{m}{\alpha}}\Big(1+F\big(\boldsymbol{\ell}_2,\frac{\boldsymbol{\ell}}{\boldsymbol{\ell}_{2}}\big)\Big)\in\mathcal G_{AN},\\
\widehat g_{\alpha,m}^{-1}(x)&=(a\alpha^{m})^{-\frac{1}{\alpha}}\cdot x^{\frac{1}{\alpha}}\,\boldsymbol{\ell}^{-\frac{m}{\alpha}}\Big(1+\widehat F\big(\boldsymbol{\ell}_2,\frac{\boldsymbol{\ell}}{\boldsymbol{\ell}_{2}}\big)\Big)\in\widehat {\mathcal{L}}_{2}.
\end{align*}
Here, $F(t,s)$ is a germ of two variables analytic at $(0,0)$, $F(0,0)=0$, with Taylor expansion $\widehat F$.
 
In particular, if $m=0$, then \begin{align*}&g_{\alpha,0}^{-1}(x)=\widehat g_{\alpha,0}^{-1}(x)=a^{-\frac{1}{\alpha}}x^{\frac{1}{\alpha}}.\end{align*}
\end{lem}

\begin{proof} First, we estimate the leading term of $g_{\alpha,m}^{-1}$.
Put 
\begin{equation}
y=ax^{\alpha}\boldsymbol{\ell}(x)^{m}.\label{el}
\end{equation}
Applying the logarithm function to both sides of this equality
leads to: 
\[
\log y=\log a+\alpha\log x+m\log\boldsymbol{\ell}(x).
\]
It follows that $\log y\sim\alpha\log x$ when $y\rightarrow0$, and,
consequently, that $\boldsymbol{\ell}(x)\sim\alpha\boldsymbol{\ell}(y)$.
From \eqref{el}, 
\[
x(y)\sim(a\alpha^{m})^{-\frac{1}{\alpha}}y^{\frac{1}{\alpha}}\boldsymbol{\ell}(y)^{-\frac{m}{\alpha}},\ y\to0.
\]
Therefore, 
\begin{equation}
g_{\alpha,m}^{-1}(x)=(a\alpha^{m})^{-\frac{1}{\alpha}}x^{\frac{1}{\alpha}}\boldsymbol{\ell}^{-\frac{m}{\alpha}}\big(1+h(x)\big),\ h(x)=o(1),\ x\to0.\label{eq2}
\end{equation}
Putting \eqref{eq2} in the equation $g_{\alpha,m}(g_{\alpha,m}^{-1}(x))=x$,
after some simplification, we get 
\[
\big(1+h(x)\big)^{\alpha}\cdot\Big(\frac{1}{1+\log(a\alpha^{m})\cdot\boldsymbol{\ell}-m\frac{\boldsymbol{\ell}}{\boldsymbol{\ell}_{2}}-\alpha\boldsymbol{\ell}\log(1+h(x))}\Big)^{m}=1.
\]
By the \emph{analytic implicit function theorem}, 
\[
h(x)=F_1\Big(\boldsymbol{\ell},\frac{\boldsymbol{\ell}}{\boldsymbol{\ell}_{2}}\Big)=F\Big(\boldsymbol{\ell}_2,\frac{\boldsymbol{\ell}}{\boldsymbol{\ell}_{2}}\Big),\ F_1,\ F\text{ analytic germs at \ensuremath{(0,0)}.}
\]
Since $h(x)=o(1)$ and $\boldsymbol{\ell}_2,\ \frac{\boldsymbol{\ell}}{\boldsymbol{\ell}_{2}}\to0$,
as $x\to0$, we conclude that $F(0,0)=0$. 
\end{proof}

\begin{lem}\label{twwo} Let $\widehat{\varphi}$ be as defined in
\eqref{comp}. Then $\widehat{\varphi}$ is parabolic and $\widehat{\varphi}\in\widehat{\mathcal{L}}_{2}$.
In particular, if $m=0$, then $\widehat{\varphi}\in\widehat{\mathcal{L}}$.
Moreover, 
\[
\mathcal{S}(\widehat{\varphi}-\mathrm{id})\subseteq\mathcal{R}_{\widehat{g}}.
\]
\end{lem}

\begin{proof} \emph{The case $m\neq0$.} By \eqref{comp}, 
\[
\widehat{\varphi}=\widehat g_{\alpha,m}^{-1}\circ\widehat{g}.
\]
Let us write $\widehat{g}(x)=ax^{\alpha}\boldsymbol{\ell}^{m}\big(1+\widehat{T}(x)\big)$.
Then $\widehat{T}\in\widehat{\mathcal{L}}$, $\text{ord}(\widehat{T})\succ(0,0,0)$
and 
\[
\mathcal{S}(\widehat{T})=\Big\{\big(\beta-\alpha,\ell-m\big)\big|\big.\ (\beta,\ell)\in\mathcal{S}(\widehat g)\Big\}.
\]
We now compute 
\begin{align*}
\widehat{\varphi}(x) & =\widehat g_{\alpha,m}^{-1}\big(\widehat{g}(x)\big)=\widehat g_{\alpha,m}^{-1}\Big(ax^{\alpha}\boldsymbol{\ell}^{m}\big(1+\widehat{T}(x)\big)\Big)=\nonumber \\
 & =x\cdot\Big(1+F_{2}(\boldsymbol{\ell}_{2},\frac{\boldsymbol{\ell}}{\boldsymbol{\ell}_{2}},\widehat{T})\Big),
\end{align*}
where $F_{2}$ is an analytic germ of three variables vanishing at $0$.

In the computation we use the expansions: 
\begin{align*}
 & \boldsymbol{\ell}(\widehat{g}(x))=\frac{1}{\alpha}\boldsymbol{\ell}\cdot\Big(1+F_{3}\big(\boldsymbol{\ell}_2,\frac{\boldsymbol{\ell}}{\boldsymbol{\ell}_{2}},\widehat{T}\big)\Big),\\
 & \boldsymbol{\ell}_{2}(\widehat{g}(x))=\boldsymbol{\ell}_{2}\Big(1+F_{4}\big(\boldsymbol{\ell}_2,\frac{\boldsymbol{\ell}}{\boldsymbol{\ell}_{2}},\widehat{T}\big)\Big),
\end{align*}
where $F_3,\ F_4$ are analytic germs of three variables vanishing at $0$.

Obviously, $\mathcal{S}(\widehat{\varphi}-\mathrm{id})\subseteq\mathcal{R}_{\widehat{g}}$,
with $\mathcal{R}_{\widehat{g}}$ defined at the beginning of the
section, so it is well-ordered. Therefore, $\widehat{\varphi}\in\widehat{\mathcal{L}}_{2}$.
\medskip

\emph{The case $m=0$.} As the proof is similar and simpler, we
omit it. \end{proof}

\begin{lem}\label{three} \

1. Let $\widehat{\varphi}\in\widehat{\mathcal{L}}_{2}$
be parabolic. The formal inverse $\widehat{\varphi}^{-1}$ is parabolic and belongs
to $\widehat{\mathcal{L}}_{2}$. Moreover, let $R\subseteq\mathbb{R}_{\geq 0}\times\mathbb{Z}\times\mathbb{Z}$
be the semigroup generated by $(\beta-1,p,q)$ for $(\beta,p,q)\in\mathcal{S}(\widehat{\varphi}-\mathrm{id})$,
$(0,1,0)$ and $(0,1,1)$. Then $\mathcal{S}(\widehat{\varphi}^{-1}-\mathrm{id})-(1,0,0)\subseteq R$.
\medskip

2. Let $\widehat{\varphi}\in\widehat{\mathcal{L}}$
be parabolic. The formal inverse $\widehat{\varphi}^{-1}$ is parabolic and belongs
to $\widehat{\mathcal{L}}$. Moreover, let $R\subseteq\mathbb{R}_{\geq 0}\times\mathbb{Z}$
be the semigroup generated by $(\beta-1,p)$ for $(\beta,p)\in\mathcal{S}(\widehat{\varphi}-\mathrm{id})$ and $(0,1)$.  Then $\mathcal{S}(\widehat{\varphi}^{-1}-\mathrm{id})-(1,0)\subseteq R$.
\end{lem}

\begin{proof} We prove the statement $1$ for an element of $\widehat{\mathcal{L}}_{2}$.
The statement $2.$ for an element of $\widehat{\mathcal{L}}$ is handled
in the same way.

Let $\widehat{\varphi}\in\widehat{\mathcal{L}}_{2}$ be parabolic.
Put $\widehat{h}=\widehat{\varphi}-\mathrm{id}$. Then $\text{ord}(\widehat{h})\succ(1,0,0)$.
Let us define the \emph{Schröder operator} $\Phi_{\widehat{\varphi}}:\widehat{\mathcal{L}}_{2}\to\widehat{\mathcal{L}}_{2}$,
\begin{equation*}
\Phi_{\widehat{\varphi}}\cdot\widehat{f}=\widehat{f}\circ\widehat{\varphi},\ \widehat{f}\in\widehat{\mathcal{L}}_{2}.
\end{equation*}
Put $\Phi_{\widehat{\varphi}}=Id+H_{\widehat{\varphi}}$, $H_{\widehat{\varphi}}:\widehat{\mathcal{L}}_{2}\to\widehat{\mathcal{L}}_{2}$.
By formal Taylor expansion (see \cite{dries}, $\widehat{\varphi}$
parabolic): 
\begin{equation}
H_{\widehat{\varphi}}\cdot\widehat{f}=\Phi_{\widehat{\varphi}}\cdot\widehat{f}-\widehat{f}=\widehat{f}\circ\widehat{\varphi}-\widehat{f}=\widehat{f}'\widehat{h}+\frac{1}{2!}\widehat{f}''\widehat{h}^{2}+\cdots,\ \widehat{f}\in\widehat{\mathcal{L}}_{2}.\label{eq:supp}
\end{equation}
Since $\text{ord}(\widehat{h})\succ(1,0,0)$, the operator $H_{\widehat{\varphi}}$
is a \emph{small operator}, see \cite{dries} or \cite[Section 5.1]{mrrz2}
for definition and properties of small operators. Therefore, the inverse
operator $\Phi_{\widehat{\varphi}}^{-1}:\widehat{\mathcal{L}}_{2}\to\widehat{\mathcal{L}}_{2}$
is well-defined by the series (formally convergent in the product topology with respect to the discrete topology, see \cite{mrrz2}) : 
\begin{equation}
\Phi_{\widehat{\varphi}}^{-1}=(\mathrm{Id}+H_{\widehat{\varphi}})^{-1}:=\sum_{k=0}^{\infty}(-1)^{k}H_{\widehat{\varphi}}^{k}\label{eq:svei}
\end{equation}
Consequently, $\widehat{\varphi}^{-1}=\Phi_{\widehat{\varphi}}^{-1}\cdot\mathrm{id}\in\widehat{\mathcal{L}}_{2}$.

We analyze now the support of the formal inverse $\widehat{\varphi}^{-1}=\Phi_{\widehat{\varphi}}^{-1}\cdot\mathrm{id}$
more precisely. Differentiating a monomial from the support of $\widehat{f}\in\widehat{\mathcal L}_2$, we
get: 
\[
(x^{\gamma}\boldsymbol{\ell}^{m}\boldsymbol{\ell}_{2}^{n})'=\gamma x^{\gamma-1}\boldsymbol{\ell}^{m}\boldsymbol{\ell}_{2}^{n}+mx^{\gamma-1}\boldsymbol{\ell}^{m+1}\boldsymbol{\ell}_{2}^{n}+nx^{\gamma-1}\boldsymbol{\ell}^{m+1}\boldsymbol{\ell}_{2}^{n+1},\ \gamma>0,\ m,\ n\in\mathbb{Z}.
\]
It is then easy to deduce from \eqref{eq:supp} that: 
\begin{equation}
\mathcal{S}(H_{\widehat{\varphi}}\cdot\widehat{f})\subseteq R+\mathcal{S}(\widehat{f}),\ f\in\widehat{\mathcal{L}}_{2},\label{son}
\end{equation}
where $R\subseteq\mathbb{R}_{\geq 0}\times\mathbb{Z}\times\mathbb{Z}$
is a sub-semigroup generated by $(\beta-1,p,q)$ for $(\beta,p,q)\in\mathcal{S}(\widehat{h})$
and $(0,1,0)$ and $(0,1,1)$. Iterating \eqref{son} we get that
$\mathcal{S}(H_{\widehat{\varphi}}^{n}\cdot\widehat{f})\subseteq R+\mathcal{S}(\widehat{f}),\ n\in\mathbb{N}$.
Consequently, by \eqref{eq:svei}, 
\[
\mathcal{S}\Big(\Phi_{\widehat{\varphi}}^{-1}\cdot\widehat{f}-\widehat{f}\Big)\subseteq R+\mathcal{S}(\widehat{f}),\ \widehat{f}\in\widehat{\mathcal{L}}_{2}.
\]
Since $\widehat{\varphi}^{-1}=\Phi_{\widehat{\varphi}}^{-1}\cdot\mathrm{id}$,
we get that $\mathcal{S}(\widehat{\varphi}^{-1}-\mathrm{id})-(1,0,0)\subseteq R$.
\end{proof}

\begin{cory}\label{cor:four} Let $\widehat{\varphi}\in\widehat{\mathcal{L}}_{2}$
$($resp. $\widehat{\mathcal{L}}$, if $m=0$$)$ be as in \eqref{comp}.
Let $\mathcal{R}_{\widehat{g}}$ be as defined at the beginning of
the section. Then $\widehat{\varphi}^{-1}\in\widehat{\mathcal{L}}_{2}$ $($resp.
$\widehat{\mathcal{L}}$, if $m=0$$)$ and 
\[
\mathcal{S}(\widehat{\varphi}^{-1}-\mathrm{id})\subseteq\mathcal{R}_{\widehat{g}}.
\]
\end{cory}

\begin{proof} By \eqref{comp}, $\widehat{\varphi}$ is parabolic.
By Lemma~\ref{twwo}, $\mathcal{S}(\widehat{\varphi}-\mathrm{id})\subseteq\mathcal{R}_{\widehat{g}}$.
By Lemma~\ref{three}, $\mathcal{S}(\widehat{\varphi}^{-1}-\mathrm{id})-(1,0,0)\subseteq \widetilde{\mathcal{R}_{\widehat{g}}}$. \end{proof}
\begin{proof}[Proof of Proposition \ref{th1}]
\emph{}

\emph{The case $m\neq0$.} By \eqref{inv}, we have that 
\[
\widehat{g}^{-1}=\widehat{\varphi}^{-1}\circ \widehat g_{\alpha,m}^{-1}.
\]
By Lemma~\ref{onne} and Corollary~\ref{cor:four}, $\widehat{g}^{-1}\in\widehat{\mathcal{L}}_{2}$.
Let us analyze the support. 
We have, using Lemma~\ref{onne}: 
\begin{align}
 \widehat{g}^{-1}(x)&=\widehat g_{\alpha,m}^{-1}(x)+\big(\widehat{\varphi}^{-1}-\mathrm{id}\big)\big(\widehat g_{\alpha,m}^{-1}(x)\big)=\label{formi}\\
 &=(a\alpha^{m})^{-\frac{1}{\alpha}}\big(x\boldsymbol{\ell}^{-m}\big)^{\frac{1}{\alpha}}\big(1+\widehat F(\boldsymbol{\ell}_2,\frac{\boldsymbol{\ell}}{\boldsymbol{\ell}_{2}})\big)+\nonumber\\
&\qquad\qquad\quad +\big(\widehat{\varphi}^{-1}-\mathrm{id}\big)\Big((a\alpha^{m})^{-\frac{1}{\alpha}}\big(x\boldsymbol{\ell}^{-m}\big)^{\frac{1}{\alpha}}\big(1+\widehat F(\boldsymbol{\ell}_2,\frac{\boldsymbol{\ell}}{\boldsymbol{\ell}_{2}})\big)\Big).\nonumber 
\end{align}
Let $x^{\gamma}\boldsymbol{\ell}^{r}\boldsymbol{\ell}_{2}^{s}\in\mathcal{S}(\widehat{\varphi}^{-1}-\mathrm{id})$.
By Corollary~\ref{cor:four}, $(\gamma,r,s)\in\mathcal{R}_{\widehat{g}}$.
We compute: 
\begin{align}
(x^{\gamma}\boldsymbol{\ell}^{r}\boldsymbol{\ell}_{2}^{s})\circ g_{\alpha,m}^{-1}(x)=(a\alpha^{m})^{-\frac{\gamma}{\alpha}}\alpha^{r}\cdot\big(x\boldsymbol{\ell}^{-m}\big)^{\frac{\gamma}{\alpha}}\cdot\boldsymbol{\ell}^{r}\boldsymbol{\ell}_{2}^{s}\cdot\Big(1+\widehat F_{5}\big(\boldsymbol{\ell}_{2},\frac{\boldsymbol{\ell}}{\boldsymbol{\ell}_{2}}\big)\Big),\label{monom}
\end{align}
where $\widehat F_{5}$ is the Taylor expansion of an analytic germ $F_5$ vanishing at the origin.
In the computation we use the following: 
\begin{align}\label{ellappl}
 & \boldsymbol{\ell}(\widehat g_{\alpha,m}^{-1})=\alpha\cdot\boldsymbol{\ell}\cdot\Big(1+\widehat F_{6}\big(\boldsymbol{\ell}_2,\frac{\boldsymbol{\ell}}{\boldsymbol{\ell}_{2}}\big)\Big),\nonumber\\
 & \boldsymbol{\ell}_{2}(\widehat g_{\alpha,m}^{-1})=\boldsymbol{\ell}_{2}\cdot\Big(1+\widehat F_{7}\big(\boldsymbol{\ell}_2,\frac{\boldsymbol{\ell}}{\boldsymbol{\ell}_{2}}\big)\Big),
\end{align}
where $\widehat F_{6},\ \widehat F_{7}$ are Taylor expansions of analytic germs $F_6,\ F_7$ vanishing at the origin.
Combining \eqref{formi} and \eqref{monom}, we get the statement in the case $m\neq 0$.
\medskip

\emph{The case $m=0$.} Lemmas~\ref{onne} - \ref{three} are simpler
in this case. The proof is a simple exercise. 
\end{proof}

\subsection{The inverse of a Dulac series and of a Dulac germ.}\label{sub:Dulac}

\noindent Let $g\in \mathcal G_{AN}$ be a Dulac germ and $\widehat g\in\widehat{\mathcal L}$ its Dulac series. Put
$$
\widehat g(x)=x^\alpha P_m(\boldsymbol\ell^{-1})+o(x^\alpha), 
$$
where $P_m$ is a polynomial of degree $m\in\mathbb N_0$. Let $\widetilde {\mathcal A}\subset \mathbb R_{\geq 0}$ be a sub-semigroup generated by $\{\beta-\alpha:\ \beta\in\mathcal S_x(\widehat g)\}$. Here, $\mathcal S_x$ denotes the support of $\widehat g$ with respect to powers of $x$ only. Put  
\begin{equation}\label{eq:aa}
\mathcal A=\{\gamma+1:\ \gamma\in\widetilde{\mathcal A}\}\subset \mathbb R_{>0}.
\end{equation}
Note that $\widetilde{\mathcal A}$ and thus also $\mathcal A$ are countable, of order type $\omega$ (finitely generated), or finite.
\smallskip

We compute in Proposition~\ref{prop:refin} the formal inverse $\widehat g^{-1}$ of a \emph{Dulac
series} $\widehat{g}\in\widehat{\mathcal{L}}$. We refine the statement
of Proposition~\ref{th1} in this special case, using the polynomial
form of the coefficient functions in Dulac series. 

We show further in Proposition~\ref{lem:haa} that the formal inverse $\widehat g^{-1}$ is the sectional asymptotic expansion with respect to integral sections of the inverse $g^{-1}$ of the Dulac germ $g$. The results will
be used in the proof of Theorem~B.

\begin{prop}[A refinement of Proposition~\ref{th1} for Dulac series]\label{prop:refin}
Let $\widehat{g}\in\widehat{\mathcal{L}}$, $\widehat g(x)=ax^\alpha\boldsymbol\ell^{-m}+\textrm{h.o.t.}$, $\alpha>0,\ m\in\mathbb N_0$, be a Dulac series. 
\smallskip

1. $m\neq 0$. Then the formal inverse
$\widehat{g}^{-1}$ belongs to $\widehat{\mathcal{L}}_{2}$. More precisely,
\begin{align}
\widehat g^{-1}(x)=\sum_{i=1}^{\infty} &\widehat f_{\beta_i} (\boldsymbol\ell) x^{\frac{\beta_i}{\alpha}}, \ \beta_i\in\mathcal A, \text{ with convergent coefficients}\nonumber\\
&\widehat{f}_{\beta_i}(y)=y^{-M_{\beta_i}+\frac{m\beta_i}{\alpha}}\widehat G_{\beta_i}\Big(\boldsymbol{\ell}(y),\frac{y}{\boldsymbol{\ell}(y)}\Big)\in \widehat{\mathcal{L}}^\infty,\ M_{\beta_i}\in\mathbb{N}_0.\label{Dulaccoef}
\end{align}
Here, $\widehat G_{\beta_i}$ are Taylor expansions of analytic germs $G_{\beta_i}$ of two variables. The leading term of $\widehat g^{-1}(x)$ does not contain $\boldsymbol\ell_2$ $($the double logarithm$)$.
\smallskip

2. $m=0$. Then
$\widehat{g}^{-1}\in\widehat{\mathcal{L}}$ is a \emph{Dulac series}.  In particular, the coefficients
$\widehat{f}_{\beta_i}(y)\in\widehat{\mathcal{L}}_{0}^{\infty}$ 
are \emph{polynomial} in $y^{-1}$.
\end{prop}

\begin{proof} Denote by $g_\alpha(x):=x^{\alpha}P_m(\boldsymbol{\ell}^{-1})$, $P_m$ a polynomial of degree $m\in\mathbb N_0$, the \emph{leading block} of $g$. Put 
\begin{equation}\label{eq:new}
\widehat{g}=g_{\alpha}\circ  \widehat \varphi.
\end{equation}

\emph{1. The case $m\neq 0$.}
Similarly as in the proof of Lemma~\ref{onne}, we get the following:
\begin{align}
g_{\alpha}^{-1}(x)&=(a\alpha^{-m})^{-\frac{1}{\alpha}}\cdot x^{\frac{1}{\alpha}}\,\boldsymbol{\ell}^{\frac{m}{\alpha}}\Big(1+F\big(\boldsymbol{\ell}_2,\frac{\boldsymbol{\ell}}{\boldsymbol{\ell}_{2}}\big)\Big)\in\mathcal G_{AN},\nonumber\\
\widehat g_{\alpha}^{-1}(x)&=(a\alpha^{-m})^{-\frac{1}{\alpha}}\cdot x^{\frac{1}{\alpha}}\,\boldsymbol{\ell}^{\frac{m}{\alpha}}\Big(1+\widehat F\big(\boldsymbol{\ell}_2,\frac{\boldsymbol{\ell}}{\boldsymbol{\ell}_{2}}\big)\Big)\in\widehat {\mathcal{L}}_{2}.\label{eq:galpha}
\end{align}
Here $F(t,s)$ is a germ of two variables analytic at $(0,0)$, $F(0,0)=0$, with Taylor expansion $\widehat F$.
Since $\widehat g_\alpha^{-1}$ is of the same form as $\widehat g_{\alpha,m}^{-1}$ from Lemma~\ref{onne}, by the proof of Lemma~\ref{twwo} we immediately get that $\widehat{\varphi}:=\widehat g_\alpha^{-1}\circ \widehat g$ belongs to $\widehat{\mathcal{L}}_{2}$ and is
parabolic. Moreover, it follows that:
\begin{equation}\label{phij}
\widehat{\varphi}(x)=x\cdot\Big(1+\widehat F_{1}(\boldsymbol{\ell}_{2},\frac{\boldsymbol{\ell}}{\boldsymbol{\ell}_{2}},\widehat{T})\Big).
\end{equation}
Here, $\widehat F_{1}$ is the Taylor expansion of an analytic germ of three variables vanishing at $0$. The improvement of Lemma~\ref{twwo} in the Dulac case is the special form
of $\widehat{T}$ in \eqref{phij}. Since $\widehat{g}(x)=ax^{\alpha}\boldsymbol{\ell}^{-m}\cdot(1+\widehat{T}(x))$, $\widehat T(x)=o(1)$,
is a Dulac series, we get: 
\begin{align}
\widehat{T}(x) & =\boldsymbol{\ell}^{m}P_{0}(\boldsymbol{\ell}^{-1})+\sum_{i=1}^{\infty}x^{\alpha_{i}-\alpha}\boldsymbol{\ell}^{m}P_{i}(\boldsymbol{\ell}^{-1}).\label{oh}
\end{align}
Here, $P_0$ is a polynomial of degree strictly smaller than $m$, $\left(P_{i}\right)_{i\ge1}$ is a sequence of polynomials, and
$\alpha_{i}>\alpha$, $i\in\mathbb N$. 

We now show that $\widehat \varphi:=\widehat g_\alpha^{-1}\circ\widehat g$ is \emph{strictly} parabolic: that is, the order of $\widehat\varphi-\text{id}$ in $x$ is \emph{strictly} bigger than $1$. Indeed, suppose the contrary, that: 
$$
\widehat\varphi (x)=x+cx\boldsymbol\ell^k+\text{h.o.t.},
$$
where $c\neq 0$ and $k\in \mathbb N$. Computing the formal composition $\widehat g_\alpha (x+cx\boldsymbol\ell^k+\text{h.o.t.})$, we obtain
\begin{align*}
\widehat g_\alpha (x+cx\boldsymbol\ell^k+&\text{h.o.t.})=(x+cx\boldsymbol\ell^k+\text{h.o.t.})^\alpha P_m\Big(-\log\big(x+cx\boldsymbol\ell^k+\text{h.o.t.}\big)\Big)=\\
&=x^\alpha (1+c\boldsymbol\ell^k+o(\boldsymbol\ell^k))^\alpha\cdot P_m\Big(\boldsymbol\ell^{-1}+c\boldsymbol\ell^k+o(\boldsymbol\ell^k)\Big)=\\
&=x^\alpha  (1+c\boldsymbol\ell^k+o(\boldsymbol\ell^k))^\alpha \Big(P_m(\boldsymbol\ell^{-1})+P_m'(\boldsymbol\ell^{-1})(c\boldsymbol\ell^k+\text{h.o.t.})+...\Big)=\\
&=x^\alpha P_m(\boldsymbol\ell^{-1})+acx^\alpha\boldsymbol\ell^{-m+k} +\text{h.o.t.}
\end{align*}
Since by \eqref{eq:new} this composition should equal $\widehat g(x)$, which is the sum of $x^\alpha P_m(\boldsymbol\ell^{-1})$ and the terms of strictly higher order than $x^\alpha$ in $x$, it necessarily follows that $c=0.$

We now invert formally $\widehat{\varphi}\in\widehat{\mathcal L}_2$, which is \emph{strictly} parabolic. 
By Lemma~\ref{three} and Corollary~\ref{cor:four}, $\widehat\varphi^{-1}\in\widehat{\mathcal L}_2$, and the support in $x$ follows from Corollary~\ref{cor:four}. To estimate the precise form of series in $\boldsymbol\ell,\ \boldsymbol\ell_2$ multiplying  any power of $x$ in $\widehat{\varphi}^{-1}$ in the Dulac case,
we use the Neumann inverse series formula \eqref{eq:svei}, as in the proof
of Lemma~\ref{three}: 
\begin{align}\label{eq:how}
\widehat{\varphi}^{-1}=&\,(\mathrm{Id}+H_{\widehat{\varphi}})^{-1}\cdot\mathrm{id}:=\sum_{k=0}^{\infty}(-1)^{k}H_{\widehat{\varphi}}^{k}\cdot\mathrm{id}=\nonumber\\
=&\,\textrm{id}-\widehat h+\Big(\widehat h\circ \widehat\varphi -\widehat h\Big)+\Big(\big(\widehat h\circ \widehat\varphi -\widehat h\big)\circ\widehat\varphi-\big(\widehat h\circ \widehat\varphi -\widehat h\big)\Big)+\ldots
\end{align}
Here, $H_{\widehat{\varphi}}\cdot\widehat{f}=\widehat{f}\circ\widehat{\varphi}-\widehat{f},\ \widehat{f}\in\widehat{\mathcal{L}}_{2}$, and $\widehat h:=\widehat \varphi-\mathrm{id}$.
Due to the \emph{strict} parabolicity of $\widehat\varphi$, the above series converges in the \emph{formal topology}. 

It can be seen from \eqref{phij} and \eqref{oh} that the coefficients in front of any power of $x$ in $\widehat\varphi(x)$, that is, in $\widehat h(x)$, are  finite sums of terms of the form:
\begin{equation}\label{eq:koef}
P(\boldsymbol\ell^{-1}) \widehat G(\boldsymbol\ell_2,\frac{\boldsymbol\ell}{\boldsymbol\ell_2}),
\end{equation}
where $P$ is a polynomial and $\widehat G$ a Taylor expansion of an analytic germ $G$ of two variables. Consequently, by \eqref{eq:how}, computing formal compositions with $\widehat\varphi$ strictly parabolic, we conclude that the \emph{coefficient} of
a fixed power $x^{\beta}$ in $\widehat{\varphi}^{-1}-\mathrm{id}$
is of the same form: 
$$
\sum_{i=1}^{n_{\beta}}P_{i}(\boldsymbol{\ell}^{-1})\cdot \widehat F_{i}(\boldsymbol{\ell}_{2},\frac{\boldsymbol{\ell}}{\boldsymbol{\ell}_{2}}),\ \beta\geq1,\ n_{\beta}\in\mathbb{N},
$$
Here, $P_i$ are polynomials and $\widehat F_i$ are the Taylor expansions of analytic germs $F_i$ of two variables.
Putting $M_\beta:=\max_{j=1\ldots n_\beta}\deg(P_j)$, the coefficient can be written in a reduced form:
\begin{equation}\label{help2}
\boldsymbol\ell^{-M_\beta}\cdot \widehat G_\beta(\boldsymbol{\ell}_{2},\frac{\boldsymbol{\ell}}{\boldsymbol{\ell}_{2}}),\ M_\beta\in\mathbb N_0.
\end{equation}
Here, $\widehat G_\beta$ is the Taylor expansion of an analytic germ $G_\beta$ of two variables.

Finally, $$\widehat{g}^{-1}=\widehat{\varphi}^{-1}\circ \widehat g_{\alpha} ^{-1}=\widehat g_{\alpha} ^{-1}+(\widehat{\varphi}^{-1}-\mathrm{id})\circ \widehat g_\alpha^{-1}.$$ Note that, since $\widehat\varphi$ and thus $\widehat\varphi^{-1}$ are strictly parabolic, the leading block of  $\widehat g^{-1}$ is exactly $\widehat g_\alpha^{-1}$. The support in $x$ of $\widehat g^{-1}-\widehat g_\alpha^{-1}$ is obtained from the more general Proposition~\ref{th1}. We analyze the blocks in the particular Dulac case. Using \eqref{ellappl},\ \eqref{eq:galpha} and \eqref{help2}, we conclude
that $\widehat{g}^{-1}-\widehat g_\alpha^{-1}$ contains blocks of the type: 
\[
\boldsymbol{\ell}^{-M_\beta}\widehat H_{\beta}\big(\boldsymbol{\ell}_2,\frac{\boldsymbol{\ell}}{\boldsymbol{\ell}_{2}}\big)\cdot\big(x\boldsymbol{\ell}^{m}\big)^{\frac{\beta}{\alpha}}=\boldsymbol{\ell}^{-M_\beta+\frac{m\beta}{\alpha}}\widehat H_{\beta}\big(\boldsymbol{\ell}_2,\frac{\boldsymbol{\ell}}{\boldsymbol{\ell}_{2}}\big)\cdot x^{\frac{\beta}{\alpha}},\ \beta\in \mathcal A,\ \ M_\beta\in\mathbb N_0.
\]
Here, $\widehat H_\beta$ are Taylor expansions of  analytic germs $H_\beta$ of two variables. The exponents $\beta$ belong to the support of $x$ in $\widehat \varphi^{-1}-\mathrm{id}$, which is by Corollary~\ref{cor:four} described by $\mathcal A$ in \eqref{eq:aa}. 

Since the first block of $\widehat g^{-1}$ is $\widehat g_\alpha^{-1}$  given in \eqref{eq:galpha}, $\widehat g^{-1}$ does not contain the double logarithm in the leading term. 
\medskip

\emph{2. The case $m=0$.} We follow the same steps, but the computation is easier. In this case, $g_\alpha(x)=ax^\alpha,\ a\in\mathbb R$, and $\widehat\varphi$ is strictly parabolic and Dulac. It follows that $\widehat\varphi^{-1}$  and consequently $\widehat g^{-1}=\widehat \varphi^{-1}\circ \widehat g_\alpha^{-1}$ are also Dulac.
\end{proof}

\begin{prop}\label{lem:haa} Let $g\in \mathcal G_{AN}$ be a Dulac germ and $\widehat g\in \widehat{\mathcal L}$ its Dulac expansion. Then the formal inverse $\widehat g^{-1}\in\widehat{\mathcal L}_2$ from \eqref{Dulaccoef} is the sectional asymptotic expansion of  the inverse germ $g^{-1}\in\mathcal G_{AN}$ with respect to any integral section. That is,
$$
\widehat{g^{-1}}=\widehat g^{-1}.
$$
More precisely, if $m\neq 0$, for every $n\in\mathbb N$, 
\begin{align}\label{eq:gmoins}
g^{-1}(x)-\sum_{i=1}^{n} & f_{\beta_i} (\boldsymbol\ell) x^{\frac{\beta_i}\alpha}=o(x^{\frac{\beta_n}\alpha}), \ x\to 0, \nonumber\\
&{f}_{\beta_i}(y)=y^{-M_{\beta_i}+\frac{m\beta_i}{\alpha}} G_{\beta_i}\Big(\boldsymbol{\ell}(y),\frac{y}{\boldsymbol{\ell}(y)}\Big)\in \mathcal G_{AN},\ M_{\beta_i}\in\mathbb{N}_0,
\end{align}
where $\beta_i\in\mathcal A$, $i\in\mathbb N$, are as in \eqref{Dulaccoef} and $G_{\beta_i}$ are analytic counterparts of $\widehat G_{\beta_i}$ from \eqref{Dulaccoef}.
\end{prop}

\noindent The proof is in the Appendix.

\section{Proof of Theorem~A}

\label{sec:proof_theorem_A}
\begin{proof}[Proof of Theorem~A]
 Let $\widehat{g}=\mathrm{id}-\widehat{f}$. The formal continuous
time length of $\varepsilon$-neighborhoods of orbits for $\widehat{f}$
is given by the formula: 
\begin{equation}\label{eq:flaforA}
\widehat{A}_{\widehat{f}}^{c}(\varepsilon)=\widehat{g}^{-1}(2\varepsilon)+2\varepsilon\cdot\widehat{\Psi}\big(\widehat{g}^{-1}(2\varepsilon)\big).
\end{equation}
By Proposition~\ref{th1}, we have that $\widehat{g}^{-1}\in\widehat{\mathcal{L}}_{2}$ and does not contain a double logarithm in the leading term.

By \cite[Proposition~4.3]{MRRZ2Fatou},
the formal Fatou coordinate $\widehat{\Psi}$ of $\widehat{f}$ exists in $\widehat{\mathcal{L}}_{2}^{\infty}$
and is unique (up to an additive constant) in $\widehat{\mathfrak{L}}$.
Furthermore, by \cite[Remark~7.1]{MRRZ2Fatou} there exists $\rho\in\mathbb{R}$
such that 
\[
\widehat{\Psi}-\rho\boldsymbol\ell_2^{-1}\in\widehat{\mathcal{L}}^{\infty}.
\]
By formal composition it now easily follows that $\widehat{\Psi}\circ\widehat{g}^{-1}\in\widehat{\mathcal{L}}_{2}^{\infty}$. The conclusion of  Theorem~A follows. 
\end{proof}
\medskip

The following remark provides more precise information on $\widehat{A}_{\widehat{f}}^{c}(\varepsilon)$:
\begin{obs} \label{obs:normal} Let $\widehat f\in\widehat{\mathcal L}$ parabolic, $\widehat g=\mathrm{id}-\widehat f$. Let $$\widehat g(x)=ax^\alpha\boldsymbol\ell^m+\mathrm{h.o.t.},\ a\neq 0,\ (\alpha,m)\succ (1,0),$$
where $\succ$ denotes the lexicographical order on $\mathbb{R}^2$.\\ 

(1) (The residual term of $\widehat{A}_{\widehat{f}}^{c}(\varepsilon)$) If $\widehat{g}$ does not contain a logarithm
in the leading term, then by Proposition~\ref{th1} $\widehat g^{-1}\in\widehat{\mathcal L}$, without logarithm in the leading term. By \cite[Remark~7.1]{MRRZ2Fatou}, there exists $\rho\in\mathbb R$ such that $\widehat\Psi-\rho\boldsymbol\ell_2^{-1}\in \widehat{\mathcal L}^\infty$. Consequently, by \eqref{eq:flaforA},
$$\widehat{A}_{\widehat{f}}^{c}(\varepsilon)-\rho\cdot2 \varepsilon\boldsymbol{\ell}_{2}(\varepsilon)^{-1}\in\widehat{\mathcal{L}}.$$ 
We will call $\rho\cdot2 \varepsilon\boldsymbol{\ell}_{2}(\varepsilon)^{-1}$ \emph{the residual term} of $\widehat{A}_{\widehat{f}}^{c}(\varepsilon)$. See Example~\ref{dubl} in Section~\ref{sec:examples}.

Notice that, in the general case where $\widehat g$ contains logarithm in the leading term, $\widehat g^{-1}\in\widehat{\mathcal L}_2$ and there may be other terms with double logarithms in $\widehat{A}_{\widehat{f}}^{c}(\varepsilon)$. See Section~\ref{sec:ptC} for a detailed analysis.
\smallskip

(2) (The leading term of $\widehat{A}_{\widehat{f}}^{c}(\varepsilon)$) Inverting formally as in Proposition~\ref{th1}, we get 
\begin{equation*}
\widehat{g}^{-1}(2\varepsilon)=(2/a)^{1/\alpha}\alpha^{-m/\alpha}\varepsilon^{1/\alpha}\boldsymbol{\ell}(\varepsilon)^{-m/\alpha}+\mathrm{h.o.t.}
\end{equation*}
By \cite[Remark~7.1]{MRRZ2Fatou} (or directly computing the first term of the Fatou coordinate $\widehat \Psi$ from the Abel equation):
\begin{equation*}
\widehat{\Psi}(x)=\begin{cases}-\frac{1}{a}\frac{1}{1-\alpha}x^{-\alpha+1}\boldsymbol{\ell}^{-m}+\text{h.o.t.},& \alpha\neq 1,\\
\frac{1}{a(m+1)}\boldsymbol\ell^{-m-1}+\mathrm{h.o.t.},& \alpha=1,\ m\in\mathbb N.
\end{cases}
\end{equation*}
 Therefore, by \eqref{eq:flaforA},
$$
\widehat{A}_{\widehat{f}}^{c}(\varepsilon)=\begin{cases}c \varepsilon^{1/\alpha}\boldsymbol\ell(\varepsilon)^{-m/\alpha}+\mathrm{h.o.t.},\ c\in\mathbb R,&\ \alpha\neq 1,\\
\frac{1}{a(m+1)}\varepsilon\boldsymbol\ell(\varepsilon)^{-m-1}+\mathrm{h.o.t.},&\ \alpha=1,\ m\in\mathbb N.\end{cases}
$$
\end{obs}

\section{Proof of Theorem~B}\label{sec:proof_theorem_B}

\noindent We first state and prove Definition~\ref{def:vari} and Lemmas~\ref{lem:aepsi}--\ref{lema:triB} used in the proof of Theorem~B.

\begin{defi}\label{def:vari} \

$(1)$ We say that a transseries $\widehat f\in\widehat{\mathcal L}^\infty$ is a \emph{Laurent transseries in $\widehat {\mathcal L}^\infty$} if its terms can be regrouped in the form:
\begin{equation}\label{def:gL}
\widehat f(y):= y^{\gamma}\widehat F\Big(\boldsymbol\ell(y),\frac{y}{\boldsymbol\ell(y)}\Big),
\end{equation}
where $\gamma\in\mathbb R$ and $\widehat F$ is the Taylor expansion of a germ $F$ of two variables analytic at $(0,0)$.
\smallskip

$(2)$ We say that a transseries $\widehat f\in\widehat{\mathcal L}^\infty$ is a \emph{generalized Laurent transseries in $\widehat {\mathcal L}^\infty$}, if it can be written as a \emph{finite} sum of Laurent transseries.
\end{defi}

Note that generalized Laurent transseries are \emph{convergent} in the sense of \cite[Definition~3.7]{MRRZ2Fatou}. The sum of Laurent transseries $\widehat f$ from \eqref{def:gL} is thus unique: 
$$f(y):=y^\gamma F\Big(\boldsymbol\ell(y),\frac{y}{\boldsymbol\ell(y)}\Big)\in\mathcal G_{AN}.$$ The sum of a generalized Laurent transseries is analogously the sum of its Laurent summands, thus unique. 

Note further that all derivatives of generalized Laurent transseries are again generalized Laurent, therefore \emph{convergent}. Moreover, the sums and the derivatives commute.

\begin{lem}\label{lem:aepsi}
Let $f(x)=x-ax^\alpha\boldsymbol\ell^m+o(x^\alpha\boldsymbol\ell^m)$, $\alpha>1,\ m\in\mathbb N_0,\ a>0$, be a Dulac germ. There exists $\delta>0$ such that $A_f^c(\varepsilon,x_0)-A_f(\varepsilon,x_0)=O(\varepsilon^{1+\delta})$. More precisely, it holds that $$A_f^c(\varepsilon,x_0)-A_f(\varepsilon,x_0)=\varepsilon^{2-\frac{1}{\alpha}}\boldsymbol\ell^{\frac{m}{\alpha}}k(\varepsilon),$$
where $k(\varepsilon)=O(1)$ is a high-amplitude oscillatory function $($and consequently, does not admit an asymptotic behavior, as $\varepsilon\to 0)$. That is, there exist two sequences $(\varepsilon_n^1)\to 0$ and $(\varepsilon_n^2)\to 0$ and $A,\ B\in\mathbb R,\ A<B,$ such that $h(\varepsilon_n^1)<A<B<h(\varepsilon_n^2),\ n\in\mathbb N$.   
\end{lem} 
                    
Note that Lemma~\ref{lem:aepsi} implies that $\varepsilon\mapsto A_f(\varepsilon,x_0)$ indeed does not have a complete asymptotic expansion in a power-logarithm scale. By Lemma~\ref{lem:aepsi}, since $\varepsilon\mapsto A_f^c(\varepsilon,x_0)$ has a complete asymptotic expansion in the power-log scale, the asymptotic expansion of $\varepsilon\mapsto A_f(x_0,\varepsilon)$ still exists at all power-logarithmic terms of order \emph{strictly smaller} than $\varepsilon^{2-\frac{1}{\alpha}}$ in $\varepsilon$. 

\begin{proof}
We compute:
\begin{align}\label{eq:gere}
A_f^c(\varepsilon,x_0)-&A_f(\varepsilon,x_0)=x_{\tau_\varepsilon}-x_{n_\varepsilon}+2\varepsilon (\tau_\varepsilon-n_\varepsilon)=\nonumber\\
&=\Psi^{-1}(\tau_\varepsilon+\Psi(x_0))-\Psi^{-1}(n_\varepsilon+\Psi(x_0))+2\varepsilon(\tau_\varepsilon-n_\varepsilon)=\nonumber\\
&=-\Big((\Psi^{-1})'\big(\tau_\varepsilon+\Psi(x_0)\big)+\frac{1}{2}(\Psi^{-1})''(\eta(\varepsilon))(n_\varepsilon-\tau_\varepsilon)+2\varepsilon\Big)(n_\varepsilon-\tau_\varepsilon).
\end{align}
Here, $\eta(\varepsilon)$ lies between $\tau_\varepsilon+\Psi(x_0)$ and $n_\varepsilon+\Psi(x_0)$. Note that
$n_\varepsilon=\lceil \tau_\varepsilon \rceil$, so $n_\varepsilon-\tau_\varepsilon=O(1)$. Additionaly, let $(\varepsilon_n)\to 0$ be the critical sequence: $\varepsilon_n:=\frac{x_n-x_{n+1}}{2}$, $n\in\mathbb N$. Denote $h(\varepsilon):=n_\varepsilon-\tau_\varepsilon$. It holds that $h$ is continuous on $(\varepsilon_{n+1},\varepsilon_n)$ and $h(\varepsilon_n-)=h(\varepsilon_n)=1, \ h(\varepsilon_n+)=0$, $n\in\mathbb N$. This shows that $h(\varepsilon)=O(1)$ does not have an asymptotic behavior in any continuous scale and does not tend to $0$. So $O(1)$ is indeed the optimal estimate. 

Let $X=\xi\frac{d}{dx}$ be the vector field whose time-one map is $f$. Then $\Psi'=\frac{1}{\xi}$ and $f=\text{Exp}(\xi\frac{d}{dx}).\mathrm{id}=\mathrm{id}+\xi+\frac{1}{2}\xi'\xi+\ldots$, for details see \cite[Section 4]{MRRZ2Fatou}. Since $\alpha>1$, the orders of the terms in the expansion are strictly increasing. Therefore,
$$
\xi(x)=-g(x)-\frac{1}{2}\xi'(x)\xi(x)+\ldots
$$
Thus, $\xi(x)\sim -ax^\alpha\boldsymbol\ell^m$, $\xi'(x)\sim -a\alpha x^{\alpha-1}\boldsymbol\ell^m$, as $x\to 0$. By Proposition~\ref{th1} $$g^{-1}(2\varepsilon)\sim (a\alpha^m)^{-1/\alpha}2^{\frac{1}{\alpha}}\cdot \varepsilon^{1/\alpha} \boldsymbol\ell(\varepsilon)^{-\frac{m}{\alpha}},\ \varepsilon\to 0.$$

\noindent We now estimate the leading terms in \eqref{eq:gere}:
\begin{align}\label{eq:firstone}
(\Psi^{-1})'(\tau_\varepsilon+\Psi(x_0))&=\frac{1}{\Psi'(\Psi^{-1}(\tau_\varepsilon+\Psi(x_0)))}=\xi(g^{-1}(2\varepsilon))=\nonumber\\
&=-g(g^{-1}(2\varepsilon))-\frac{\xi(g^{-1}(2\varepsilon))\xi'(g^{-1}(2\varepsilon))}{2}+\ldots=\nonumber\\
&=-2\varepsilon-a^{1/\alpha}\alpha^{1+m/\alpha}2^{1-1/\alpha}\varepsilon^{2-\frac{1}{\alpha}}\boldsymbol\ell^{\frac{m}{\alpha}}+o(\varepsilon^{2-\frac{1}{\alpha}}\boldsymbol\ell^{\frac{m}{\alpha}}).
\end{align}

\noindent Also, it can easily be computed that:
\begin{align}\label{eq:firsttwo}
(\Psi^{-1})''(\eta(\varepsilon))&\sim \xi(g^{-1}(2\varepsilon))\xi'(g^{-1}(2\varepsilon))\nonumber\\
&=a^{1/\alpha}\alpha^{1+m/\alpha}2^{2-1/\alpha}\varepsilon^{2-\frac{1}{\alpha}}\boldsymbol\ell^{\frac{m}{\alpha}}+o(\varepsilon^{2-\frac{1}{\alpha}}\boldsymbol\ell^{\frac{m}{\alpha}}).
\end{align}

\noindent Finally, putting \eqref{eq:firstone} and \eqref{eq:firsttwo} into \eqref{eq:gere}, we get:  
$$
A_f(\varepsilon,x_0)-A_f^c(\varepsilon,x_0)=a^{1/\alpha}\alpha^{1+m/\alpha}2^{1-1/\alpha}\varepsilon^{2-\frac{1}{\alpha}}\boldsymbol\ell^{\frac{m}{\alpha}}\cdot\big(h(\varepsilon)-h^2(\varepsilon)+o(1)\big),\ \varepsilon\to 0.
$$
Put $k(\varepsilon):=a^{1/\alpha}\alpha^{1+m/\alpha}2^{1-1/\alpha}\big(h(\varepsilon)-h^2(\varepsilon)\big)$. Obviously, $k(\varepsilon)=O(1)$ (since $h$ is) and is high-amplitude oscillatory by the properties of $h$. Indeed, there exists a sequence $(\varepsilon_n^1)\to 0$ such that $h(\varepsilon_n)=\frac{1}{2}$. Then $k(\varepsilon_n^1)=a^{1/\alpha}\alpha^{1+m/\alpha}2^{-1-1/\alpha}>0$ and $k(\varepsilon_n)=0$, $n\in\mathbb N$.
\end{proof}   

\begin{lem}\label{lema:prvaB} Let $f\in\mathcal G_{AN}$ be a Dulac germ and let $\widehat f$ be its Dulac expansion. Then the formal continuous length of $\varepsilon$-neighborhoods of orbits  $\widehat A_{\widehat f}^c(\varepsilon)$ belongs to $\widehat{\mathcal L}_2$. It can moreover be written in the form:
\begin{equation}\label{eq:ahah}
\widehat{A}_{\widehat{f}}^{c}(\varepsilon)=\sum_{j\in\mathbb N} \widehat F_j(\boldsymbol\ell(\varepsilon))\varepsilon^{\beta_j},
\end{equation}
where $\beta_j>0,\ j\in\mathbb N$, form a strictly increasing sequence tending to $+\infty$ or finite, and $\widehat F_j\in\widehat{\mathcal L}_1^I\subset \widehat{\mathcal L}_1^\infty$ are \emph{integrally summable} in the sense of Definition~\ref{diifi}.
\end{lem}      

\begin{proof} The fact that  $\widehat A_{\widehat f}^c(\varepsilon)\in\widehat{\mathcal L}_2$ has already been proven in Theorem~A in more generality (for all parabolic transseries in $\widehat{\mathcal L}$). Therefore it can be written in the form \eqref{eq:ahah}, where $\widehat {F}_j\in \widehat{\mathcal{L}}_1^\infty$. 

The only non-trivial thing to prove is that $\widehat {F}_j$ belong to $\widehat{\mathcal{L}}_1^I$, that is, they are integrally summable. Recall that $$
\widehat{A}_{\widehat{f}}^{c}(\varepsilon)=\widehat{g}^{-1}(2\varepsilon)+2\varepsilon\cdot\widehat{\Psi}\big(\widehat{g}^{-1}(2\varepsilon)\big).$$

\noindent We have shown in Subsection~\ref{sub:Dulac} in Proposition~\ref{prop:refin} that 
\begin{equation}\label{eq:gprv}
\widehat g^{-1}(x)=\sum_{i\in\mathbb N} \widehat g_i(\boldsymbol\ell)x^{\beta_i},\ \beta_i>0,\end{equation} with $\widehat g_i\in\widehat {\mathcal L}^\infty$  \emph{Laurent} (as in Definition~\ref{def:gL}), and that the first coefficient $\widehat g_1(\boldsymbol\ell)$ is without double logarithm in the leading term. Put 
\begin{equation}\label{decc}
\widehat\Psi_1:=\widehat\Psi-\rho\boldsymbol\ell_2^{-1}.
\end{equation}It follows from \cite[Proposition 6.4]{MRRZ2Fatou} that \begin{equation}\label{eq:psijedan}\widehat\Psi_1(x)=\sum_{j\in\mathbb N} \widehat{f_j}(\boldsymbol\ell)x^{\alpha_j},\ \alpha_j\in\mathbb R,\ \widehat f_j\in\widehat{\mathcal L}_0^{I}.\end{equation} By formal composition, 
\begin{equation}\label{eq:finaly}
\widehat\Psi_1\circ \widehat{g}^{-1}(2\varepsilon)=\sum_{j\in\mathbb N}  \widehat h_j\big(\boldsymbol\ell(\varepsilon)\big) \varepsilon^{\gamma_j}\in\widehat{\mathcal L}_2^\infty,\ \gamma_j\in\mathbb R.
\end{equation}
We prove now that $\widehat h_j\in\widehat{\mathcal L}_1^I$ by analyzing the form of these \emph{coefficients}. 
\smallskip


By Proposition~\ref{prop:refin}, $\widehat g^{-1}(2\varepsilon)=\varepsilon^{\beta_1}\cdot \widehat g_1\big(\boldsymbol\ell(\varepsilon)\big)+O(\varepsilon^{\beta_1+\delta})$, $\beta_1>0$, $\delta>0$. Here, $\widehat g_1\in\widehat{\mathcal L}^\infty$ is a Laurent transseries with no double logarithm in the first term:
\begin{equation}\label{gfirst}
\widehat g_1(\boldsymbol\ell)=a\boldsymbol\ell^\gamma (1+\widehat R(\boldsymbol\ell_2,\frac{\boldsymbol\ell}{\boldsymbol\ell_2})),\ \gamma\in\mathbb R,\ 
\end{equation}
where $\widehat R$ is the Taylor expansion of an analytic germ $R$ at $(0,0)$, vanishing at $0$.
Note that $O(\varepsilon^{\beta_1+\delta})\in\widehat {\mathcal L}_2$ with Laurent coefficients in $\widehat{\mathcal L}^\infty$. 

Put $$\widehat H_i(x):=\widehat f_i(\boldsymbol\ell)x^{\alpha_i},\ i\in\mathbb N.$$ 
Then $\widehat H_i'(x)=\widehat R_i(\boldsymbol\ell)x^{\alpha_i-1}$, with $\widehat R_i$ convergent in $\widehat{\mathcal L}_0^\infty$. Also, all other (formal) derivatives are of the form:
\begin{equation}\label{eqi}
\widehat H_i^{(k)}(x)=\widehat R_{i,k}(\boldsymbol\ell)x^{\alpha_i-k}, \ k\geq 2,\ \text{$\widehat R_{i,k}$ convergent in $\widehat{\mathcal L}_0^{\infty}$.}\end{equation}

Now, \begin{equation}\label{jedan}\widehat\Psi_1\big(\widehat g^{-1}(2\varepsilon)\big)=\widehat H_1\big(\widehat g^{-1}(2\varepsilon)\big)+\widehat H_2\big(\widehat g^{-1}(2\varepsilon)\big)+\widehat H_3\big(\widehat g^{-1}(2\varepsilon)\big)+\ldots\end{equation}


By the Taylor expansion,
\begin{align}\label{dva}
&\widehat H_i\big(\widehat g^{-1}(2\varepsilon)\big)=\widehat H_i\Big(\varepsilon^{\beta_1} \widehat g_1(\boldsymbol\ell(\varepsilon))\cdot(1+O(\varepsilon^{\delta}))\Big)=\\
&=\widehat H_i\big(\varepsilon^{\beta_1} \widehat g_1(\boldsymbol\ell(\varepsilon))\big)+\widehat H_i'\big(\varepsilon^{\beta_1} \widehat g_1(\boldsymbol\ell(\varepsilon))\big)\cdot O(\varepsilon^{\delta+\beta_1})+\frac{1}{2!}\widehat H_i''\big(\varepsilon^{\beta_1} \widehat g_1(\boldsymbol\ell(\varepsilon))\big)\cdot O(\varepsilon^{2(\delta+\beta_1)})+\ldots\nonumber
\end{align}

We have:
\begin{align}\label{tri}
&\widehat H_i\big(\varepsilon^{\beta_1} \widehat g_1(\boldsymbol\ell(\varepsilon))\big)=\widehat f_i\big(\boldsymbol\ell(\varepsilon^{\beta_1} \widehat g_1(\boldsymbol\ell(\varepsilon)))\big)\big(\varepsilon^{\beta_1} \widehat g_1(\boldsymbol\ell(\varepsilon))\big)^{\alpha_i},\\
&\widehat H_i^{(k)}\big(\varepsilon^{\beta_1} \widehat g_1(\boldsymbol\ell(\varepsilon))\big)=\widehat R_{i,k}\big(\boldsymbol\ell(\varepsilon^{\beta_1} \widehat g_1(\boldsymbol\ell(\varepsilon)))\big)\big(\varepsilon^{\beta_1} \widehat g_1(\boldsymbol\ell(\varepsilon))\big)^{\alpha_i-k},\ k\in\mathbb N,\nonumber
\end{align}
$\widehat R_{i,k}$ convergent. Note, using \eqref{gfirst}, that $\widehat g_1(\boldsymbol\ell)^{\alpha_i-k}$ is a Laurent transseries in $\widehat{\mathcal L}^\infty$. It is easy to see by \eqref{gfirst} that $\widehat R_{i,k}\big(\boldsymbol\ell(x^{\beta_1} \widehat g_1(\boldsymbol\ell))\big)$ is also a Laurent transseries in $\widehat{\mathcal L}^\infty$.
\smallskip

Since by \eqref{decc} $\widehat\Psi=\widehat\Psi_1+\rho\boldsymbol\ell_2^{-1}$, it is left to analyze the formal composition $\boldsymbol\ell_2^{-1}\big(\widehat g^{-1}(2\varepsilon)\big)$. By the formal Taylor expansion, we have:
\begin{align}\label{eq:eldva}
\boldsymbol\ell_2^{-1}(g^{-1}(2\varepsilon))=\boldsymbol\ell_2^{-1}\big(\varepsilon^{\beta_1} \widehat g_1(\boldsymbol\ell(\varepsilon))\big)&+\big(\boldsymbol\ell_2^{-1}\big)'\big(\varepsilon^{\beta_1} \widehat g_1(\boldsymbol\ell(\varepsilon))\big)\cdot O(\varepsilon^{\beta_1+\delta})+\nonumber\\
&+\frac{1}{2}\big(\boldsymbol\ell_2^{-1}\big)''\big(\varepsilon^{\beta_1} \widehat g_1(\boldsymbol\ell(\varepsilon))\big)\cdot O(\varepsilon^{2{\beta_1}+2\delta})+\ldots
\end{align}
Using $(\boldsymbol\ell_2^{-1})^{(k)}(x)=x^{-k}P_k(\boldsymbol\ell)$, where $P_k$ is a polynomial of degree $k$, $k\in\mathbb N$, we get:
\begin{align}\label{eq:eldva1}
\boldsymbol\ell_2^{-1}\big(\varepsilon^{\beta_1} \widehat g_1(\boldsymbol\ell(\varepsilon))\big)&= \boldsymbol\ell_2^{-1}(\varepsilon)+\widehat R_1\Big(\boldsymbol\ell(\varepsilon),\frac{\boldsymbol\ell(\varepsilon)}{\boldsymbol\ell_2(\varepsilon)}\Big)\\
(\boldsymbol\ell_2^{-1})^{(k)}\big(\varepsilon^{\beta_1} \widehat g_1(\boldsymbol\ell(\varepsilon))\big)&=P_k\big(\boldsymbol\ell(\varepsilon^{\beta_1} \widehat g_1(\boldsymbol\ell(\varepsilon)))\big)\cdot\big(\varepsilon^{\beta_1} \widehat g_1(\boldsymbol\ell(\varepsilon))\big)^{-k},\ k\in\mathbb N.\nonumber
\end{align}
Here, $\widehat R_1$ is a Taylor expansion of an analytic germ of two variables at $(0,0)$. Note that $\boldsymbol\ell_2^{-1}\big(x^\beta \widehat g_1(\boldsymbol\ell)\big)$ is generalized Laurent transseries in $\widehat{\mathcal L}^\infty$.
As before, $\widehat{g}_1(\boldsymbol\ell)^{-k}$ and $P_k\big(\boldsymbol\ell(\varepsilon^{\beta_1} \widehat g_1(\boldsymbol\ell(\varepsilon)))\big)$, $k\in\mathbb N$, are Laurent transseries in $\widehat{\mathcal L}^\infty$. 

\smallskip
Finally, putting \eqref{tri} in \eqref{dva} and then in \eqref{jedan}, and \eqref{eq:eldva1} in \eqref{eq:eldva}, and grouping the terms in $\varepsilon\widehat\Psi\big(\widehat g^{-1}(\varepsilon)\big)$ with the same power of $\varepsilon$, we get the summands in $\widehat A_{\widehat f}^c\in\widehat{\mathcal L}_2$ of two possible types, with respect to the power of $\varepsilon$:
\begin{equation}\label{joj}
\begin{cases}
\varepsilon^{1+\beta_1\alpha_i}\Big(\widehat g_1(\boldsymbol\ell(\varepsilon))^{\alpha_i}\cdot \widehat f_i\big(\boldsymbol\ell(\varepsilon^{\beta_1} \widehat g_1(\boldsymbol\ell(\varepsilon)))\big)+\widehat H(\boldsymbol\ell(\varepsilon))\Big),&\text{for $\alpha_i\in\mathcal{S}_x(\widehat\Psi),$}\\
\varepsilon^{1+\gamma} \widehat G(\boldsymbol\ell(\varepsilon)),&\text{if $\gamma\neq \beta_1\alpha_i$, $\forall\alpha_i\in\mathcal{S}_x(\widehat\Psi)$,}
\end{cases}
\end{equation}
where $\widehat G,\ \widehat H\in\widehat{\mathcal L}^\infty$ are generalized Laurent transseries, and $\alpha_i$ is the exponent of integration of $\widehat f_i\in\widehat{\mathcal L}_0^I$.

The \emph{coefficients} $\widehat F_j$ from \eqref{eq:ahah} obviously belong to $\widehat{\mathcal L}_1^I$ as defined in Definition~\ref{diifi} and one decomposition \eqref{dvaa} is given in \eqref{joj}. 
\end{proof}                

\begin{lem}\label{lema:dvaB} Let $f$ be a Dulac map, $\widehat f$ its Dulac expansion and let the formal length $\widehat A_{\widehat f}^c(\varepsilon)$ be as in \eqref{eq:ahah}. Then the continuous time length of the $\varepsilon$-neighborhood of an orbit $\varepsilon\mapsto A^c_f(x_0,\varepsilon)$ satisfies, up to an additive term $\varepsilon K$, $K\in\mathbb R$, the following asymptotics in powers of $\varepsilon$:
\begin{equation}\label{eq:choicesum}
A^c_f(x_0,\varepsilon)=\sum_{j=1}^n F_j(\boldsymbol\ell(\varepsilon))\varepsilon^{\beta_j}+o(\varepsilon^{\beta_n}),\ n\in\mathbb N,\ \varepsilon\to 0,
\end{equation}
where $F_j\in\mathcal G_{AN}$ are (any) integral sums of $\widehat F_j$. That is, the formal length $\widehat{A}_{\widehat{f}}^{c}(\varepsilon)$ from \eqref{eq:ahah} is the sectional asymptotic expansion with respect to integral sections of $A^c_f(x_0,\varepsilon)$, up to $\varepsilon K$, as $\varepsilon\to 0$.
\end{lem}

\begin{proof}
We now show the asymptotics \eqref{eq:choicesum} of $\varepsilon\mapsto A^c_f(x_0,\varepsilon)$. The construction of $\Psi\in\mathcal G_{AN}$ in the Theorem in \cite{MRRZ2Fatou} follows the same term-by-term algorithm as the formal construction of $\widehat\Psi$. Let $\Psi_1=\Psi-\rho\boldsymbol\ell_2^{-1}$. We have, by the proof of the Theorem in \cite{MRRZ2Fatou}:
\begin{equation}\label{eq:st1}
\Psi_1(x)-\sum_{j=1}^{n} x^{\alpha_j}f_j(\boldsymbol\ell)=o(x^{\alpha_{n}}),\ n\in\mathbb N,
\end{equation} 
where $f_j\in\mathcal G_{AN}$ is an integral sum (in the sense of Definition~\ref{defi}), unique up to a constant, of $\widehat f_j\in\widehat{\mathcal L}_0^I$ from \eqref{eq:psijedan}, $j\in\mathbb N$. The sequence $(\alpha_j)_j$ is strictly increasing to $+\infty$ or finite, the same as in \eqref{eq:psijedan}. 
Put $
H_i(x):=x^{\alpha_i}f_i(\boldsymbol\ell),\ i\in\mathbb N
$. Then $H_i\in\mathcal G_{AN}$, $H_i'(x)=x^{\alpha_i-1} R_i(\boldsymbol\ell)$, and $$H_i^{(k)}(x)=R_{i,k}(\boldsymbol\ell)x^{\alpha_i-k}, \ k\geq 2,$$
where $R$, $R_{i,k}$ are the sums of convergent Laurent series $\widehat R,\ \widehat R_{i,k}\in\widehat{\mathcal L}_0^\infty$ from the formal construction \eqref{eqi}.  

On the other hand, we have proved in Proposition~\ref{lem:haa} in Section~\ref{sec:inverse} the following expansion:
\begin{equation}\label{eq:st2}
g^{-1}(2\varepsilon)-\sum_{i=1}^{k}g_i\big(\boldsymbol\ell(\varepsilon)\big) \varepsilon^{\beta_i}=o(\varepsilon^{\beta_{k}}),\ k\in\mathbb N,
\end{equation}
where $g_i\in\mathcal G_{AN}$ are the sums of convergent Laurent transseries $\widehat g_i\in\widehat{\mathcal L}^\infty$ from \eqref{eq:gprv}.

As in the formal counterpart above, using \eqref{eq:st1} and \eqref{eq:st2}, we have: $$\Psi_1(g^{-1}(2\varepsilon))=\sum_{j=1}^{n} H_j(g^{-1}(\varepsilon))+o\big((g^{-1}(2\varepsilon))^{\alpha_{n}}\big),\ n\in\mathbb N.$$ As above, we repeat the Taylor expansion for $H_j(g^{-1}(\varepsilon))$, $j\in\mathbb N$ and $\boldsymbol\ell_2(g^{-1}(2\varepsilon))$. By the correspondence of the procedure with the formal one above, we finally get:
$$
\Psi(g^{-1}(2\varepsilon))-\sum_{j=1}^{n} h_j\big(\boldsymbol\ell(\varepsilon)\big) \varepsilon^{\gamma_j}=o(\varepsilon^{\gamma_{j}}),\ n\in\mathbb N.
$$
where $\gamma_j\in\mathbb R$ are strictly increasing to $+\infty$, as in \eqref{eq:finaly} and $h_j\in\mathcal G_{AN}$ are exactly the integral sums of $\widehat h_j\in\mathcal L_1^I$ from \eqref{eq:finaly}, as in Definition~\ref{def:is}. The asymptotics \eqref{eq:choicesum} now follows.
\end{proof}

\begin{lem}\label{lema:triB} Let $f$ be a Dulac map. The asymptotic expansions of the continuous time length of the $\varepsilon$-neighborhood of an orbit $\varepsilon\mapsto A^c_f(x_0,\varepsilon)$ with respect to different integral sections differ only by an additive term $K\varepsilon$, $K\in\mathbb R$.
\end{lem}

\begin{proof} We show that different choices of integral sums $F_j$ by Definition~\ref{def:is} correspond to different choices of the integral section $\mathbf s$. 
If $\widehat{A}_{\widehat{f}}^{c}(\varepsilon)$ is the sectional asymptotic expansion of $A_f^c(x_0,\varepsilon)$ with respect to an integral section $\mathbf s$ and if we choose a different integral section $\mathbf s_1$, then there exists $K\in\mathbb R$ such that $A^c_f(x_0,\varepsilon)+K\varepsilon$ admits $\widehat{A}_{\widehat{f}}^{c}(\varepsilon)$ as its sectional asymptotic expansion with respect to $\mathbf s_1$.

By the precise form of the coefficients $\widehat F_j\in\widehat{\mathcal L}_1^I$ given in \eqref{joj}, we see that their integral sums $F_j\in\mathcal G_{AN}$ are non-unique by the following terms:

1. In front of the power $\varepsilon^{\beta_1\alpha_i+1}$ for $\alpha_i\in\mathcal S_x(\widehat\varphi)$ and $\alpha_i<0$, the integral sum $F_j\in\mathcal G_{AN}$ is by Definition~\ref{diifi} and Remark~\ref{rem:helpdif} unique only up to an additive term:
$$
K\, \big(g_1(\boldsymbol\ell(\varepsilon))\big)^{\alpha_i}e^{\frac{\alpha_i}{\boldsymbol\ell(\varepsilon^\beta g_1(\boldsymbol\ell(\varepsilon)))}}=K\, \big(g_1(\boldsymbol\ell(\varepsilon))\big)^{\alpha_i} \cdot \big(\varepsilon^\beta g_1(\boldsymbol\ell(\varepsilon))\big)^{-\alpha_i}=K\varepsilon^{-\beta\alpha_i},\ c\in\mathbb R.
$$
For $\alpha_i\in\mathcal S_x(\widehat\varphi)$ and $\alpha_i\geq 0$, the integral sum $F_j$ in front of the power $\varepsilon^{\beta\alpha_i+1}$ is unique.

2. In front of all other powers of $\varepsilon$, $\widehat F_j\in\widehat{\mathcal L}_1^I$ are generalized Laurent, therefore \emph{convergent}, and have thus unique sums $F_j\in\mathcal G_{AN}$. 

Note that in the formal Fatou coordinate there are only finitely many $\alpha_i$ such that $\alpha_i<0$ (see the construction of the formal Fatou coordinate in the Theorem in \cite{MRRZ2Fatou}). Therefore, depending on the choice of the integral sums, that is, on the choice of the integral section $\mathbf s$, a fixed $\varepsilon\mapsto A^c_f(x_0,\varepsilon)$ admits a fixed $\widehat A_{\widehat f}^c$ as its sectional asymptotic expansion with respect to $\mathbf s$ only up to a term $K\varepsilon$, $K\in\mathbb R$.
\end{proof}
\smallskip

\begin{proof}[Proof of Theorem~B]\

\noindent $(i)$ Lemma~\ref{lema:prvaB}.
\smallskip

\noindent $(ii)$  Recall that $$
\widehat{A}_{\widehat{f}}^{c}(\varepsilon)=\widehat{g}^{-1}(2\varepsilon)+2\varepsilon\cdot\widehat{\Psi}\big(\widehat{g}^{-1}(2\varepsilon)\big).$$
Here, $\widehat\Psi\in\widehat{\mathcal L}_2^\infty$ is the formal Fatou coordinate for the Dulac expansion $\widehat f$ of $f$ constructed in \cite{MRRZ2Fatou}, unique in $\widehat{\mathfrak L}$ up to a constant term $K\in\mathbb R$. The formal inverse $\widehat g^{-1}$ constructed in Section~\ref{sec:inverse} is unique. Therefore, $\widehat{A}_{\widehat{f}}^{c}(\varepsilon)$ is uniquely defined up to $\varepsilon K$, $K\in\mathbb R$, depending on the choice of the constant in $\widehat\Psi$. The statement follows by Lemmas~\ref{lema:prvaB}--\ref{lema:triB}.

\smallskip

\noindent $(iii)$ Lemma~\ref{lem:aepsi}.
\hfill $\Box$

\section{Proof of Theorem~C}\label{sec:ptC}  

Let $f(x)=x-x^\alpha\boldsymbol\ell^m+o(x^\alpha\boldsymbol\ell^m),\ \alpha>1,\ m\in\mathbb N_0^-,$ be a prenormalized parabolic Dulac map. Let $\widehat f$ be its Dulac expansion. 

In the proof, by $[.]$, $$[x^\beta\boldsymbol\ell^r\boldsymbol\ell_2^s]\,\widehat h(x),$$ we will denote the coefficient in front of the monomial $x^\beta\boldsymbol\ell^r\boldsymbol\ell_2^s$ in $\widehat h\in\widehat{\mathcal L}_2^\infty$ (may be also $0$). 
\medskip 

\noindent In the proof, we need the following lemmas. 

\begin{lem}\label{lem:dva} Let $f(x)=x-x^\alpha\boldsymbol\ell^m+o(x^\alpha\boldsymbol\ell^m)$ be a (prenormalized) Dulac germ with the Dulac expansion $\widehat f$. Let $\widehat g=\mathrm{id}-\widehat f$. For a given $f$, let $A_{r,s}(f)$ denote the coefficients of monomials $y\boldsymbol\ell(y)^r\boldsymbol\ell_2(y)^s$ in $\widehat g^{-1}(y)$, for integer $r,\,s\in\mathbb Z$. That is,
\begin{equation}\label{eq:ars}A_{r,s}(f):=[\boldsymbol\ell^{r}\boldsymbol\ell_2^{s}]\Big(\frac{\widehat g^{-1}(t)}{t}\Big),\ r,\ s \in\mathbb Z.\end{equation}

\noindent Then the formal invariant $\rho$ of the Dulac germ $f$ can be expressed as the sum:                                                                        
\begin{equation}\label{eq:roo1}
\rho=\sum_{r=-\infty}^{1}\sum_{s=-\infty}^{-r+1} a_{rs}(r,s,\alpha,m)\cdot A_{r,s}(f).
\end{equation}
Here, for a given $f$, only finitely many of $A_{r,s}(f)$ are non-zero, so the sum is finite. The coefficients $a_{rs}\in\mathbb R$ are explicit universal functions of $\alpha,\ m$ and $r$ and $s$, independent of the germ $f$ other than through $\alpha$ and $m$.
\end{lem}

In the proof of Lemma~\ref{lem:dva}, we use Proposition~\ref{lem:jen}.
\begin{prop}\label{lem:jen}  Let $f(x)=x-x^\alpha\boldsymbol\ell^m+o(x^\alpha\boldsymbol\ell^m)$ be a (prenormalized) Dulac germ with the Dulac expansion $\widehat f$. Let $\widehat g=\mathrm{id}-\widehat f$. The formal invariant $\rho$ of $f$ can be expressed as a linear combination of coefficients of $\widehat g^{-1}$ in the following way:
\begin{equation}\label{eq:rooo}
\rho=\alpha[\boldsymbol\ell]\frac{\widehat g^{-1}(t)}{t}\big|_{t=x^\alpha\boldsymbol\ell^m}+m[1]\frac{\widehat g^{-1}(t)}{t}\big|_{t=x^\alpha\boldsymbol\ell^m}.
\end{equation}
\end{prop}

\noindent The proof is in the Appendix. Note that in the special case when $m=0$, we have simply
\begin{equation}\label{eq:sc}
\rho=\alpha[\boldsymbol\ell]\frac{\widehat g^{-1}(t)}{t}\big|_{t=x^\alpha}=[\boldsymbol\ell]\frac{\widehat g^{-1}(x)}{x}.
\end{equation}
\medskip

\noindent\emph{Proof of Lemma~\ref{lem:dva}.} By Proposition~\ref{lem:jen}, we analyze which monomials in $\frac{\widehat g^{-1}(t)}{t}$ contribute to the coefficients $[\boldsymbol\ell]$ and $[1]$ in $\frac{\widehat g^{-1}(t)}{t}\big|_{t=x^\alpha\boldsymbol\ell^m}$. Take a general monomial $Ct^\beta\boldsymbol\ell^r\boldsymbol\ell_2^s$, $\beta>-1$, $r,\,s\in\mathbb R$, $C\in\mathbb R$, from $\frac{\widehat  g^{-1}(t)}{t}\in\widehat{\mathcal L}_2$. 

If $\beta\neq 0$, after composing with $x^\alpha\boldsymbol\ell^m$ it can easily be seen that it contributes only to monomials of $\frac{\widehat g^{-1}(t)}{t}\big|_{t=x^\alpha\boldsymbol\ell^m}$ beginning with non-zero power of $x$, $x^{\beta\alpha}$, and therefore neither to $[\boldsymbol\ell]$ nor to $[1]$.

The monomials contributing to $[\boldsymbol\ell]$ or $[1]$ are therefore of the form $\boldsymbol\ell(t)^r\boldsymbol\ell_2(t)^s$, $r,\,s\in\mathbb R$. We compute:
\begin{align}\label{eq:ekspl}
\boldsymbol\ell(t)&^r\boldsymbol\ell_2(t)^s\Big|_{t=x^\alpha\boldsymbol\ell^m}=\nonumber\\
&=\alpha^{-r}r^{-s}\boldsymbol\ell^r\boldsymbol\ell_2^s \Big(\frac{1}{1+\frac{m}{\alpha}\frac{\boldsymbol\ell}{\boldsymbol\ell_2}}\Big)^r \Big(\frac{1}{1+\log\alpha\boldsymbol\ell_2+\boldsymbol\ell_2\log\big(1+\frac{m}{\alpha}\frac{\boldsymbol\ell}{\boldsymbol\ell_2}\big)}\Big)^s.
\end{align}
We now see that the monomials that contribute to $[\boldsymbol\ell]$ or $[1]$ necessarily have integer $r,\,s\in\mathbb Z$. Moreover, that the \emph{admissible} set of exponents (in the sense that no other exponents can contribute to $[\boldsymbol\ell]$ or $[1]$) is:
\begin{equation}\label{eq:admis}
(r,s)\in\Big\{(-k,s): s=k+1,\,k,\,k-1,\ldots,\ k\in \{-1\}\cup \{0\}\cup \mathbb N\Big\}\subseteq \mathbb Z\times\mathbb Z.
\end{equation}
This admissible set \emph{does not depend on the germ $f$}.
The set $\textrm{supp}(\frac{\widehat g^{-1}(t)}{t})$ is well-ordered, therefore only finitely many $r$ and $s$ exist in the support for a specific $f$, and for others the coefficient is equal to $0$. Finally, the coefficients that contribute to $[\boldsymbol\ell]$ and $[1]$ in $\frac{\widehat g^{-1}(t)}{t}\big|_{t=x^\alpha\boldsymbol\ell^m}$, that is, to $\rho$, are $$A_{r,s}(f)=[\boldsymbol\ell^{r}\boldsymbol\ell_2^{s}]\Big(\frac{\widehat g^{-1}(t)}{t}\Big),\ r=-\infty \ldots 1,\ s=-\infty\ldots (-r+1),$$ where, for a specific $f$, only finitely many of them are non-zero. By \eqref{eq:ekspl},
\begin{align}\label{eq:ali}
[\boldsymbol\ell]\big(\frac{g^{-1}(t)}{t}\Big|_{t=x^\alpha\boldsymbol\ell^m}\big)=\sum_{r=-\infty}^1\sum_{j=-\infty}^{-r+1} A_{r,s}(f) b_{rs}(r,s,\alpha,m),\\
[1]\big(\frac{g^{-1}(t)}{t}\Big|_{t=x^\alpha\boldsymbol\ell^r}\big)=\sum_{r=-\infty}^1\sum_{s=-\infty}^{-r+1} A_{r,s}(f) c_{rs}(r,s,\alpha,m).\nonumber
\end{align}
Here, $b_{rs}(r,s,\alpha,m)$ and $c_{rs}(r,s,\alpha,m)$ are universal functions only of $r,\ s,\ \alpha, \ m$ that can be computed explicitely, independent of $f$ other than through $\alpha$ and $m$, while $A_{r,s}(f)$ depend on the specific $f$ and for a specific $f$ only finitely many of them are non-zero (and thus the given sum is finite). Now putting \eqref{eq:ali} in \eqref{eq:rooo} we get \eqref{eq:roo1}, where $a_{rs}:=\alpha b_{rs}+m c_{rs}$. 
\end{proof}

\begin{obs} Note that the admissible set in \eqref{eq:admis} is estimated very coarsely. In fact, one can compute the optimal admissible set which is smaller. The optimality is in the sense that $\boldsymbol\ell^r\boldsymbol\ell_2^s\Big|_{x^\alpha\boldsymbol\ell^m}$ indeed contain monomials $\boldsymbol\ell$ or $1$, for $m\neq 0$.
$$
(r,s)\in \{(0,s):\,s\in\mathbb N_0^-\}\cup\{(1,s):\,s\in\mathbb N_0^-\}\cup \bigcup_{k\in\mathbb N} \{(-k,s), 1\leq s\leq k\}.
$$
For all other $(r,s)$ we have $a_{rs}(r,s,\alpha,m)=0$ in \eqref{eq:roo1}.
\end{obs}

\begin{lem}\label{lem:tri} Let $f$ be a (prenormalized) Dulac germ and let $\widehat A_{\widehat f}^c(\varepsilon)$ be its formal continuous length of $\varepsilon$-neighborhoods of orbits. For $p,\ q\in\mathbb Z$, $p^2+q^2\neq 0,$ let
$$B_{p,q}:=[\varepsilon\boldsymbol\ell^p\boldsymbol\ell^q]\widehat A_{\widehat f}^c\Big(\frac{\varepsilon}{2}\Big).$$
Then the admissible coefficients $A_{r,s}(f)$ from \eqref{eq:ars}, $r\leq 1$, $s\leq -r$, can be expressed $($by solving finitely many triangular systems$)$ as polynomial functions of $B_{p,q}$, $p,\ q\in\mathbb Z$, $(p,q)\prec (0,0)$, with coefficients independent of the germ $f$.  
\end{lem}
\begin{obs}
It is worth noticing that in the statement of Lemma~\ref{lem:tri} $B_{0,0}=[\varepsilon^1]\widehat A_{\widehat f}^c(\varepsilon)$ is not used for obtaining admissible $A_{r,s}(f)$ and then by Lemma~\ref{lem:dva} for expressing $\rho$ . This is coherent, since the coefficient of $\varepsilon^1$ in $\widehat A_{\widehat f}^c(\varepsilon)$ is not uniquely defined.
\end{obs}

\noindent \emph{Proof of Lemma~\ref{lem:tri}.} By Proposition~\ref{prop:izvod} in the Appendix, there exists $\delta>0$ such that:
\begin{equation}\label{eq:jaoj}
\widehat{A}^c_{\widehat f}(\frac{\varepsilon}{2})=-\varepsilon\int^\varepsilon \frac{\widehat g^{-1}(t)}{t^2}\,dt-\varepsilon\boldsymbol\ell(\varepsilon)^{-1}+o(\varepsilon^{1+\delta}),\ \varepsilon\to 0.
\end{equation}
By \eqref{eq:ars}, $A_{r,s}(f)=[\varepsilon \boldsymbol\ell(\varepsilon)^r\boldsymbol\ell_2(\varepsilon)^s]\widehat g^{-1}(\varepsilon)$, $r,\,s\in\mathbb Z,$ $r=-\infty\ldots 1,\ s=-\infty\ldots (-r+1)$. For a particular $f$, only finitely many of admissible $A_{r,s}(f)$ are non-zero, due to well-orderedness of $\widehat g^{-1}$. 

By \eqref{eq:jaoj}, to express $B_{p,q}=[\varepsilon\boldsymbol\ell^p\boldsymbol\ell^q]\widehat A_{\widehat f}^c\big(\frac{\varepsilon}{2}\big)$ by $A_{r,s}(f)=[\boldsymbol\ell^r\boldsymbol\ell_2^s]\frac{\widehat g^{-1}(t)}{t}$, $r,\, s\in\mathbb Z$, we express the coefficient 
$$
[\boldsymbol\ell^p \boldsymbol\ell_2^q]\Big(\int^x \frac{\frac{\widehat g^{-1}(t)}{t}}{t}\,dt\Big)
$$
by $A_{r,s}(f)$. First note that only coefficients of monomials of $\frac{\widehat g^{-1}(t)}{t}$ with zero power of $t$ and integer powers of $\boldsymbol\ell$, $\boldsymbol\ell_2$ contribute to $[\boldsymbol\ell^p \boldsymbol\ell_2^q]\Big(\int^x \frac{\frac{\widehat g^{-1}(t)}{t}}{t}\,dt\Big)$, $p,\,q\in\mathbb Z$. That is, exactly $A_{r,s}(f),\ r,\,s\in\mathbb Z$. To determine the contribution of each such monomial, we compute:
\begin{align*}
&\int^x \frac{A_{r,s}(f)\boldsymbol\ell^r\boldsymbol\ell_2^s}{t}=-A_{r,s}(f)\int ^{\boldsymbol\ell} u^{r-2}\boldsymbol\ell(u)^s \,du=\\
&\ =
\small{
\begin{cases}
-\frac{1}{r-1}A_{r,s}(f)\boldsymbol\ell^{r-1}\boldsymbol\ell_2^s+\frac{s}{(r-1)^2}A_{r,s}(f)\boldsymbol\ell^{r-1}\boldsymbol\ell_2^{s+1}-\frac{s(s+1)}{(r-1)^3}A_{r,s}(f)\boldsymbol\ell^{r-1}\boldsymbol\ell_2^{s+2}+\ldots,& r\neq 1,\\
-\frac{A_{1,s}(f)}{s-1}\boldsymbol\ell_2^{s-1},&r=1,\ s\neq 1,\\
-A_{1,1}(f)\log(\boldsymbol\ell_2), &r=1,\,s=1.
\end{cases}}
\end{align*}

\noindent Therefore,
\begin{align*}
&[\boldsymbol\ell^p \boldsymbol\ell_2^q]\int^x \frac{\frac{g^{-1}(t)}{t}}{t}\,dt=\\
&=\small{ \begin{cases}
-\frac{1}{p}A_{p+1,q}(f)+\frac{q-1}{p^2}A_{p+1,q-1}(f)-\frac{(q-1)(q-2)}{p^3}A_{p+1,q-2}(f)+\frac{(q-1)(q-2)(q-3)}{p^4}A_{p+1,q-3}(f)&+\ldots,\\
&p\neq 0,\\[0.2cm]
-\frac{1}{q}A_{1,q+1}(f),&\!\!\!\!\!\!\!\!\!\!\!\! p=0,\,q\neq 0.\\
\end{cases}}
\end{align*}
For $p=q=0$, the coefficient $[1]\int^x \frac{\frac{g^{-1}(t)}{t}}{t}\,dt$ is not well-defined (we may add any constant). Note also that the above sum is finite, since $\widehat g^{-1}$ is well-ordered: for every $p\in\mathbb Z$ there exists $q$ such that $A_{p+1,r}=0$, for all $r<q.$
\smallskip

\noindent Finally,
\begin{align}\label{eq:jaod}
&B_{p,q}=[\varepsilon\boldsymbol\ell(\varepsilon)^p\boldsymbol\ell_2(\varepsilon)^q]\widehat A^c_{\widehat f}\big(\frac{\varepsilon}{2}\big)=\nonumber\\
&=
\small{
\begin{cases}
\frac{1}{p}A_{p+1,q}(f)-\frac{q-1}{p^2}A_{p+1,q-1}(f)+\frac{(q-1)(q-2)}{p^3}A_{p+1,q-2}(f)-\frac{(q-1)(q-2)(q-3)}{p^4}A_{p+1,q-3}(f)+\ldots,\\
&\hspace{-2cm} p\neq -1,0,\\[0.2cm]
-A_{0,q}(f)-(q-1)A_{0,q-1}(f)-(q-1)(q-2)A_{0,q-2}(f)-(q-1)(q-2)(q-3)A_{0,q-3}(f)+\ldots,\\&\hspace{-2cm} p=-1,\ q\neq 0,\\
-1-A_{0,0}(f)+A_{0,-1}(f)-2A_{0,-2}(f)+3! A_{0,-3}(f)+\ldots,&\hspace{-2cm}p=-1,\ q= 0,\\[0.2cm]
\frac{1}{q}A_{1,q+1}(f),&\hspace{-2cm}p=0,\,q\neq 0.
\end{cases}}
\end{align}
\smallskip

The following algorithm expresses the admissible coefficients $A_{r,s}(f)$, $r,\,s\in\mathbb Z,$ $r\leq 1,\ s\leq (-r+1)$, from coefficients $B_{r,s}$, $(r,s)\prec (0,0)$, of the formal continuous length. It is based on solving a \emph{finite number of upper-triangular systems} which we derive from \eqref{eq:jaod}.

Note that for a given $f$, due to well-orderedness of $\widehat g^{-1}$, there exists $R\in\mathbb Z$, such that $A_{p,q}(f)=0$, for every $p\leq R$ and for every $q\in\mathbb Z$. Let $R$ be the biggest such, that is, there exists a $q\in\mathbb Z$ such that $A_{R+1,q}(f)\neq 0$. Also, for every $p\in\mathbb Z$, $p>R$, there exists $q\in\mathbb Z$ such that $A_{p,r}(f)=0$, for every $r<q$. 

We separately solve special levels $p=-1$ and $p=0$. Note that we do not need  to consider levels $p>0$, since $r\leq 1$ for the admissible set of $A_{r,s}(f)$.
\medskip

\noindent \emph{The algorithm.}

\emph{1. Level $p=0$ $($expressing admissible $A_{1,q}(f),\ q\leq 0$$)$.} 

\noindent Let $q_1$ be the lowest $q$ such that $A_{1,q_1}(f)$ does not vanish. From \eqref{eq:jaod}, we see that $B_{0,q}=0$ for $q< q_1-1$ and $B_{0,q_1-1}\neq 0$. The admissible coefficients that we express on this level are $A_{1,0}(f),\,A_{1,-1}(f),\ldots,A_{1,q_1}(f)$. In the asymptotic expansion of the length of the $\varepsilon$-neighborhood of an orbit, we read the coefficients 
$B_{0,-1},\,B_{0,-2}\,\ldots,\,B_{0,q_1-1}$. We start with $B_{0,-1}$. The stopping condition is the point $q_1-1$ which is recognized since there are no more lower non-zero coefficients $B_{0,q}$, $q<q_1-1$, on this level. Recall that $\widehat A_{\widehat f}^c$ is well-ordered, so this must eventually happen. By \eqref{eq:jaod}, we have: $$A_{1,q+1}(f)=q B_{0,q},\ q_1-1\leq q\leq -1.$$
Note that $B_{0,0}$ is not needed for expressing admissible $A_{1,q}(f)$, since $A_{1,1}(f)$ does not belong to the admissible set.
\smallskip

\emph{2. Levels $p<-1$ $($expressing admissible $A_{p+1,q}(f),\ q\leq -p$$)$.}
 
\noindent Let $p<-1$ be any level such that there exists $q\leq -p$ such that $B_{p,q}\neq 0$. Otherwise, this level needs not to be considered, since by \eqref{eq:jaod} all admissible $A_{p,q}(f)$ on this level are necessarily zero. Due to well-orderedness, there are only finitely many such levels. 

On a level $p<-1$, in the expansion of the length of the $\varepsilon$-neighborhood of an orbit read $B_{p,-p}$, $B_{p,-p-1},\ \ldots.$ Continue until the biggest $q_{p+1}$ such that all further coeficients on this level vanish, that is, $B_{p,q}=0$, $q<q_{p+1}$ (well-orderedeness). This is the stopping condition. The admissible coefficients that we need to express by $B_{p,q}$ are $A_{p+1,q}(f),\ q_{p+1}\leq q\leq -p$. Indeed, by \eqref{eq:jaod}, $B_{p,q}=0$ for $q<q_{p+1}$ implies that $A_{p+1,q}(f)=0$ for $q<q_{p+1}$. By \eqref{eq:jaod}, we solve the upper-triangular system:
$$
\scriptsize{
\begin{cases}
\begin{array}{lrrrrr}
B_{p,-p}=&\frac{1}{p}A_{p+1,-p}(f)&+\frac{p+1}{p^2}A_{p+1,-p-1}(f)&+\frac{(p+1)(p+2)}{p^3}A_{p+1,-p-2}(f)&+\ldots &+\frac{(p+1)\cdots(-q_{p+1})}{p^{-q_{p+1}-p+1}} A_{p+1,q_{p+1}}(f),\\
B_{p,-p-1}=&&\frac{1}{p}A_{p+1,-p-1}(f)&+\frac{p+2}{p^2}A_{p+1,-p-2}(f)&+\ldots& +\frac{(p+2)\cdots(-q_{p+1})}{p^{-q_{p+1}-p}} A_{p+1,q_{p+1}}(f),\\
B_{p,-p-2}=&&&\frac{1}{p}A_{p+1,-p-2}(f)&+\ldots&+\frac{(p+3)\cdots(-q_{p+1})}{p^{-q_{p+1}-p-1}} A_{p+1,q_{p+1}}(f),\\
&\vdots&&\vdots&&\\
B_{p,q_{p+1}}=&&&&&\frac{1}{p}A_{p+1,q_{p+1}}(f).
\end{array}
\end{cases}
}
$$
\emph{3. Level $p=-1$ $($expressing admissible $A_{0,q}(f),\ q\leq 1$$)$.}

\noindent By a similar analysis, we solve an upper-triangular system to express (finitely many due to well-orderedness) admissible $A_{0,q}(f),\ q\leq 1$ by finitely many $B_{-1,q}$, $q\leq 1$.

Finally, we see that all $B_{r,s}$ used for expressing the admissible $A_{p,q}$ satisfy $(r,s)\prec (0,0)$.
\hfill $\Box$
\bigskip

\noindent {\emph{Proof of Theorem~C.}} By Theorem~B $(iii)$, the initial part of the asymptotic expansion of the standard length $\varepsilon\mapsto A_f(x_0,\varepsilon)$ for any orbit coincides with the formal length $\widehat A_{\widehat f}^c (\varepsilon)$ up to the order $O(\varepsilon^{1+\delta})$, for some $\delta>0$. Therefore, we can work with the formal length and analyse its initial terms. The leading term of $\widehat A_{\widehat f}^c(\frac\varepsilon 2)$ is given by $c\varepsilon^{\frac{1}{\alpha}}\boldsymbol\ell^{-\frac{m}{\alpha}},\ c\in\mathbb R$, by Remark~\ref{obs:normal}. From the exponents, we read the formal invariants $\alpha,\,m$. For the invariant $\rho$, the statement of Theorem~C follows directly from Lemmas~\ref{lem:dva} and \ref{lem:tri}.

\hfill $\Box$

\medskip

\noindent \emph{Proof of Corollary~\ref{cor:nolog}}. By \cite[Remark 7.1]{MRRZ2Fatou}, the formal Fatou coordinate $\widehat{\Psi}\in\widehat{\mathcal L}_2^\infty$ of any Dulac germ $f(x)=x-x^\alpha\boldsymbol\ell^m+o(x^\alpha\boldsymbol\ell^m),\ \alpha>1,\ m\in\mathbb N_0^-$, with formal invariants $(\alpha,\rho)$, contains the double logarithm in \emph{only one} term: $\big(\frac{m}{2}+\rho\big)\boldsymbol\ell_2^{-1}$. On the other hand, if $\widehat g(x)=x^\alpha+o(x^\alpha)$ ($g$ does not contain a logarithm in the leading term), then it can be shown as in the proof of Proposition~\ref{th1} that its inverse $\widehat g^{-1}(\varepsilon)=\varepsilon^{\frac{1}{\alpha}}+o\big(\varepsilon^{\frac{1}{\alpha}}\big)$ belongs to $\widehat{\mathcal L}_1$, i.e., does not contain double logarithms. Since $\widehat g^{-1}(\varepsilon)$ contains no double logarithms and since its leading term does not contain a logarithm, composing $\widehat\Psi (\widehat g^{-1}(\varepsilon))$ we see that the term $\rho\boldsymbol\ell_2^{-1}$ is the only term in $\widehat\Psi (\widehat g^{-1}(\varepsilon))$ which contains the double logarithm. Consequently, $\rho\cdot \varepsilon\boldsymbol\ell_2^{-1}$ is the only term in $\widehat A_{\widehat f}^c(\varepsilon/2)$ containing the double logarithm. The formal invariant $\rho$ is explicitely lisible from its coefficient. 
\smallskip

There is yet another, more illustrative way to see that in the case $m=0$ the formal invariant $\rho$ is the coefficient in front of $\varepsilon\boldsymbol\ell_2^{-1}$. We use Proposition~\ref{lem:jen} for $m=0$ and Lemma~\ref{lem:tri}. Putting \eqref{eq:sc} in \eqref{eq:jaoj}, the result follows immediately.
\hfill $\Box$

\section{Appendix}
\label{sec:appendix}
\noindent The following proposition is necessary for Definition~\ref{def:is} of \emph{integral sections}.

\begin{prop}\label{integralsects} Let $\mathbf s$ be any section such that $\mathbf s\Big|_{\widehat {\mathcal S}_0}$ is coherent. Then $\widehat {\mathcal L}_1^I\subset \widehat{\mathcal S}_0 \cup \widehat{\mathcal S}_1^{\mathbf s}$.
\end{prop}

\begin{proof}
Take $\widehat F\in\widehat{\mathcal L}_1^I$ and let $F\in\mathcal G$ be its one integral sum, as in \eqref{intsum}. It is sufficient to verify that the algorithm of Poincar\' e applied to $F$ with respect to any section $\mathbf s$ coherent on $\widehat{\mathcal S}_0$ gives the asymptotic expansion $\widehat F$. 

\noindent Due to coherence of $\mathbf s$ on $\widehat{\mathcal S}_0$, it is sufficient to prove the following:

1. That the terms of $\widehat F\in\widehat{\mathcal L}_1^I$ from \eqref{dvaa} can be grouped as:
\begin{equation}\label{eq:forma}
\widehat F(y)=\sum_{i=1}^{\infty} \widehat g_i\big(\boldsymbol\ell(y)\big)  y^{\alpha_i},
\end{equation}
where $(\alpha_i)$ is a strictly increasing sequence of real numbers tending to $+\infty$ or finite and $\widehat g_i$ are \emph{convergent}.

2.  At the same time, that the integral sum $F$ of $\widehat F$ given in \eqref{intsum} satisfies:
\begin{equation}\label{eq:germi}
F(y)-\sum_{i=1}^{n} g_i\big(\boldsymbol\ell(y)\big)y^{\alpha_i}=o(y^{\alpha_n}),\ y\to 0,\ n\in\mathbb N,
\end{equation}
where $g_i$ are the sums of the \emph{convergent} transseries $\widehat g_i$ and $\alpha_i$ are the same as in 1. 

First, by \emph{Fubini's theorem} and absolute convergence of $\widehat G_i$, $i=0,1,$ and $\widehat h$ in \eqref{dvaa}, we have that:
\begin{align}
&\widehat G_i(y)=\sum_{j=1}^{\infty}\widehat g_j^i\big(\boldsymbol\ell(y)\big) y^{\beta_j^i},\ i=0,1,\nonumber\\
&\widehat h(y)=a y^{\gamma_1}+\sum_{j=2}^{\infty} \widehat h_j\big(\boldsymbol\ell(y)\big) y^{\gamma_j},\ a\in\mathbb R.\label{eq:fia}
\end{align}
Here, $(\beta_j^i)_{j}$ and $(\gamma_j)_j$ are strictly increasing and tending to $+\infty$ or finite, 
and $\widehat g_j^i\in\widehat{\mathcal L}_0^\infty$, $i=0,1,$ and $\widehat h_j\in\widehat{\mathcal L}_0^\infty$, $j\in\mathbb N$, are convergent with the sums $g_j^i$, $h_j$ respectively. For the sums $G_i\in\mathcal G,\ i=0,1,$ and $h\in\mathcal G$ from \eqref{intsum} it holds that:
\begin{align}
&G_i(y)-\sum_{j=1}^{n}g_j^i\big(\boldsymbol\ell(y)\big) y^{\beta_j^i}=o(x^{\beta_n^i}),\ n\in\mathbb N,\ i=0,1,\nonumber\\
&h(y)-a y^{\gamma_1}-\sum_{j=2}^{n} h_j\big(\boldsymbol\ell(y)\big) y^{\gamma_j}=o(y^{\gamma_n}),\ y\to 0,\ n\in\mathbb N.\label{eq:se}
\end{align}
\smallskip

It is easy to see, with $\widehat h$ and $h$ as above, that: 
\begin{equation}\label{eq:ka}
\boldsymbol\ell(e^{-\frac{\gamma}{y}}\widehat h(y))=\frac{y}{\gamma}+\widehat k(y),
\end{equation}
as well as that
$$
\boldsymbol\ell(e^{-\frac{\gamma}{y}} h(y))=\frac{y}{\gamma}+ k(y),
$$
where $\widehat k\in\widehat{\mathcal L}$ is a convergent transseries with the sum $k\in\mathcal G$. In particular,
\begin{align*}
\widehat k(y)=c y^{\delta_0}+\sum_{i=2}^{\infty}\widehat k_i(\boldsymbol\ell(y)) y^{\delta_i},\ 
k(y)-cy^{\delta_0}- \sum_{i=2}^{n} k_i(\boldsymbol\ell(y)) y^{\delta_i}=o(y^{\delta_n}),\ n\in\mathbb N,
\end{align*}
where $(\delta_i)_i$ are strictly increasing to $+\infty$ or finite, $\delta_0>1$, $c \in\mathbb R$, and $\widehat k_i\in\widehat{\mathcal L}_0^\infty$ are convergent power asymptotic expasions of $k_i\in\mathcal G$, $i\in\mathbb N$.
\medskip

Now suppose that in \eqref{dvaa} $\widehat f(y)=\sum_{k=N}^{\infty} a_k y^k \in\widehat{\mathcal L}_0^I$, $N\in\mathbb Z$, with integral factor $\alpha$. Then
$$\frac{d}{dy}\Big(y^\alpha \widehat f\big(\boldsymbol\ell(y)\big)\Big)=y^{\alpha-1}\widehat R\big(\boldsymbol\ell(y)\big),$$
with $\widehat R\in\widehat {\mathcal L}_0^\infty$ convergent Laurent.
 Let $f$ be its integral sum. Then, by \cite[Remark 3.13]{MRRZ2Fatou}, $f$ admits $\widehat f$ as its power asympotic expansion. Moreover, differentiating \eqref{deff} and \eqref{eq:joj} and since $\widehat R$ is a convergent Laurent series, inductively it follows that $f^{(k)}$ admits the formal derivative $\widehat f^{(k)}$, $k\in\mathbb N$, as its power asymptotic expansion. Indeed, inductively, $\widehat f^{(k)}$ is a finite combination of $\widehat f$, $\widehat R$, and the formal derivatives $\widehat R'$, $\ldots$, $\widehat R^{(k-1)}$, $k\in\mathbb N$. The same combination holds for the germ counterparts.
\medskip

Using \eqref{eq:ka}, we have the following Taylor expansions (formal and for germs):
\begin{align}
\widehat f\big(\boldsymbol\ell(e^{-\frac{\gamma}{y}}\widehat h(y))\big)=\widehat f\big(\frac{y}{\gamma}\big)+\widehat f'\big(\frac{y}{\gamma}\big)\widehat k(y)+\frac{1}{2!}\widehat f''\big(\frac{y}{\gamma}\big)\widehat k(y)^2+\ldots,\label{eq:gio1}\\
f\big(\boldsymbol\ell(e^{-\frac{\gamma}{y}} h(y))\big)=f\big(\frac{y}{\gamma}\big)+ f'\big(\frac{y}{\gamma}\big) k(y)+\frac{1}{2!} f''\big(\frac{y}{\gamma}\big) k(y)^2+\ldots\label{eq:gio2}.
\end{align}

Combining $\eqref{eq:fia}$ with \eqref{eq:gio1}, as well as on the other hand $\eqref{eq:se}$ with \eqref{eq:gio2}, we conclude \eqref{eq:forma} formally for $\widehat F\in\widehat{\mathcal L}_1^I$ and analogously \eqref{eq:germi} for its sum $F$. Here, $g_i\in\mathcal G$ are exactly the sums of convergent series $\widehat g_i$, $i\in\mathbb N$, since they are given as the same finite combinations of convergent series, respectively their sums. 
\end{proof}

\medskip

\begin{prop}[Uniqueness of the integral sum]\label{difi} Let $\widehat F\in\widehat{\mathcal L}_1^I$.
Let 
\begin{equation}\label{eq:F}
\widehat F(y)=\widehat G_1(y)\cdot \widehat f\big(\boldsymbol\ell(e^{-\frac{\gamma}{y}} \widehat h_1(y))\big)+\widehat G_0(y)
\end{equation}
be a decomposition of the form \eqref{dvaa}, not necessarily unique. Let $\alpha\in\mathbb R$ be the exponent of integration of $\widehat f$.

1. If $\alpha<0$, then the integral sum $F\in\mathcal G_{AN}$ corresponding to this decomposition is unique up to an additive term $c G_1(y)\cdot \big(e^{-\frac{\gamma}{y}} h_1(y)\big)^{-\alpha}$, $c\in\mathbb R$. Here, $h_1\in\mathcal G_{AN}$ is the sum of $\widehat h_1$.

2. If $\alpha\geq 0$, the integral sum $F\in\mathcal G_{AN}$ is unique.
\medskip

\noindent Morever, let \begin{equation}\label{eq:G}\widehat F(y)=\widehat H_1(y)\cdot \widehat g\big(\boldsymbol\ell(e^{-\frac{\delta}{y}} \widehat h_2(y))\big)+\widehat H_0(y)\end{equation}
be another decomposition \eqref{dvaa} of the same $\widehat F\in\widehat{\mathcal L}_1^I$, with the exponent of integration $\beta\in\mathbb R$ of $\widehat g$ not necessarily equal to $\alpha$. Then its sum is again equal to $F$, up to an additive term $c G_1(y)\cdot \big(e^{-\frac{\gamma}{y}} h_1(y)\big)^{-\alpha}$, $c\in\mathbb R$.
\end{prop}	

\begin{proof} If $\alpha=0$, then necessarily $\beta=0$ (the case $\widehat F$ convergent), and the sum $F\in\mathcal G_{AN}$ is unique. We therefore suppose in the proof that $\alpha,\ \beta\neq 0$.
\smallskip

The first statement of the proposition follows directly from Definition~\ref{diifi} of the integral sum and Remark~\ref{rem:helpdif}. Therefore, the integral sum corresponding to decomposition \eqref{eq:F} is unique if $\alpha>0$ and unique up to $c G_1(y)\cdot \big(e^{-\frac{\gamma}{y}} h_1(y)\big)^{-\alpha}$, $c\in\mathbb R,$ if $\alpha<0$. Analogously, the integral sum corresponding to decomposition \eqref{eq:G} is unique if $\beta>0$ and unique up to $d H_1(y)\cdot \big(e^{-\frac{\delta}{y}} h_2(y)\big)^{-\beta}$, $d\in\mathbb R$, if $\beta<0$. We show below that if \eqref{eq:F} and \eqref{eq:G} are two decompositions of the same $\widehat F$, then formally in $\widehat{\mathcal L}_2^\infty$: \begin{align}
&\frac{\widehat H_1(\boldsymbol\ell)}{\widehat G_1(\boldsymbol\ell)}\cdot \frac{(\widehat h_1(\boldsymbol\ell))^\alpha}{(\widehat h_2(\boldsymbol\ell))^\beta}=Cx^{\delta\beta-\gamma\alpha},\ \text{for some constant } C\in\mathbb R.\label{eq:firststep}
\end{align}
It follows that $\gamma\alpha=\delta\beta$. Moreover, since  $\widehat H_1,\ \widehat G_1,\ \widehat h_1,\ \widehat h_2\in\widehat{\mathcal L}^\infty$ are convergent, their sums are unique, and it follows that:
\begin{equation}\label{uh}\frac{H_1(y)}{G_1(y)}\cdot \frac{(h_1(y))^\alpha}{(h_2(y))^\beta}=C,\ \text{for some constant } C\in\mathbb R.\end{equation}
 Consequently, the additive term in the integral sum of $\widehat F\in\widehat{\mathcal L}_1^I$ is the same for all decompositions of $\widehat F$. 
\medskip

We prove now that the integral sums of both decompositions \eqref{eq:F} and \eqref{eq:G} of $\widehat F$ are equal, up to the above mentioned additive term. Substracting \eqref{eq:F} and \eqref{eq:G} and using \eqref{eq:firststep}, we get:
\begin{equation}\label{huh}
\widehat f\big(\boldsymbol\ell(x^\gamma \widehat h_1(\boldsymbol\ell))\big)-C\frac{(x^\delta \widehat h_2(\boldsymbol\ell))^\beta}{(x^\gamma \widehat h_1(\boldsymbol\ell))^\alpha}\widehat g\big(\boldsymbol\ell(x^\delta \widehat h_2(\boldsymbol\ell))\big)=\frac{\widehat H_0(\boldsymbol\ell)-\widehat G_0(\boldsymbol\ell)}{\widehat G_1(\boldsymbol\ell)}.
\end{equation}
We multiply \eqref{eq:jen} by $\big(x^\gamma \widehat h_1(\boldsymbol\ell)\big)^\alpha\in\widehat{\mathcal L}_2^\infty$ and differentiate formally in $\widehat{\mathcal L}_2^\infty$ by $\frac{\mathrm d}{\mathrm dx}$. By \eqref{deff}, $$\frac{\mathrm d}{\mathrm dx}\big(x^\alpha \widehat f(\boldsymbol\ell)\big)=x^{\alpha-1}\widehat R_1(\boldsymbol\ell),\ \frac{\mathrm d}{\mathrm dx}\big(x^\beta \widehat g(\boldsymbol\ell)\big)=x^{\beta-1}\widehat R_2(\boldsymbol\ell),$$where $\widehat R_1,\ \widehat R_2\in\widehat{\mathcal L}_0^\infty$ are convergent Laurent series. We get formally in $\widehat{\mathcal L}_2^\infty$:
\begin{align}
&\big(x^\gamma \widehat h_1(\boldsymbol\ell)\big)^{\alpha-1}\widehat R_1\big(\boldsymbol\ell(x^\gamma \widehat h_1(\boldsymbol\ell))\big)\cdot \frac{\mathrm d}{\mathrm dx}(x^\gamma \widehat h_1(\boldsymbol\ell))-\label{eq:jen}\\
&-C\big(x^\delta \widehat h_2(\boldsymbol\ell)\big)^{\beta-1}\widehat R_2\big(\boldsymbol\ell(x^\delta \widehat h_2(\boldsymbol\ell))\big)\cdot \frac{\mathrm d}{\mathrm dx}(x^\delta\widehat h_2(\boldsymbol\ell))=\frac{\mathrm d}{\mathrm dx}\Big(\frac{\widehat H_0(\boldsymbol\ell)-\widehat G_0(\boldsymbol\ell)}{\widehat G_1(\boldsymbol\ell)}\cdot \big(x^\gamma \widehat h_1(\boldsymbol\ell)\big)^\alpha\Big).\nonumber
\end{align}
Since $\widehat h_{1,2},\ \widehat G_{0,1},\ \widehat H_0$ are convergent in $\widehat{\mathcal L}^\infty$ and their derivatives commute with the sums, and since $\widehat R_{1,2}$ are also convergent, we may \emph{remove the hats} and  get the following analogue of \eqref{eq:jen} in $\mathcal G_{AN}$:
\begin{align}\label{eq:fui}
&\big(x^\gamma h_1(\boldsymbol\ell)\big)^{\alpha-1}R_1\big(\boldsymbol\ell(x^\gamma  h_1(\boldsymbol\ell))\big)\cdot \frac{\mathrm d}{\mathrm dx}(x^\gamma  h_1(\boldsymbol\ell))-\\
&-C\big(x^\delta  h_2(\boldsymbol\ell)\big)^{\beta-1}R_2\big(\boldsymbol\ell(x^\delta h_2(\boldsymbol\ell))\big)\cdot \frac{\mathrm d}{\mathrm dx}(x^\delta h_2(\boldsymbol\ell))=\frac{\mathrm d}{\mathrm dx}\Big(\frac{ H_0(\boldsymbol\ell)-G_0(\boldsymbol\ell)}{G_1(\boldsymbol\ell)}\cdot \big(x^\gamma  h_1(\boldsymbol\ell)\big)^\alpha\Big).\nonumber
\end{align}
The goal is to get the equivalent of \eqref{huh} for germs. So, once we have reached the equation \eqref{eq:fui} for germs by \emph{removing hats} due to convergence, we reverse the procedure. We integrate \eqref{eq:fui} in $\mathcal G_{AN}$ with respect to $\int_d^{x} \,\mathrm dt$ or, equivalently,\\ $\int_d^{x^\gamma h_1(\boldsymbol\ell)} \mathrm d\big(t^\gamma h_1(\boldsymbol\ell(t))\big)$, where $d\geq 0$ and $d=0$ iff $\alpha,\ \beta>0$. The notation means: $$\frac{\mathrm d\big(t^\gamma h_1(\boldsymbol\ell(t))\big)}{\mathrm dt}=t^{\gamma-1}\big(\gamma h_1(\boldsymbol\ell(t))+h_1(\boldsymbol\ell(t))\boldsymbol\ell^2(t)\big).$$ By \eqref{eq:joj}, for any two integral sums $f,\ g\in\mathcal G_{AN}$ of $\widehat f, \ \widehat g\in\widehat{\mathcal L}_0^\infty$, there exist $d_1,\ d_2\geq 0$ such that:
\begin{align}
&f\big(\boldsymbol\ell(x^\gamma h_1)\big)(x^\gamma h_1)^\alpha=\int_{d_1}^{x^\gamma h_1(\boldsymbol\ell)} R_1\big(\boldsymbol\ell(s) \big)s^{\alpha-1}\,\mathrm ds,\nonumber\\
&g(\boldsymbol\ell(x^\delta h_2))(x^\delta h_2)^\beta=\int_{d_2}^{x^\delta h_2(\boldsymbol\ell)}R_2\big(\boldsymbol\ell(s)\big)s^{\beta-1}\,\mathrm ds,\label{eq:al}
\end{align}
where $d_1=0$ and $d_2=0$ if and only if $\alpha,\ \beta>0$, otherwise $d_1,\ d_2>0$. The integral sums $f$ and $g$ are obviously uniquely defined by \eqref{eq:al} up to an additive constant of integration. Since by \eqref{uh} $CG_1(x^\delta h_2)^\beta=H_1(x^\gamma h_1)^{\alpha}$, using \eqref{eq:al} in \eqref{eq:fui} after integration, we get:
\begin{equation}\label{fui}
f\big(\boldsymbol\ell(x^\gamma h_1)\big)-C\frac{(x^\delta h_2)^\beta}{(x^\gamma h_1)^\alpha}g\big(\boldsymbol\ell(x^\delta h_	2)\big)=\frac{H_0(\boldsymbol\ell)-G_0(\boldsymbol\ell)}{G_1(\boldsymbol\ell)}+ D \big(x^\gamma h_1(\boldsymbol\ell)\big)^{-\alpha},\ D\in\mathbb R.
\end{equation}
Comparing \eqref{eq:F}, \eqref{eq:G} with \eqref{fui}, we conclude that the integral sum of $\widehat F$ is unique up to $D G_1(\boldsymbol\ell) (x^\gamma h_1(\boldsymbol\ell))^{-\alpha}$, $D\in\mathbb R$. Note that this term is not dependent on decomposition, and that $G_1$, $\alpha,\ \gamma,\ h_1$ are elements of an \emph{arbitrarily chosen} decomposition.
\medskip

It remains only to prove \eqref{eq:firststep}. We have (formally in $\widehat {\mathcal L}^\infty$):
\begin{equation}\label{eq:eq}
\widehat F(\boldsymbol\ell)=\widehat G_1(\boldsymbol\ell)\cdot \widehat f\big(\boldsymbol\ell(x^\gamma \widehat h_1(\boldsymbol\ell)))\big)+\widehat G_0(\boldsymbol\ell)=\widehat H_1(\boldsymbol\ell)\cdot \widehat g\big(\boldsymbol\ell(x^\delta \widehat h_2(\boldsymbol\ell)))\big)+\widehat H_0(\boldsymbol\ell).
\end{equation}
It is not possible that $\widehat G_1\equiv 0$ and $\widehat H_1\equiv\!\!\!\!\!\!/ \ 0$. Indeed, in that case, we would have that $\boldsymbol\ell\mapsto\widehat g\big(\boldsymbol\ell(x^\delta \widehat h_2(\boldsymbol\ell))\big)$ is convergent in $\widehat{\mathcal L}^\infty$ as a quotient of convergent transseries:
\begin{equation}\label{divergence}
\widehat g\big(\boldsymbol\ell(x^\delta h_2(\boldsymbol\ell))\big)=\frac{G_0(\boldsymbol\ell)-H_0(\boldsymbol\ell)}{H_1(\boldsymbol\ell)}.
\end{equation}
We denote convergent transseries in $\widehat{\mathcal L}^\infty$ without hats. Note that $\widehat g$ is divergent close to $0$. Take $\boldsymbol\ell$ sufficiently small so that the convergent transseries on the right-hand side, as well as $\widehat h_2(\boldsymbol\ell)$, evaluated at $\boldsymbol\ell$, converge.  On the other hand, since $\boldsymbol\ell(x^\delta h_2(\boldsymbol\ell))=\boldsymbol\ell(1+o(1)),\ \boldsymbol\ell\to 0$
, by taking $\boldsymbol\ell$ sufficiently small, we may ensure that $\widehat g$ evaluated at $\boldsymbol\ell(x^\delta h_2(\boldsymbol\ell))$ diverges. This is a contradiction with the equality \eqref{divergence}. Therefore, either both $\widehat G_1$ and $\widehat H_1$ are zero or both are different from zero. In the first case it trivially follows that $G_0\equiv H_0$ (both are convergent) and the decompositions \eqref{eq:F} and \eqref{eq:G} are exactly the same. 

Suppose now without loss of generality that $\widehat G_1\equiv\!\!\!\!\!\!/ \ 0$. 
Dividing both sides of the equality \eqref{eq:eq} by $G_1$, we get formally in $\widehat{\mathcal L}_2^\infty$:
$$
\widehat f(\boldsymbol\ell(x^\gamma h_1(\boldsymbol\ell)))-\frac{H_1(\boldsymbol\ell)}{ G_1(\boldsymbol\ell)}\widehat g(\boldsymbol\ell(x^\delta h_2(\boldsymbol\ell)))=\frac{H_0(\boldsymbol\ell)-G_0(\boldsymbol\ell)}{G_1(\boldsymbol\ell)}.
$$
Multiplying by $(x^\gamma h_1(\boldsymbol\ell))^\alpha\in\widehat{\mathcal L}_2^\infty$ and differentiating formally by $\frac{\mathrm d}{\mathrm dx}$, we get (in $\widehat{\mathcal L}_2^\infty$):
\begin{align*}
\frac{\mathrm d}{\mathrm dx}\Big((x^\gamma  h_1(\boldsymbol\ell))^\alpha &\widehat f\big(\boldsymbol\ell(x^\gamma h_1(\boldsymbol\ell))\big)\Big)-\frac{\mathrm d}{\mathrm dx}\Big(\widehat g\big(\boldsymbol\ell(x^\delta  h_2(\boldsymbol\ell))\big) (x^\delta h_2(\boldsymbol\ell))^\beta \cdot \frac{H_1(\boldsymbol\ell)}{ G_1(\boldsymbol\ell)}\frac{(x^\gamma h_1(\boldsymbol\ell))^\alpha}{(x^\delta  h_2(\boldsymbol\ell))^\beta}\Big)\\
&=\frac{\mathrm d}{\mathrm dx}\Big(\frac{H_0(\boldsymbol\ell)-G_0(\boldsymbol\ell)}{G_1(\boldsymbol\ell)}(x^\gamma h_1(\boldsymbol\ell))^\alpha\Big).
\end{align*}
Differentiating, dividing by $x^{\alpha\gamma-1}$ and grouping the convergent transseries we get that
$$
\boldsymbol\ell\mapsto\widehat g\big(\boldsymbol\ell(x^\delta h_2(\boldsymbol\ell))\big)h_2(\boldsymbol\ell)^\beta\cdot\frac{\frac{\mathrm d}{\mathrm dx}\Big(\frac{H_1(\boldsymbol\ell)}{G_1(\boldsymbol\ell)}\frac{(x^\gamma  h_1(\boldsymbol\ell))^\alpha}{(x^\delta  h_2(\boldsymbol\ell))^\beta}\Big)}{x^{\gamma\alpha-1-\delta\beta}}
$$
is a convergent transseries in $\widehat{\mathcal L}^\infty$. If the derivative in the parenthesis is different from $0$, it  follows that
$
\boldsymbol\ell\mapsto\widehat g(\boldsymbol\ell(x^\delta h_2(\boldsymbol\ell)))
$ is a convergent transseries in $\widehat{\mathcal L}^\infty$. This is a contradiction, as already explained in detail above. Therefore, it necessarily holds that the derivative is $0$, that is:
$$
\frac{H_1(\boldsymbol\ell)}{G_1(\boldsymbol\ell)}\frac{(x^\gamma  h_1(\boldsymbol\ell))^\alpha}{(x^\delta  h_2(\boldsymbol\ell))^\beta}=C,\ C\in\mathbb R.
$$
Now \eqref{eq:firststep} directly follows.
\end{proof}
\medskip

\noindent \emph{Proof of Proposition~\ref{lem:haa}.}\

The blocks of the formal inverse $\widehat g^{-1}$ are by \eqref{Dulaccoef} \emph{convergent transseries}. Since every integral section is coherent (respects convergence), it is sufficient to prove that $g^{-1}$ can be expanded in increasing powers in $x$ in the form \eqref{eq:gmoins}, where $f_{\beta_i}$ are the \emph{sums} of convergent $\widehat f_{\beta_i}$. Coarsely, this is proven by repeating the same steps of construction as described in the proof of Proposition~\ref{prop:refin}, but this time on germs in $\mathcal G_{AN}$. 

Let $g_\alpha(x)=x^\alpha P_m(\boldsymbol\ell^{-1})$ be the first block in the Dulac expansion. Then
$$
g=g_\alpha\circ \varphi,\ g^{-1}=\varphi^{-1}\circ g_\alpha^{-1}.
$$
Computing as in the formal case, we get:
\begin{align}
&g_{\alpha}^{-1}(x)=(a\alpha^{-m})^{-\frac{1}{\alpha}}\cdot x^{\frac{1}{\alpha}}\,\boldsymbol{\ell}^{\frac{m}{\alpha}}\Big(1+F\big(\boldsymbol{\ell}_2,\frac{\boldsymbol{\ell}}{\boldsymbol{\ell}_{2}}\big)\Big),\label{geealpha}\\
&\varphi(x)=g_\alpha^{-1}\big(g(x)\big)=x\cdot \big(1+F_2(\boldsymbol\ell_2,\frac{\boldsymbol\ell}{\boldsymbol\ell_2},T(x))\big).\label{eq:fi}
\end{align}
Here, $F,\ F_2$ are analytic germs in two variables, with Taylor expansions $\widehat F$ and $\widehat F_2$ from $\widehat g_\alpha^{-1}$ resp. $\widehat\varphi$. The germ $T\in\mathcal G_{AN}$ is defined by $g(x)=ax^\alpha\boldsymbol\ell^{-m}(1+T(x))$. Since $g$ is a Dulac germ with Dulac expansion $\widehat g$, it follows that:
\begin{equation}\label{eq:te}
T(x)-\boldsymbol\ell^m P_0(\boldsymbol\ell^{-1})-\sum_{i=1}^{n} x^{\alpha_i-\alpha}\boldsymbol\ell^m P_i(\boldsymbol\ell^{-1})=o(x^{\alpha_n-\alpha}),\ \forall n\in\mathbb N,
\end{equation}
with $P_i$ as in \eqref{oh}.

Write $\widehat\varphi(x)=\sum_{i=1}^{\infty} x^{\beta_i}\boldsymbol\ell^{m_{\beta_i}}\widehat G_{\beta_i}(\boldsymbol\ell_2,\frac{\boldsymbol\ell}{\boldsymbol\ell_2})$, $\beta_i>0$ and strictly increasing, $m_{\beta_i}\in\mathbb Z$, $i\in\mathbb N$. Putting \eqref{eq:te} in \eqref{eq:fi}, and expanding $F_2$, we get immediately that:
\begin{align}
\varphi(x)=&\sum_{i=1}^{n} x^{\beta_i}\boldsymbol\ell^{m_{\beta_i}} G_{\beta_i}(\boldsymbol\ell_2,\frac{\boldsymbol\ell}{\boldsymbol\ell_2})+o(x^{\beta_n}),\ n\in\mathbb N,\label{fgcorresp}
\end{align}
where $G_{\beta_i}$ are analytic germs of two variables with Taylor expansion $\widehat G_{\beta_i}$, $i\in\mathbb N$.

In particular, as was the case for $\widehat\varphi$, the leading term of $\varphi(x)-x$ is of power strictly bigger than $1$ in $x$ ($\varphi$ is strictly parabolic).

We now analyze the \emph{blocks} in the asymptotic expansion of $\varphi^{-1}$ by increasing powers in $x$, using the Neumann inverse series. We prove that they are the sums of the corresponding  convergent blocks (see \eqref{eq:koef}) of $\widehat\varphi^{-1}$. By coherence of integral sections, this implies that $\widehat\varphi^{-1}$ is the sectional asymptotic expansion of $\varphi^{-1}$ with respect to any integral section. 

Recall the \emph{Schröder operator} $\widehat \Phi_{\varphi}$ from Lemma~\ref{three},
used for obtaining the formal inverse $\widehat\varphi^{-1}$ of  $\widehat{\varphi}$.
We define similarly here the linear operator $\Phi_{\varphi}$ acting
on $\mathcal G_{AN}$, $\Phi_{\varphi}\in L(\mathcal G_{AN})$, by: 
\[
\Phi_{\varphi}\cdot f=f\circ\varphi,\ f \in\mathcal G_{AN}.
\]
Denote here $h=\varphi-\mathrm{id}\in\mathcal G_{AN}$.
Furthermore, let us introduce the linear operator $H_{\varphi}:=\Phi_{\varphi}-\mathrm{Id}\in L(\mathcal G_{AN})$,
\[
H_{\varphi}\cdot f=f\circ\varphi-f,\ f\in \mathcal G_{AN}.
\]
Let us consider the \emph{Neumann series}: 
\begin{equation*}
\sum_{k=0}^{\infty}(-1)^{k}H_{\varphi}^{k}\cdot id.
\end{equation*}
Denote its partial sums by 
\[
S_{n}:=\sum_{k=0}^{n}(-1)^{k}H_{\varphi}^{k}\cdot id\in\mathcal G_{AN},\ n\in\mathbb{N}.\]
We prove that the Neumann partial sums $S_n$ \emph{approximate} $\varphi^{-1}$, as $n\to\infty$.
More precisely, we prove that, for every $\gamma>0$, there exists
$n_{\gamma}\in\mathbb{N}$ such that 
\begin{equation*}
S_{n_{\gamma}}(x)=\varphi^{-1}(x)+O(x^{\gamma}),\ x\to 0.
\end{equation*}
In other words, we prove that: 
\[
S_{n_{\gamma}}\big(\varphi(x)\big)=x+O(x^{\gamma}),\ x\to 0.
\]
Indeed, 
\begin{align}
S_{n_{\gamma}}\big(\varphi(x)\big)&=\big(\Phi_{\varphi}\cdot S_{n_{\gamma}}\big)(x)= \big(\mathrm{Id}+H_{\varphi}\big)\cdot \big(\sum_{k=0}^{n_{\gamma}}(-1)^{k}H_{\varphi}^{k}\cdot\mathrm{id}\big)=\label{eq:first}\\
= & \sum_{k=0}^{n_{\gamma}}(-1)^{k}H_{\varphi}^{k}\cdot\mathrm{id}+\sum_{k=0}^{n_{\gamma}}(-1)^{k}H_{\varphi}^{k+1}\cdot\mathrm{id}=x+(-1)^{n_{\gamma}}H_{\varphi}^{n_{\gamma}+1}\cdot\mathrm{id}.\nonumber 
\end{align}
Since $\varphi$ is \emph{strictly} parabolic, there exists some $\delta>0$ such that $H_{\varphi}\cdot\mathrm{id}=o(x^{1+\delta})$.
Inductively, there exists $n_{\gamma}\in\mathbb{N}$ such that $H_{\varphi}^{n_{\gamma}+1}\cdot\mathrm{id}=O(x^{\gamma})$.
Now \eqref{eq:first} transforms to: 
\[
S_{n_{\gamma}}\big(\varphi(x)\big)=x+O(x^{\gamma}),
\]
that is 
\begin{equation}
S_{n_{\gamma}}(x)=\varphi^{-1}(x)+O\big((\varphi^{-1}(x))^{\gamma}\big)=\varphi^{-1}(x)+O(x^{\gamma}).\label{eq:conclu}
\end{equation}
By \eqref{eq:conclu}, we have, for $\gamma\to\infty$, the following expansion of $\varphi^{-1}$ in strictly increasing powers of $x$:
\begin{align}
\varphi^{-1}=&\,\mathrm{id}- h+\Big( h\circ \varphi - h\Big)+\Big(\big( h\circ \varphi -h\big)\circ\varphi-\big( h\circ \varphi -h\big)\Big)+\ldots+O(x^\gamma),\label{eq:final}
\end{align} 
as compared with its formal analogue \eqref{eq:how}.
In \eqref{eq:final}, for a fixed $\gamma>0$, the number of summands up to the order $O(x^\gamma)$ is finite and equal to $n_\gamma$. We now expand the compositions in summands of $\varphi^{-1}$ in increasing powers of $x$, using expansion for $\varphi$ given in \eqref{fgcorresp} and the fact that $\varphi$ is \emph{strictly} parabolic. Since $\varphi$ is strictly parabolic, the order of $x$ in the consecutive brackets of $\varphi^{-1}$ is strictly increasing. Thus only finitely many terms contribute to a block with a fixed power of $x$. The blocks in $x$ of $\varphi^{-1}$ are the sums of the corresponding convergent blocks of $\widehat\varphi^{-1}$. 

Finally, we analyze the blocks with increasing powers of $x$ in the composition $g^{-1}=\varphi^{-1}\circ g_\alpha^{-1}$. Using the expansion \emph{by blocks} of $\varphi^{-1}$ and \eqref{geealpha}, we show that they are the sums of the corresponding convergent blocks of  $\widehat g^{-1}=\widehat \varphi^{-1}\circ \widehat g_\alpha^{-1}$.
\hfill $\Box$
\medskip

\begin{prop}\label{prop:rho}
Let $\widehat f\in\widehat{\mathcal L}$. Then the formal invariant $\rho$ of $\widehat f$ is given by 
$$
\rho=\Big[\frac{\boldsymbol\ell}{x}\Big]\frac{1}{\widehat g(x)}=-[\boldsymbol\ell_2^{-1}]\int^x\frac{ds}{\widehat g(s)}.
$$
\end{prop}


\begin{proof} The second equality is obvious. We prove the first equality in two steps:

1. For the formal normal form $\widehat f_0(x)=x-x^{\alpha}\boldsymbol\ell^m+\rho x^{2\alpha-1}\boldsymbol\ell^{2m+1}$, it obviously holds that:
 $$\Big[\frac{\boldsymbol\ell}{x}\Big]\frac{1}{\widehat g_0(x)}=\rho.$$ This can easily be seen expanding $\frac{1}{\widehat g_0(x)}.$

2. Let $\widehat f_1=\widehat \varphi\circ \widehat f\circ \widehat\varphi^{-1}$, $\widehat\varphi(x)=x+cx^\beta\boldsymbol\ell^r\in\widehat{\mathcal L},\ c\in\mathbb R,\ \beta\in\mathbb R,\ r\in\mathbb Z,$ such that $(\beta,r)\succ (1,0)$. We prove that:$$\Big[\frac{\boldsymbol\ell}{x}\Big]\frac{1}{\widehat g(x)}=\Big[\frac{\boldsymbol\ell}{x}\Big]\frac{1}{\widehat g_1(x)}.$$ That is, we prove that the coefficient $\big[\frac{\boldsymbol\ell}{x}\big]\frac{1}{\widehat g(x)}$ does not change by formal changes of variables. 
 
Indeed, we estimate the difference:
\begin{align}
&\Big[\frac{\boldsymbol\ell}{x}\Big]\frac{1}{\widehat g(x)}-\Big[\frac{\boldsymbol\ell}{x}\Big]\frac{1}{\widehat g_1(x)}=
[\boldsymbol\ell_2^{-1}]\int^x \frac{ds}{\widehat g_1(s)}-\Big[\boldsymbol\ell_2^{-1}]\int^x\frac{ds}{\widehat g(s)}=\nonumber\\
&=[\boldsymbol\ell_2^{-1}] \Big(\int^x\frac{ds}{\widehat g_1(s)}-\int^x\frac{ds}{\widehat g(s)}\Big).\label{eq:prvvi}
\end{align}
Compute:
\begin{align}
\int^x \frac{ds}{\widehat g_1(s)}&=\int^x \frac{ds}{s-\widehat f_1(s)}=\Big|s=\widehat\varphi(t)\Big|=\int^{\widehat\varphi^{-1}(x)} \frac{\widehat\varphi'(t)}{\widehat\varphi(t)-\widehat{\varphi}(\widehat f(t))}\, dt=\nonumber\\
&=\int^{\widehat\varphi^{-1} (x)}\frac{dt}{\widehat g(t)}\cdot \frac{\widehat\varphi'(t)\widehat g(t)}{\widehat\varphi(t)-\widehat\varphi(\widehat f(t))}=\int^{\widehat\varphi^{-1}(x)}\frac{dt}{\widehat g(t)}\cdot\Big(1+\frac{1}{2}\frac{\widehat\varphi''(t)\widehat g(t)}{\widehat\varphi'(t)}+\ldots \Big)=\nonumber\\
&=\int^{\widehat\varphi^{-1}(x)}\frac{dt}{\widehat g(t)}+\frac{1}{2}\int^{\widehat\varphi^{-1}(x)}\frac{d}{dt}(\log\widehat\varphi'(t))dt+\ldots=\int^{\widehat\varphi^{-1}(x)}\frac{dt}{\widehat g(t)}+o(1).\label{eq:drrugi}
\end{align}
Here, $o(1)$ denotes infinitesimal terms in the formal series. Furthermore, since $\widehat\varphi\in \widehat{\mathcal L}$ and parabolic, it can be seen that
\begin{equation}\label{eq:trreci}
\int^{\widehat\varphi^{-1}(x)}\frac{dt}{\widehat g(t)}-\int^x \frac{dt}{\widehat g(t)}\in\widehat{\mathcal L}_1,
\end{equation}
that is, the difference does not contain the double logarithm. Indeed, if we denote by $\widehat P$ the formal antiderivative of $\frac{1}{\widehat g}$, we get that the difference is equal to $\widehat P'(x)\widehat h(x)+\frac{1}{2}\widehat P''(x)(\widehat h(x))^2+\ldots$, where $\widehat h=\mathrm{id}-\widehat \varphi^{-1}$. Obviously, $\widehat P^{(k)}\in\widehat{\mathcal L}_1$, for all $k\in\mathbb N$.

Using \eqref{eq:drrugi} and \eqref{eq:trreci}, we conclude in \eqref{eq:prvvi} that
$$
\Big[\frac{\boldsymbol\ell}{x}\Big]\frac{1}{\widehat g(x)}-\Big[\frac{\boldsymbol\ell}{x}\Big]\frac{1}{\widehat g_1(x)}=0.
$$
\end{proof}

\noindent \emph{Proof of Proposition~\ref{lem:jen}.}\

By Proposition~\ref{prop:rho} in the Appendix, by the change of variables in the integral and by integration by parts, we get:
\begin{align}\label{eq:ma}
\rho=\Big[\frac{\boldsymbol\ell}{x}\Big]\frac{1}{\widehat g(x)}=-[\boldsymbol\ell_2^{-1}]\int^x \frac{ds}{\widehat g(s)}=\Big|s=\widehat g^{-1}(2t)\Big|=&-[\boldsymbol\ell_2^{-1}]\int^{\widehat g(x)/2} \frac{(\widehat g^{-1})'(2t)}{t}dt=\nonumber\\
=&-[\boldsymbol\ell_2^{-1}]\int^{\widehat g(x)/2} \frac{\widehat g^{-1}(2t)}{2t^2}dt.
\end{align}
We show now, in the similar way as in the proof of Proposition~\ref{prop:rho} in the Appendix, that
$$
[\boldsymbol\ell_2^{-1}]\int^{\widehat g(x)/2} \frac{\widehat g^{-1}(2t)}{2t^2}dt=[\boldsymbol\ell_2^{-1}]\int^{\frac{1}{2}x^\alpha\boldsymbol\ell^m} \frac{\widehat g^{-1}(2t)}{2t^2}dt.
$$
That is, we put $\widehat P$ to be the formal antiderivative of $\frac{\widehat g^{-1}(2x)}{2x^2}$ and prove that the difference $$\widehat P\big(\frac {\widehat g(x)} 2\big)-\widehat P\big(\frac {1} {2} x^\alpha\boldsymbol\ell^m\big)=\widehat P'\big(\frac {\widehat g(x)} 2\big)\big(\frac {\widehat g(x)- x^\alpha\boldsymbol\ell^m} 2\big)+\frac{1}{2}\widehat P''\big(\frac {\widehat g(x)} 2\big)\big(\frac {\widehat g(x)- x^\alpha\boldsymbol\ell^m} 2 \big)^2+\ldots$$ belongs to $\widehat{\mathcal L}_1$. We prove that it converges formally and does not contain a double logarithm.

Therefore, by \eqref{eq:ma}:
\begin{align}\label{eq:roo}
\rho=-[\boldsymbol\ell_2^{-1}]\int^{\frac{1}{2}x^\alpha\boldsymbol\ell^m} \frac{\widehat g^{-1}(2t)}{2t^2}dt&=\big[\frac{\boldsymbol\ell}{x}\big]\frac{\widehat g^{-1}(x^\alpha\boldsymbol\ell^m)}{x^\alpha\boldsymbol\ell^m}\frac{d}{dx}\log\big(x^\alpha\boldsymbol\ell^m\big)=\nonumber\\
&=\big[\frac{\boldsymbol\ell}{x}\big]\frac{\widehat g^{-1}(x^\alpha\boldsymbol\ell^m)}{x^\alpha\boldsymbol\ell^m}\big(\frac{\alpha}{x}+m\frac{\boldsymbol\ell}{x}\big)=\nonumber\\
&=\alpha[\boldsymbol\ell]\frac{\widehat g^{-1}(x^\alpha\boldsymbol\ell^m)}{x^\alpha\boldsymbol\ell^m}+m[1]\frac{\widehat g^{-1}(x^\alpha\boldsymbol\ell^m)}{x^\alpha\boldsymbol\ell^m}=\nonumber\\
&=\alpha[\boldsymbol\ell]\frac{\widehat g^{-1}(t)}{t}\big|_{t=x^\alpha\boldsymbol\ell^m}+m[1]\frac{\widehat g^{-1}(t)}{t}\big|_{t=x^\alpha\boldsymbol\ell^m}.
\end{align}
\hfill $\Box$

\begin{prop}\label{prop:izvod} Let $f$ be a (prenormalized) Dulac germ. Then there exists $\delta>0$ such that:
\begin{equation}\label{eq:aepsimp}
\widehat{A}^c_{\widehat f}(\frac{\varepsilon}{2})=-\varepsilon\int^\varepsilon \frac{\widehat g^{-1}(t)}{t^2}\,dt-\varepsilon\boldsymbol\ell(\varepsilon)^{-1}+o(\varepsilon^{1+\delta}),\ \varepsilon\to 0.
\end{equation}
\end{prop}

\begin{proof}
Recall that
\begin{equation}\label{eq:accf}
\widehat A^c_{\widehat f}(\frac\varepsilon 2)=\widehat g^{-1}(\varepsilon)+\varepsilon+\varepsilon\widehat\Psi\big(\widehat g^{-1}(\varepsilon)\big).
\end{equation}
We have:
\begin{equation}\label{Psii}
\widehat\Psi(\widehat g^{-1}(\varepsilon))=\int^{\widehat g^{-1}(\varepsilon)}\frac{ds}{\widehat{\xi} (s)}=\int^\varepsilon \frac{(\widehat g^{-1})'(t)dt}{\widehat \xi(\widehat g^{-1}(t)).}
\end{equation}
Here, $X=\widehat{\xi}\frac{d}{dx}$ is the formal vector field such that $\widehat f$ is its time-one map, that is, $\widehat f=\text{Exp}\Big(\widehat\xi \frac{d}{dx}\Big).\mathrm{id}.$ Then $\widehat\Psi'=\frac{1}{\widehat\xi}$.
Let $P_{-m}$ be the polynomial of degree $-m$ such that $\widehat g(x)=x^\alpha P_{-m}(\boldsymbol\ell^{-1})+o(x^{\alpha+\delta})$, for some $\delta>0$. Since $-\widehat g=\widehat\xi+\widehat\xi'\widehat\xi+...$, we have that $$\widehat \xi(x)=-x^\alpha P_{-m}(\boldsymbol\ell^{-1})+o(x^{\alpha+\delta}),\text{ for some $\delta>0,$}$$ and 
$$\widehat\xi'(x)\widehat\xi(x)=\widehat g'(x)\widehat g(x)+o(x^{2\alpha-1+\delta}),\ \text{ for some }\delta>0.$$ 
Therefore,
\begin{equation}\label{eq:gore}
-\widehat g(x)=\widehat\xi(x)+\widehat g\cdot \widehat g'(x)+o(x^{2\alpha-1+\delta}), \text{ for some $\delta>0.$}
\end{equation}
Putting $\widehat g^{-1}(t)=\alpha^{-\frac{m}{\alpha}}t^{\frac{1}{\alpha}}\boldsymbol\ell^{-\frac{m}{\alpha}}+h.o.t.$ obtained in Section~\ref{sec:inverse} in \eqref{eq:gore}, we have:
\begin{align*}
&\widehat\xi(\widehat g^{-1}(t))=-\widehat g(\widehat g^{-1}(t))-\widehat g(\widehat g^{-1}(t))\cdot \widehat g'(\widehat g^{-1}(t))+o(t^{2-\frac{1}{\alpha}+\delta}),\\
&\widehat \xi(\widehat g^{-1}(t))=-t-t\cdot \widehat g'(\widehat g^{-1}(t))+o(t^{2-\frac{1}{\alpha}+\delta}),
\end{align*}
for some $\delta>0$. It follows that
\begin{align*}
\frac{(\widehat g^{-1})'(t)}{\widehat \xi(\widehat g^{-1}(t))}&=\frac{(\widehat g^{-1})'(t)}{-t\Big(1+\widehat g'\big(\widehat g^{-1}(t)\big)+o(t^{1-\frac{1}{\alpha}+\delta})\Big)}=\\
&=-\frac{(\widehat g^{-1})'(t)}{t}\Big(1-\widehat g'\big(\widehat g^{-1}(t)\big)+o(t^{1-\frac{1}{\alpha}+\delta})\Big)=\\
&=-\frac{(\widehat g^{-1})'(t)}{t}+\frac{1}{t}+o(t^{-1+\delta}).
\end{align*}
Now \eqref{Psii} becomes, using integration by parts:
\begin{align}\label{eq:zaa}
\widehat\Psi(\widehat g^{-1}(\varepsilon))=&-\int^{\varepsilon}\frac{(\widehat g^{-1})'(t)}{t}dt+\int^\varepsilon\frac{dt}{t}+\int^{\varepsilon}o(t^{-1+\delta})\,dt\nonumber\\
&=-\frac{\widehat g^{-1}(\varepsilon)}{\varepsilon}-\int^{\varepsilon} \frac{\widehat g^{-1}(t)}{t^2}\,dt-\boldsymbol\ell(\varepsilon)^{-1}+o(\varepsilon^\delta),
\end{align}
for some $\delta>0$. Putting \eqref{eq:zaa} in \eqref{eq:accf}, the statement follows.
\end{proof}

\emph{Address:}$\quad$$^{1}$ and $^{3}$ : Universit\'e de Bourgogne,
D\'epartment de Math\'ematiques, Institut de Math\'ematiques de
Bourgogne, B.P. 47 870-21078-Dijon Cedex, France 

$^{2}$ : University of Zagreb, Faculty of Science, Department of Mathematics, Bijeni\v cka 30, 10000 Zagreb, Croatia

$^{4}$ : University of Zagreb, Faculty of Electrical Engineering and Computing, Department of Applied
Mathematics, Unska
3, 10000 Zagreb, Croatia

\begin{thebibliography}{10}

	
\bibitem{cherkas} L. A. {\v C}erkas, \emph{Structure of the sequence function in the neighborhood of a
separatrix cycle during perturbation of an analytic autonomous
system on the plane}, Differentsial'nye Uravneniya, 1981, no 3, 469--478
	
\bibitem{dries} L.~van den Dries, A.~Macintyre, D.~Marker, \emph{Logarithmic-exponential
series}. Proceedings of the International Conference ``Analyse \&
Logique'' (Mons, 1997). Ann. Pure Appl. Logic 111 (2001), no. 1-2,
61-113. 



\bibitem{Dulac} H. Dulac, \emph{Sur les cycles limites}, Bull. Soc.
Math. France 51 (1923), 45-188.



\bibitem{elezovic_zupanovic_zubrinic} N.~Elezovi{ć}, V.~{Ž}upanovi{ć},
and D.~{Ž}ubrini{ć}. \newblock Box dimension of trajectories
of some discrete dynamical systems. \newblock {\em Chaos Solitons
Fractals}, 34(2):244--252, 2007.


\bibitem{falconer} K. Falconer, {\em Fractal geometry. Mathematical Foundations and Applications.} John Wiley \& Sons (2003)


\bibitem{ilya} Y.~Ilyashenko, S.~Yakovenko, {\em Lectures on
Analytic Differential Equations, Graduate Studies in Mathematics},
86. American Mathematical Society, Providence, RI, xiv+625 pp (2008)

\bibitem{ilyalim} Y.~Ilyashenko, {\em Finiteness theorems for
limit cycles}, Russ. Math. Surv. 45 (2), 143--200 (1990)  Finiteness
theorems for limit cycles, Transl. Amer. Math. Soc. 94 (1991).



\bibitem{mrz} P.~Mardešić, M.~Resman, V.~Županović, \textit{Multiplicity
of fixed points and $\varepsilon$-neighborhoods of orbits}, J. Differ.
Equ. \textbf{253} (2012), 2493--2514

\bibitem{mrrz2} P.~Mardešić, M.~Resman, J.~P.~Rolin,\ V.~Županović,
\textit{Normal forms and embeddings for power-log transseries}, Adv. Math. \textbf{303} (2016), 888--953

\bibitem{MRRZ2Fatou} P.~Mardešić, M.~Resman, J.~P.~Rolin,\ V.~Županović,
\textit{The Fatou coordinate for parabolic Dulac germs}, a preprint, arXiv:1710.01268 (2017)


\bibitem{mourtada} A.~Mourtada, \emph{Bifurcation de cycles limites au voisinage de polycycles hyperboliques et g\'en\'eriques \`a trois sommets}, Ann. Fac. Sci. Toulouse Math., \textbf{6} no 3 (1994), 259--292 

\bibitem{resman} M.~Resman, \textit{$\varepsilon$-neighborhoods
of orbits and formal classification of parabolic diffeomorphisms},
Discrete Contin.\ Dyn.\ Syst.\ \textbf{33}, 8 (2013), 3767--3790

\bibitem{nonlin} M. Resman, \emph{$\varepsilon$-neighborhoods of
orbits of parabolic diffeomorphisms and cohomological equations}.
Nonlinearity 27 (2014), 3005--3029


\bibitem{roussarie_number} R.~Roussarie, \emph{On the number of limit cycles which appear by perturbation of separatrix loop of planar vector
fields}, Bol. Soc. Brasil Math., \textbf{17}, no 2 (1986), 67--101  



\bibitem{roussarie} R.~Roussarie, \emph{Bifurcations of planar vector
fields and Hilbert's sixteenth problem}, Birkhäuser Verlag, Basel
(1998)





\bibitem{tricot} C. Tricot, \emph{Curves and fractal dimension},
Springer-Verlag, New York, (1995)

\bibitem{belgproc} D. {Ž}ubrinić, V. Županović, \textit{Poincar{é}
map in fractal analysis of spiral trajectories of planar vector fields},
Bull. Belg. Math. Soc. S. Stevin, \textbf{15} (2008) 947--960


\end{thebibliography}
\end{document}